\documentclass[3p,times,sort&compress]{elsarticle}
\usepackage{graphicx}
\usepackage{amsmath,amssymb,amsthm}

\usepackage{cite,hyperref}
\usepackage{multirow}
\usepackage{tikz,pgfplots}
        \pgfplotsset{compat = 1.3}
        \pgfplotsset{minor grid style={dotted}}
        \pgfplotsset{major grid style={dashed}}
        \pgfplotsset{every x tick label/.append style={font=\footnotesize, yshift=0.25ex}}
        \pgfplotsset{every y tick label/.append style={font=\footnotesize, xshift=0.25ex}}

\newtheorem{theorem}{Theorem}
\newtheorem{remark}{Remark}

\newtheorem{Lemma}{Lemma}
\newtheorem{corollary}{Corollary}
\counterwithin{table}{section}
\numberwithin{equation}{section}
\usepackage{caption} 
\usepackage{fancyhdr}
\def\tabstretch{\rule{0pt}{1.03\normalbaselineskip}}

\usepackage{algorithm}
\usepackage{algpseudocode}

\begin{document}

\begin{frontmatter}

\title{Fast solution of a phase-field model of pitting corrosion}

\tnotetext[t1]{\bf \small This is the authors' version of the work published in:  G. Frasca-Caccia, D. Conte, B. Paternoster. Fast solution of a phase-field model of pitting corrosion. {\em J. Comput. Phys.}, 553, 114717, 2026. \url{https://doi.org/10.1016/j.jcp.2026.114717}.}

\author{Gianluca Frasca-Caccia\corref{cor1}}
\cortext[cor1]{Corresponding author}
\ead{gfrascacaccia@unisa.it}
\author{Dajana Conte}
\ead{dajconte@unisa.it}
\author{Beatrice Paternoster}
\ead{beapat@unisa.it}
\address{Department of Mathematics, University of Salerno,
Via Giovanni Paolo II n. 132, 84084 Fisciano (SA), Italy}

\begin{abstract}
Excessive computational times represent a major challenge in the solution of corrosion models, limiting their practical applicability, e.g., as a support to predictive maintenance. In this paper, we propose an efficient strategy for solving a phase-field model for metal corrosion. Based on the Kronecker structure of the diffusion matrix in classical finite difference approximations on rectangular domains, time-stepping IMEX methods are efficiently solved in matrix form. However, when the domain is non-rectangular, the lack of the Kronecker structure prevents the direct use of the matrix-based approach. To address this issue, we reformulate the problem on an extended rectangular domain and introduce suitable iterative IMEX methods. The convergence of the iterations and the propagation of the numerical errors are analyzed. Test cases on two and three dimensional domains show that the proposed approach achieves accuracy comparable to existing methods, while significantly reducing the computational time, to the point of allowing actual predictions on standard workstations. 
\end{abstract}

\end{frontmatter}

\section{Introduction}

Corrosion is a pervasive issue with far-reaching economic and structural implications. According to recent studies, the cost of corrosion across multiple nations over the past fifty years has been equivalent to approximately 3\%-4\% of each country's Gross Domestic Product (GDP) \citep{KOCH}. It is estimated that the implementation of corrosion control practices could result in savings ranging from 15\% to 35\% of these costs \citep{KOCH}. With the reference cost of corrosion representing approximately 3.4\% of global GDP in 2020, this could translate to potential annual savings of between US \$433 billion and \$1,000 billion. These figures underscore the urgency of addressing corrosion, not only to prevent the loss of valuable resources but also to protect the structural integrity of critical infrastructure.

Metal surfaces are often protected by a thin, invisible oxide film that prevents corrosive damage. However, physical factors such as mechanical wear, abrasion, or collision can damage this protective film, leaving the metal vulnerable to environmental exposure. This damage can lead to both localized and non-localized corrosion, with the former being the most harmful due to its ability to cause sudden, catastrophic structural failure \citep{ansari}.  Localized corrosion, particularly in the form of pitting corrosion, is especially insidious because it can remain undetected for extended periods while progressively weakening the material. Pitting corrosion typically occurs under specific environmental conditions, such as areas with low oxygen concentration or high concentrations of aggressive electrolytes like chloride. When the protective film is disrupted, the uncovered area becomes anodic and the exposed metal becomes cathodic, triggering a galvanic reaction that further exacerbates corrosion in the localized area. The main challenge with pitting corrosion lies in its detection: since it attacks the deep structure of the material while leaving minimal visible damage on the surface, identifying and assessing the full extent of the damage is often extremely difficult. Additionally, corrosion products can accumulate over the pits, further concealing the damage and complicating detection efforts.

The ability to effectively predict, monitor, and prevent corrosion is critical for minimizing its economic impact and ensuring the longevity of materials used in industrial and structural applications. As such, the development of advanced mathematical models and computational tools to understand the underlying mechanisms of corrosion is crucial. Through the use of computational simulations, researchers can model and analyze corrosion processes in a controlled environment, providing valuable insights into the factors that influence material degradation in specific settings. These tools not only enhance our understanding of corrosion but also enable the development of more effective material protection techniques and predictive maintenance strategies, which are essential for cost-effective corrosion management.

A variety of analytical and computational models have been developed to study pitting corrosion, with a particular focus on simulating the movement of the corrosive front. The movement of this boundary is a defining feature of corrosion phenomena, and accurately simulating the progression of this boundary is one of the most challenging aspects of corrosion modeling. Based on the treatment of this moving boundary, corrosion models can generally be classified into two categories: non-autonomous models and autonomous models \citep{jafar}.

Non-autonomous models, which are more traditional, typically rely on transport equations that are solved using finite element (FE) or finite difference (FD) methods. These models require additional strategies to simulate the motion of the corrosion front, often necessitating remeshing the domain or the use of moving boundary techniques \citep{duddu,vagb,sun}. Thus, while effective, these methods can be computationally demanding \citep{jafar}. On the other hand, autonomous models, which represent a more modern approach, treat the evolution of the growing pits autonomously, without the need for explicit tracking of the boundary. Notable examples of these models are based on finite volume methods, cellular automata, peridynamic and phase-field formulations.

Phase-field models introduce dimensionless phase-field variables that evolve autonomously according to the minimization of a free energy functional, on the basis of thermodynamic principles. They are particularly advantageous for their inherent numerical stability, that makes them an ideal choice for simulating corrosion phenomena \citep{ansari}. However, the primary limitation of phase-field models lies in their computational cost, as they are based on systems of coupled and stiff partial differential equations (PDEs). Recent works on the numerical solution of phase-field models have considered monolithic implicit time integrators avoiding the need for remeshing techniques \citep{Valizadeh, Wen}. Nevertheless, as is typical for implicit methods, the resulting nonlinear systems of equations need to be solved iteratively using a Newton-Raphson type method. Advanced approaches have also considered semi-implicit schemes that preserve mass and ensure energy dissipation \citep{Chen,Zhang}. Nevertheless, the governing equations often need to be solved over large domains and extended time periods, leading to high computational demands \citep{bahram,jafar,choudhuri}.

Indeed, most kind of models for corrosion listed above have a common drawback: their solution is computationally expensive. Even reliable and accurate simulations are useless if the computations require excessively long times. Specifically, models based on finite volume methods require the implementation of a phase strategy to autonomously follow the pit growth \citep{cui,scheiner}. Peridynamic formulations are based on the solution of nonlocal versions of diffusion equations \citep{bobaru}. Although the non locality of the model offers a more complete description of the corrosive deterioration, its numerical treatment makes the cost of computations very high \citep{jafar}. Cellular automata models evolve on the basis of specific probabilistic rules and random-walk processes \citep{stafiej}. As these models are not PDE based, they have low computational costs. However, due to the lack of time-dependent dynamics the predictive capability of these models is questionable \citep{ansari,jafar}.

The development of new, more efficient numerical methods that can reduce computational costs while maintaining accuracy is critical for advancing corrosion modeling and facilitating its application in industries that rely on predictive maintenance. Quoting \citep[§5.5]{jafar} ``any new numerical method that reduces the computational cost for corrosion models can make a large impact''.

In light of these challenges, current research is focused on improving the computational efficiency of corrosion models, particularly through the development of novel time-stepping schemes and optimization techniques. While some progress has been made in reducing computational times, e.g., by using adaptive meshing strategies, simulations involving complex geometries or large domains remain highly resource-intensive.

This issue is especially relevant given the substantial economic impact of corrosion, as mentioned at the beginning of this section, highlighting the need for predictive tools that can be used in industrial environments. However, the high computational cost of conventional corrosion solvers has so far limited their practical use for real-time maintenance planning. In response, this work is devoted to the efficient solution of the phase-field model in \citep{mai}, which presents a system of two coupled PDEs to simulate pitting corrosion of 304 stainless steel. By developing a matrix-oriented IMEX strategy, we address the computational bottleneck, enabling accurate simulations within time frames suitable for predictive maintenance. This contributes to the broader goal of making corrosion prediction, prevention, and management more accessible and impactful. In particular, being able to run phase-field simulations within timeframes compatible with industrial needs enables the integration of accurate corrosion models into predictive maintenance strategies. This, in turn, facilitates faster decision-making, minimizes downtime, and promotes cost-effective asset management in industrial settings.

The plan of the paper is as follows. In Section \ref{sec:problem} we present the pitting corrosion model in \citep{mai}. In Section \ref{sec:rect} we introduce a matrix-oriented approach for the solution of the problem on rectangular domains. The equations are first discretized in space by finite differences and then integrated by first- and second-order accurate IMEX methods. The large systems to solve at each step are then reformulated as Sylvester equations. These are efficiently solved by the matrix-oriented technique in \citep{Simoncini}. To tackle scenarios more frequently encountered in practical applications, in Section~\ref{sec:holes} we introduce appropriate IMEX methods for domains with holes. Their efficient implementation by the matrix-oriented approach is discussed in Section \ref{sec:iterative}, where specific iterative procedures are introduced and their convergence is analyzed together with a rigorous study of error propagation. In Section \ref{sec:tests} the efficiency of the methods is compared to that of the most efficient methods that have been proposed for the solution of this model, and with another efficient IMEX method in the literature \citep{Ascher} that has been recently considered for the solution of phase-field models \citep{Fu,Zhang}. It is shown that not only the new methods are significantly cheaper in terms of computational execution times, but that on a standard computer they are able to make actual predictions with the same level of accuracy. Further numerical experiments are proposed to confirm the theoretical convergence results, and additional tests illustrate the performance of the proposed methods in a three-dimensional setting. Finally, Section \ref{sec:concl} reports some conclusive remarks.

\section{The problem}\label{sec:problem}

The following system of PDEs,
\begin{align}\label{AC}
\phi_t=&\,D_\phi \Delta\phi+F_1(\phi,c),\\\label{CH}
c_t=&\, D_c\Delta(c+F_2(\phi)),
\end{align}
where
\begin{align}\label{F1F2}
F_1(\phi,c)=&\,2AL(1-c_L)[c-h(\phi)(1-c_L)-c_L]h'(\phi)-\omega Lg'(\phi),\qquad F_2(\phi)=(c_L-1)h(\phi),
\end{align}
and
\begin{equation}\label{hg}
g(\phi)=\phi^2(1-\phi)^2,\qquad  h(\phi)=-2\phi^3+3\phi^2.
\end{equation}
has been derived in \citep{mai} by variational principles as a model for a corrosion system composed by a metal immersed in an electrolyte solution. 

It is a phase-field model that couples an Allen-Cahn equation \eqref{AC} with a Cahn-Hilliard equation \eqref{CH} that govern the phase transition and the variation of the molar concentration of dissolved ions, respectively.

We assume that the vector of the space variables $\mathbf{x}$ belongs to the set $\Omega\subset \mathbb{R}^n$, with $n\in\{2,3\}$ and $t\in[0,T]$.

The two dependent variables $\phi$ and $c$ are such that $(\phi(\mathbf{x},t),c(\mathbf{x},t))\in[0,1]^2$. Function $\phi$ is the phase-field variable and equation \eqref{AC} models the evolution of the corrosion interface. Function $c$ is the state variable measuring the corrosion level of point $(x,y)$ at time $t$. In particular, $c(\mathbf{x},t)=0$ if point $\mathbf{x}$ has been completely corroded at time $t$, whereas $c(\mathbf{x},t)=1$ if point $\mathbf{x}$ has not been attacked by corrosion at time $t$. A value in the range between 0 and 1 is assigned to intermediate levels of corrosion. 

The physical interpretation and values of the parameters in \eqref{AC}--\eqref{hg} are listed in Table \ref{tab:parameters}. These parameters, reported in \citep{scheiner2} have been used in \citep{mai} for the validation of this problem as a model for the corrosion of stainless steel exposed to a NaCl electrolyte solution. In particular, we have calculated 
$$D_\phi= L \alpha_\phi,$$
where $\alpha_\phi$ is the gradient energy coefficient and it is related to the height of the double well potential $\omega$ through the interface energy $\hat\sigma$ and the interface thickness $l$ as follows:
$$\hat\sigma \approx \sqrt{16 \omega \alpha_\phi},\qquad l=\alpha^*\sqrt{\frac{2\alpha_\phi}{\omega}}.$$
As in \citep{GaoFD,GaoFE,mai}, we have set $l=5 \mu$m, $\hat\sigma=10$ J/m$^2$ and the constant parameter $\alpha^*=2.94$, and thus we have computed the values of $\omega$, $\alpha_\phi$ ($\sim 3.01\cdot 10^{-6}$ J$/$m) and $D_\phi$ in Table \ref{tab:parameters}.
\begin{table}[t!]
\centerline{\begin{tabular}{||c||c||c||}
\hline
\hline
Parameter & Physical meaning & Value\\
\hline
$L$ & Interface kinetics coefficient&2 m$^3/$(Js)\\
$A$  & Free energy curvature &$5.35\cdot 10^{7}$ J$/$mol\\
$D_\phi$ & Diffusion coefficient& $6.02\cdot 10^{-6} $ m$^2/$s\\
$D_c$ & Diffusion coefficient& $ 8.50\cdot 10^{-10}$ m$^2/$s\\
$c_L$ & Normalized equilibrium coefficient for liquid state&$3.57\cdot 10^{-2}$\\
$\omega$ & Height of double well potential&$2.08\cdot 10^{6}$ J$/$m$^3$\\
\hline
\hline
\end{tabular}}
\caption{Physical meaning and value of the parameters in model \eqref{AC}--\eqref{hg}}\label{tab:parameters}
\end{table}

While the model has been validated and verified through analytical and experimental data, it also highlights several challenges in its numerical solution. Specifically, the coupled nature of the equations requires extensive computational resources. Additionally, as observed in \citep{GaoFE,GaoFD}, the Allen-Cahn equation exhibits extreme stiffness due to the very large ratio between the reaction term and the diffusion coefficient ($\sim 10^{13}$, see Table \ref{tab:parameters}), which is a major challenge in the solution process. In fact, conventional backward differentiation formulae are impractical as they require to advance by excessively small time steps \citep{GaoFE,GaoFD}.

More advanced and modern time integrators, such as exponential Rosenbrock methods, can advance by larger time steps. However, the computational cost of each iteration is excessively high. For example, several hours of computations are reported in \citep{GaoFD} to simulate a 225-second corrosion process of a rectangular bar on a standard workstation. As shown in \citep{GaoFE}, these times can be reduced by using adaptive meshes in space. However, the computational experiments are still slower than the real phenomena and thus the existing algorithms are far from making actual predictions on standard computers. To the best of our knowledge, these are currently the most efficient numerical methods in the literature for the solution of the corrosion problem \eqref{AC}--\eqref{hg}.

This highlights the need for further advancements in computational techniques for the solution of phase-field models such as \eqref{AC}--\eqref{hg}. Providing an efficient strategy is the main goal of this paper. We start by considering in the next section the simplest case that is, in fact, when the corroding sample is a rectangular bar.

\section{Matrix-oriented approach on rectangular domains}\label{sec:rect}
We assume here that the domain is bi-dimensional and rectangular, i.e., $(x,y)\in \Omega=[0,L_x]\times [0,L_y]$. The extension to 3D domains is provided in Section \ref{sec:3drec}, whereas the case of non-rectangular domains is discussed in Section \ref{sec:holes}.
A suitable initial condition,
\begin{equation}\label{initcon}
\phi(x,y,0)=\phi_0(x,y), \quad c(x,y,0)=c_0(x,y),\qquad (x,y)\in {\Omega}
\end{equation}
is assigned to system \eqref{AC}--\eqref{hg} together with a Dirichlet or Neumann boundary condition on each side of the rectangular domain.
\subsection{Space discretization}
For the space discretization of equations \eqref{AC} and \eqref{CH} we consider a uniform mesh on $\Omega$ defined by the nodes
$$(x_i,y_j)=(i\Delta x,j\Delta y),\qquad i=0,\ldots, m_x+1,\qquad j=0,\ldots,m_y+1,$$
so that
$\Delta x={L_x}/{(m_x+1)}$ and $\Delta y={L_y}/{(m_y+1)}.$
Let ${u}$ be a generic function defined on this spatial mesh. We denote by ${u}_{i,j}={u}_{i,j}(t)$ the approximations of ${u}(x_i,y_j,t)$.
The Laplacian of ${u}$ is approximated by using the standard second-order centered difference formula,
\begin{equation}\label{centdiff}
 \Delta u_{i,j}:=\frac{{u}_{i-1,j}-2{u}_{i,j}+{u}_{i+1,j}}{\Delta x^2}+\frac{{u}_{i,j-1}-2{u}_{i,j}+{u}_{i,j+1}}{\Delta y^2}\approx\Delta {u}(x_i,y_j,t) .
\end{equation}
The set of points  $\{(x_i,y_j),(x_{i-1},y_j),(x_{i+1},y_j),(x_{i},y_{j-1}),(x_{i},y_{j+1})\}$
is called the stencil of approximation \eqref{centdiff} centered at $(x_i,y_j)$ and we say that nodes
$$(x_{i-1},y_j),\quad (x_{i+1},y_j),\quad (x_{i},y_{j-1}),\quad (x_{i},y_{j+1})$$
are the neighbors of node $(x_i,y_j)$.

Let $\boldsymbol{\phi}(t)$ and $\mathbf{c}(t)$ be the vectors of the approximations $\phi_{i,j}(t)$ and $c_{i,j}(t),$ respectively. Here and henceforth, when the approximations on a two dimensional grid are collected into a vector, we assume that they are sorted following a lexicographic order of the nodes. Then, the space discretization of system \eqref{AC}--\eqref{CH} can be cast into the form
\begin{align}\label{ACode}
\frac{\mathrm{d}\boldsymbol{\phi}}{\mathrm{d}t}=&\,D_\phi (\mathcal{M}\boldsymbol{\phi}+\boldsymbol{\psi}_\phi)+F_1(\boldsymbol{\phi},\mathbf{c}),\\\label{CHode}
\frac{\mathrm{d}\mathbf{c}}{\mathrm{d}t}=&\, D_c[\mathcal{M}(\mathbf{c}+F_2(\boldsymbol{\phi}))+\boldsymbol{\psi}_c+\boldsymbol{\psi}_{F_2}],
\end{align}
where functions $F_1$ and $F_2$ are evaluated entrywise on vectors $\boldsymbol{\phi}$ and $\mathbf{c}$, and the vectors $\boldsymbol{\psi}_r=\boldsymbol{\psi}_r(t)$ include boundary contributions depending on function $r\in\{c,\phi,F_2(\phi)\}$.
Matrix $\mathcal{M}$ is the discrete Laplace operator. In the case of Dirichlet boundary conditions it has the Kronecker sum structure,
\begin{equation}\label{kronsum}
\mathcal{M}=I_{m_y}\otimes \mathcal{M}_x+\mathcal{M}_y\otimes I_{m_x},
\end{equation}
where $I_{m}$ denotes the identity matrix of dimension ${m}$ and 
\begin{equation}\label{MD}
\mathcal{M}_r=\frac{1}{\Delta r^2}\left[\begin{array}{ccccc}
-2 & 1 & & &\\
1 & -2 & 1 & &\\
& \ddots & \ddots & \ddots &\\
& & 1 & -2 & 1\\
& & & 1 &-2
\end{array}\right],
\end{equation}
of dimension $m_r$, for $r\in\{x,y\}$. In the case of Neumann boundary conditions matrix $\mathcal{M}$ is still a Kronecker sum of the form \eqref{kronsum} with
\begin{equation}\label{MN}
\mathcal{M}_r=\frac{1}{\Delta r^2}\left[\begin{array}{ccccc}
-2 & 2 & & &\\
1 & -2 & 1 & &\\
& \ddots & \ddots & \ddots &\\
& & 1 & -2 & 1\\
& & & 2 &-2
\end{array}\right],
\end{equation}
but the dimension of all involved matrices is augmented by two to include the approximations at the boundary. Boundary conditions of Dirichlet kind on some edges and of Neumann type on others, can be treated by modifying matrices $\mathcal{M}_x$ and  $\mathcal{M}_y$ accordingly. For the rest of the discussion we assume that the number of unknowns is $m_xm_y$ and so $\mathcal{M}$ is defined as in \eqref{kronsum}.
\subsection{Time integration}
For the time integration of system \eqref{ACode}--\eqref{CHode} we introduce the nodes
\begin{equation*}
t_n=n\Delta t,\qquad n=0,\ldots,N,\qquad \Delta t=T/N,
\end{equation*}
and use implicit-explicit (IMEX) methods. For the sake of simplicity we introduce these methods for a generic initial value problem of the form
\begin{equation}\label{genODE}
\frac{\mathrm{d}\mathbf{u}(t)}{\mathrm{d} t}=\mathcal{M}\mathbf{u}(t)+F(\mathbf{u}(t)),\qquad \mathbf{u}(0)=\mathbf{u}_0\\
\end{equation}
where matrix $\mathcal{M}$ is the Kronecker sum \eqref{kronsum}, and $\mathbf{u}\in\mathbb{R}^{m_xm_y}$. The methods in this section treat implicitly the linear term $\mathcal{M}\mathbf{u}(t)$, and explicitly the nonlinear term $F(\mathbf{u}(t))$. The linear systems obtained at each time step have dimension $m_xm_y$, and thus their solution is burdensome on large domains. By exploiting the Kronecker sum structure of the coefficient matrix, we will provide an equivalent formulation of these systems in Sylvester form, showing how they can be solved highly efficiently by the matrix-oriented strategy in \citep{Simoncini}.
\subsubsection{First-order accuracy}\label{sec:FirstOrder}
 The first method that we consider is the IMEX Euler method. Denoting by $\mathbf{u}^n$ the approximation to $\mathbf{u}(t_n)$, the IMEX Euler method applied to system \eqref{genODE} amounts to
$$\mathbf{u}^{n+1}-\mathbf{u}^n=\Delta t(\mathcal{M}\mathbf{u}^{n+1}+F(\mathbf{u}^{n})),$$
or, equivalently,
\begin{equation}\label{genIMEX}
(I_{m_xm_y}-\Delta t\mathcal{M})\mathbf{u}^{n+1}=\mathbf{u}^n+\Delta t F(\mathbf{u}^n).
\end{equation}
For the efficient implementation of IMEX Euler method applied to problem \eqref{genODE} we follow the matrix-oriented approach in \citep{Simoncini} and define the matrix $U^n\in\mathbb{R}^{m_x\times m_y}$ such that 
\begin{equation}\label{matrU}
\mathbf{u}^n=\mathrm{vec}(U^n),
\end{equation}
where ``$\mathrm{vec}$'' denotes the vectorization operation that converts a matrix into a vector, by stacking its columns on top of one another. With $\mathcal{M}$ in \eqref{kronsum}, it holds that
$$\mathcal{M}\mathbf{u}^{n+1}=\mathrm{vec}(\mathcal{M}_xU^{n+1}+U^{n+1}\mathcal{M}_y^T).$$
Thus, method \eqref{genIMEX} can be equivalently written in matrix form as
\begin{align}\label{Sylv}
(I_{m_x}-\Delta t \mathcal{M}_x)U^{n+1}-\Delta tU^{n+1}\mathcal{M}_y^T=&\,U^n+\Delta t F(U^n),
\end{align}
where function $F$ is computed entrywise on matrix arguments. 
Equation \eqref{Sylv} is a Sylvester equation and can be solved efficiently as described in Section \ref{sec:MatOriApp} (see also \citep{Simoncini}).

When the problem to be solved is highly stiff, as in the case of equation \eqref{ACode}, one may recur to the relaxation procedure in \citep{Simoncini} to achieve stability with larger time steps. According to this technique, equation \eqref{genODE} is relaxed as the equivalent problem,
\begin{equation}\label{relgen}
\frac{\mathrm{d}\mathbf{u}(t)}{\mathrm{d} t}=\mathcal{M}\mathbf{u}(t)-w \mathbf{u}(t)+F(\mathbf{u}(t))+w \mathbf{u}(t),\qquad \mathbf{u}(0)=\mathbf{u}_0.\\
\end{equation}
where $w$ is a free parameter.

The vector based form of IMEX Euler method applied to equation \eqref{relgen} amounts to
\begin{equation}\label{relIMEX}
(I_{m_xm_y}-\Delta t(\mathcal{M}-wI_{m_xm_y}))\mathbf{u}^{n+1}=\mathbf{u}^n+\Delta t (w\mathbf{u}^n+F(\mathbf{u}^n)).
\end{equation}

The free parameter $w$ is to be chosen in order to have a stable method at each time step. Based on the linear stability theory in \citep{Rosales, Seibold, Pagano} an IMEX method is stable for problem 
$$\frac{\mathrm{d}\mathbf{u}(t)}{\mathrm{d} t}=\mathcal{A}_w\mathbf{u}(t)+\mathcal{B}_w\mathbf{u}(t),$$
if
\begin{equation}\label{fov}
W_p(\mathcal{A}_w,\mathcal{B}_w)\subseteq \mathcal{D},
\end{equation}
where $\mathcal{A}_w\mathbf{u}(t)$ and $\mathcal{B}_w\mathbf{u}(t)$ are to be approximated implicitly and explicitly, respectively, and
\begin{itemize}
\item $W_p(\mathcal{A},\mathcal{B})=W\left((-\mathcal{A})^{\frac{p}{2}-1}\mathcal{B}(-\mathcal{A})^{-\frac{p}{2}}\right),$ 
where $W(\mathcal{X})$ denotes the field of values of matrix $\mathcal X$, that is
$$W(\mathcal{X})=\left\{\langle\mathbf{v},\mathcal{X}\mathbf{v}\rangle:\|\mathbf{v}\|=1,\mathbf{v}\,\, \text{complex vector}\right\}.$$
\item $\mathcal{D}$ denotes the unconditional stability diagram of the IMEX method. For IMEX Euler,
\begin{equation}\label{Dord1}
\mathcal{D}=\{\mu\in\mathbb{C}: |\mu|<1\},
\end{equation}
that is the unit circle centered at the origin.
\end{itemize}
Considering the linearization of the general problem \eqref{relgen} around time $t_n$, we set $w$ by requiring that \eqref{fov} holds true at each time step with 
\begin{equation*}
\mathcal{A}_w=\mathcal{M}-wI_{m_xm_y},\qquad \mathcal{B}_w=\mathcal{J}_{F}+wI_{m_xm_y},
\end{equation*}
where $\mathcal{J}_{F}$ denotes the Jacobian matrix of $F$ at time $t_n$.

Following similar arguments as before, method \eqref{relIMEX} is equivalent to the solution of the Sylvester equation
\begin{align*}
(I_{m_x}-\Delta t (\mathcal{M}_x-w I_{m_x}))U^{n+1}-&\Delta tU^{n+1}\mathcal{M}_y^T=U^n+\Delta t (F(U^n)+w U^n).
\end{align*}

For the solution of system \eqref{ACode}--\eqref{CHode}, we apply the relaxed IMEX Euler method \eqref{relIMEX} to the highly stiff equation \eqref{ACode}, and the standard IMEX Euler method \eqref{genIMEX} to equation \eqref{CHode}. 
Denoting by $\boldsymbol{\phi}^n$ and $\mathbf{c}^n$ the vectors of the approximations of $\boldsymbol{\phi}(t_n)$ and $\mathbf{c}(t_n)$, respectively, this amounts to the solution of
\begin{align*}
[I_{m_xm_y}-\Delta t(D_\phi\mathcal{M}-wI_{m_xm_y})]\boldsymbol{\phi}^{n+1}=&\,\boldsymbol{\phi}^n+\Delta t[D_\phi \boldsymbol{\psi}_\phi(t_{n+1})+w \boldsymbol{\phi}^n+F_1(\boldsymbol{\phi}^n,\mathbf{c}^n)],\\
(I_{m_xm_y}-\Delta tD_c\mathcal{M})\mathbf{c}^{n+1}=&\,\mathbf{c}^n+\Delta tD_c [\mathcal{M}F_2(\boldsymbol{\phi}^{n+1})+ \boldsymbol{\psi}_c(t_{n+1})+\boldsymbol{\psi}_{F_2}(t_{n+1})].
\end{align*}
The free parameter $w$ is chosen such that \eqref{fov} is satisfied at each time step with 
\begin{equation}\label{AwBw2}
\mathcal{A}_w=D_\phi\mathcal{M}-wI_{m_xm_y},\qquad \mathcal{B}_w=\mathcal{J}_{F_1}+wI_{m_xm_y}.
\end{equation}
\begin{remark}\label{rem:stab}
We observe that, for $p=0$, \eqref{fov} amounts to $W\left((-\mathcal{A}_w)^{-1}\mathcal{B}_w\right)\subseteq \mathcal{D}.$
One necessary condition for unconditional linear stability is that 
\begin{equation}\label{eq:CNstab}
\sigma \left((-\mathcal{A}_w)^{-1}\mathcal{B}_w\right)\subseteq{\mathcal{D}}.
\end{equation} 
In broad terms, if the space grid is not too fine, (i.e., such that $D_\phi$ is small compared to $ h^2\rho(\mathcal{J}_{F_1})$, where $h=\min\{\Delta x, \Delta y\}$, $\mathcal{J}_{F_1} = \partial_{\boldsymbol{\phi}}F_1(\boldsymbol{\phi},\mathbf{c})$,  and $\rho$ denotes the spectral radius), then $$(-\mathcal{A}_w)^{-1}\mathcal{B}_w \sim \frac{1}w \mathcal{J}_{F_1} + I_{m_xm_y}.$$
Since the largest eigenvalue of $\mathcal{J}_{F_1}$ is negative and very large in modulus ( $\sim 10^8$, with order of magnitude reflecting the one of the free energy curvature $A$ in Table \ref{tab:parameters}), we can expect that \eqref{eq:CNstab} is satisfied choosing $w\sim \rho(\mathcal{J}_{F_1})$. For the numerical experiments in this paper, condition \eqref{eq:CNstab} turns out to be also sufficient for stability.

{\color{black}Note that $\mathcal{J}_{F_1}$ depends explicitly on $\phi$ and $c$. Therefore, to maintain stability over long time intervals, the spectral radius $\rho(\mathcal{J}_{F_1})$ must be regularly updated. However, since $F_1$ is computed entry-wise, the Jacobian matrix $\mathcal{J}_{F_1}$ is diagonal, allowing its eigenvalues to be computed inexpensively. An alternative approach, used in Section \ref{sec:tests}, is to fix a unique value of $w$ in advance. This can be done by estimating the maximum values of $|\partial_{{\phi}}F_1({\phi},{c})|$ a priori, considering that ${\phi}$ and ${c}$ lie both within the interval $[0,1]$ and that their values are typically close.} It is worth to note that overestimating the value of $w$ significantly beyond what is necessary for stability, can lead to a loss of accuracy.
\end{remark}

Let $\Phi^n, C^n, \Psi_z \in\mathbb{R}^{m_x\times m_y}$ where $z\in \{c,\phi,F_2\}$ be the matrices such that $\boldsymbol{\phi}^n=\mathrm{vec}(\Phi^n),$ $\mathbf{c}^n=\mathrm{vec}(C^n),$ $\boldsymbol{\psi}_z^n=\mathrm{vec}(\Psi_z^n).$
The actual implementation of such method is performed by solving the Sylvester equations
\begin{align}\label{finAC}
&(I_{m_x}-\Delta t (D_\phi\mathcal{M}_x-wI_{m_x}))\Phi^{n+1}-\Delta tD_\phi\Phi^{n+1}\mathcal{M}_y^T=\Phi^n+\Delta t (D_\phi \Psi_\phi(t_{n+1})+w\Phi^n+F_1(\Phi^n,C^n)),\\\label{finCH}
&(I_{m_x}-\Delta tD_c \mathcal{M}_x)C^{n+1}-\Delta tD_cC^{n+1}\mathcal{M}_y^T=C^n+\Delta t D_c [\mathcal{M}_x F_2(\Phi^{n+1})+F_2(\Phi^{n+1})\mathcal{M}_y^T+ \Psi_c(t_{n+1})+\Psi_{F_2}(t_{n+1})],
\end{align}
according to the efficient strategy in Section \ref{sec:MatOriApp}. Observe that equations \eqref{finAC} and \eqref{finCH} are decoupled at each time step and the value of $\Phi^{n+1}$ can be substituted in equation \eqref{finCH} after solving equation \eqref{finAC}.
\subsubsection{Second-order accuracy}\label{sec:SecondOrder}
 Higher accuracy can be obtained by the second-order two-step backward difference formula (IMEX 2SBDF). Applied to equation \eqref{genODE} it amounts to
$$3\mathbf{u}^{n+2}-4\mathbf{u}^{n+1}+\mathbf{u}^{n}=2\Delta t(\mathcal{M}\mathbf{u}^{n+2}+2F(\mathbf{u}^{n+1})-F(\mathbf{u}^{n})),$$
that is,
\begin{equation}\label{2SBDF}
(3I_{m_xm_y}-2\Delta t\mathcal{M})\mathbf{u}^{n+2}=4\mathbf{u}^{n+1}-\mathbf{u}^n+2\Delta t(2F(\mathbf{u}^{n+1})-F(\mathbf{u}^{n})).
\end{equation}
Similarly as IMEX Euler method, this scheme can be equivalently formulated as the Sylvester equation,
\begin{align*}
(3I_{m_x}-2\Delta t \mathcal{M}_x)U^{n+2}-2\Delta tU^{n+2}\mathcal{M}_y^T=&\,4U^{n+1}-U^n+2\Delta t (2F(U^{n+1})-F(U^n)),
\end{align*}
where, for any $n$, matrix $U^{n}$ is defined as in \eqref{matrU}.

The relaxed version of method \eqref{2SBDF}, applied to equation \eqref{relgen}, amounts to
\begin{equation}\label{rel2SBDF}
(3 I_{m_xm_y}-2\Delta t(\mathcal{M}-wI_{m_xm_y}))\mathbf{u}^{n+2}=4\mathbf{u}^{n+1}-\mathbf{u}^n+2\Delta t(2F(\mathbf{u}^{n+1})+2w\mathbf{u}^{n+1}-F(\mathbf{u}^{n})-w\mathbf{u}^{n}).
\end{equation}
Parameter $w$ is chosen according to the linear stability analysis in Section \ref{sec:FirstOrder}, where $\mathcal{D}$ is the stability diagram of the IMEX 2SBDF method \citep{Seibold},
\begin{equation}\label{Dord2}
\mathcal{D}=\{\mu\in\mathbb{C}: z^2-\mu(2z-1) \text{ has all roots in the unit circle}\}.
\end{equation}
Again, we rewrite problem \eqref{rel2SBDF} as the Sylvester equation
\begin{align*}
(3I_{m_x}-2\Delta t (\mathcal{M}_x-w I_{m_x}))&U^{n+2}-2\Delta tU^{n+2}\mathcal{M}_y^T=4U^{n+1}-U^n+2\Delta t (2F(U^{n+1})+2wU^{n+1}-F(U^n)-wU^n).
\end{align*}
The second-order accurate method for system \eqref{ACode}--\eqref{CHode} is then obtained by applying the relaxed IMEX 2SBDF method to equation \eqref{ACode} and the classical IMEX 2SBDF method to equation \eqref{CHode}. With the same notation used in \eqref{finAC}--\eqref{finCH}, the actual solution of this system is performed by solving the Sylvester equations
\begin{align}\nonumber
(3I_{m_x}-2\Delta t (D_\phi\mathcal{M}_x&-w I_{m_x}))\Phi^{n+2}-2\Delta tD_\phi\Phi^{n+2}\mathcal{M}_y^T  \\\label{finACord2}
=&\,4\Phi^{n+1}-\Phi^n+2\Delta t (2F_1(\Phi^{n+1},C^{n+1})+2w \Phi^{n+1}-F_1(\Phi^n,C^n)-w \Phi^n)+2\Delta tD_\phi\Psi_\phi(t_{n+2}),\\\nonumber
(3I_{m_x}-2\Delta tD_c \mathcal{M}_x)&C^{n+2}-2\Delta tD_cC^{n+2}\mathcal{M}_y^T\\\label{finCHord2}=&\,4C^{n+1}-C^n+2\Delta t D_c [\mathcal{M}_x F_2(\Phi^{n+2})+F_2(\Phi^{n+2})\mathcal{M}_y^T+ \Psi_c(t_{n+2})+\Psi_{F_2}(t_{n+2})],
\end{align}
as described below.
\subsection{Efficient solution of Sylvester problems}\label{sec:MatOriApp}
In this section we describe the approach in \citep{Simoncini} for the efficient solution of the Sylvester equations that define the matrix formulations \eqref{finAC}--\eqref{finCH} and \eqref{finACord2}--\eqref{finCHord2} of the IMEX methods in the previous sections. We observe that all these equations are of the type
\begin{equation}\label{genSylv}
(aI_{m_x}+b\mathcal{M}_x)X+bX\mathcal{M}_y^T=Y,
\end{equation}
where $a,b\in\mathbb{R}$, $Y$ is a computable matrix. 

Equation \eqref{genSylv} can be solved efficiently by considering a spectral (or Schur) decomposition of matrices $\mathcal{M}_r$ with $r\in\{x,y\}$,
\begin{equation}\label{specMz}
\mathcal{M}_r=\Gamma_r\Lambda_r \Gamma_r^{-1},
\end{equation}
where $\Lambda_r$ is the diagonal matrix with the eigenvalues of $\mathcal{M}_r$ and $\Gamma_r$ is the non-singular matrix of the corresponding eigenvectors. Then, equation \eqref{genSylv} yields the following method for the new variable $Z=\Gamma_x^{-1}X\Gamma_y^{-T}$,
$$(aI_{m_x}+b\Lambda_x)Z+bZ\Lambda_y=\Gamma_x^{-1}Y\Gamma_y^{-T}.$$
Thus, the solution of \eqref{genSylv} is 
\begin{equation}\label{solSyl}
X=\Gamma_x Z\Gamma_y^{T},\qquad \text{with}\qquad Z=\Upsilon\circ(\Gamma_x^{-1}Y\Gamma_y^{-T}),
\end{equation}
where $\circ$ denotes the Hadamard product and $\Upsilon$ is the matrix with entries $$\Upsilon_{i,j}=\left(a+b(\Lambda_x)_{i,i}+b(\Lambda_y)_{j,j}\right)^{-1}.$$

We observe that an implementation of the IMEX methods relying on the vector-based formulations \eqref{genIMEX}, \eqref{relIMEX}, \eqref{2SBDF} or \eqref{rel2SBDF}, requires the solution of large linear systems of dimension $ m_x m_y \!\times\! m_x m_y$ at each time step. Instead, the bulk of the computation of the matrix-based algorithm described above, is given by the products of small matrices  
($\Gamma_x\in\mathbb{R}^{m_x\times m_x},$ $\Gamma_y\in\mathbb{R}^{m_y\times m_y},$ $Y\in\mathbb{R}^{m_x\times m_y}$) at each time step. The spectral decompositions \eqref{specMz} are in fact performed only once before starting the step by step process, without needing to recompute it at each time step.
\begin{remark}
A key computational challenge in solving \eqref{genSylv} for large, industrial-scale domains arises from the fact that, although the two coefficient matrices are sparse, the solution matrix $X$ is generally dense. This makes storing $X$ infeasible for high-dimensional systems. 

In corrosion modelling over very large domains, it is often the case that large portions of the domain are either entirely intact or fully corroded. As a result, most of the entries of the matrices $\Phi^n$ and $C^n$ are either 1 or 0. Then also most entries of $Y$ take the same value (see the right hand sides in \eqref{finAC}--\eqref{finCH} and \eqref{finACord2}--\eqref{finCHord2} and the formulas for $F_1$ and $F_2$ in \eqref{F1F2}--\eqref{hg}). In such scenarios, $Y$ has a low-rank structure.

This opens the door to the computation of low-rank approximations of $X$ that avoid the prohibitive storage of the full matrix. For this purpose, common numerical methods in the literature include projection-type approaches (often relying on Krylov subspace techniques) and matrix updating strategies (such as, the Alternating Direction Implicit (ADI) iteration). For a detailed overview of these methods, we refer the reader to the review \citep{Simoncinirev}.
\end{remark}
\subsection{Implementation Sketch}\label{sec:alg1}
In this section we briefly recap in a schematic way the workflow of the proposed strategy when problem \eqref{AC}--\eqref{CH} is defined on a rectangular domain. Algorithm \ref{alg:rec} outlines the steps for the implementation of IMEX Euler method \eqref{finAC}--\eqref{finCH}. Note that the majority of the computational efforts stems from the full-matrix products in \eqref{solSyl} required to solve the two Sylvester problems at every time step. Therefore, for a total of $N$ time steps, the overall computational cost of the algorithm can be roughly estimated as $\sim 4Nm_xm_y(m_x+m_y)$ floating-point operations.  The procedure for the second-order IMEX 2SBDF method is analogous; however, a starting procedure is required since this is a two-step scheme. 

From a memory perspective, the algorithms are efficient, requiring the storage of matrices whose size corresponds to the spatial resolution in one dimension only. As a result, the overall memory usage remains modest.

\begin{algorithm}
\caption{IMEX Euler method (rectangular domains)}\label{alg:rec}
\begin{algorithmic}[1]
\State Initialize the system parameters and functions and the spatial grid.
\State Define the initial conditions $\boldsymbol{\phi}^0$ and $\mathbf{c}^0$.
\State Build the matrices $\mathcal{M}_x$ and $\mathcal{M}_y$ in \eqref{genSylv} considering the assigned boundary conditions on $\partial\Omega$.

\State Compute their spectral factorization.
\State Choose a value for $w$ considering Remark \ref{rem:stab}.

\For{$n=1,\ldots N$}\Comment{Advance in time}
\State Compute the right hand side of the Sylvester problem \eqref{finAC}.
\State Compute $\boldsymbol{\phi}^{n+1}$ using the formulae in \eqref{solSyl}.
\State Compute the right hand side of the Sylvester problem \eqref{finCH}.
\State Compute $\mathbf{c}^{n+1}$ using the formulae in \eqref{solSyl}.

\EndFor
\end{algorithmic}
\end{algorithm}

\subsection{The 3D scenario}\label{sec:3drec}
In this section, we extend the previous arguments to three-dimensional rectangular domains, i.e., $\Omega=[0,L_x]\times[0,L_y]\times[0,L_z]$. The approach described here is based, as for the two-dimensional setting, on the spectral factorization of small matrices and modifies the strategy in \citep{Kirsten} that instead uses Schur factorization. The fully discrete problem is transformed into a sequence of Sylvester equations. We take a uniform grid
$$(x_i,y_j,z_k),\quad i=0,\ldots,m_x+1,\quad j=0,\ldots,m_y+1,\quad k=0,\ldots,m_z+1,$$
and the second-order approximation of the Laplacian
\begin{equation*}
 \Delta u_{i,j}=\frac{{u}_{i-1,j,k}-2{u}_{i,j,k}+{u}_{i+1,j,k}}{\Delta x^2}+\frac{{u}_{i,j-1,k}-2{u}_{i,j,k}+{u}_{i,j+1,k}}{\Delta y^2}+\frac{{u}_{i,j,k-1}-2{u}_{i,j,k}+{u}_{i,j,k+1}}{\Delta z^2}\approx\Delta {u}(x_i,y_j,z_k,t).
\end{equation*}
We consider the generic initial value problem (analogue to \eqref{genODE}),
\begin{equation}\label{genODE3d}
\frac{\mathrm{d}\mathbf{u}(t)}{\mathrm{d} t}=\mathcal{M}\mathbf{u}(t)+F(\mathbf{u}(t)),\qquad \mathbf{u}(0)=\mathbf{u}_0.
\end{equation}
Here $\mathbf{u}(t)\in \mathbb{R}^{m_xm_ym_z}$, and
\begin{equation*}
\mathcal{M}=I_{m_z}\otimes I_{m_y}\otimes \mathcal{M}_x+I_{m_z}\otimes \mathcal{M}_y\otimes I_{m_x} + \mathcal{M}_z\otimes I_{m_y}\otimes I_{m_x}\in\mathbb{R}^{m_xm_ym_z\,\times\, m_xm_ym_z},
\end{equation*}
with $\mathcal{M}_z$ defined as in \eqref{MD} or \eqref{MN}. Let $\mathcal{U}(t)\in \mathbb{R}^{m_x\,\times\, m_y\,\times\, m_z}$ be the tensor such that $\mathbf{u}(t)=\text{vec}(\boldsymbol{\mathcal{U}}).$

Applying either the IMEX Euler \eqref{genIMEX} or IMEX 2SBDF \eqref{2SBDF} (after a shift in the time index), yields to a linear system in the form
$$\left(aI_{m_xm_ym_z}+b \mathcal{M}\right)\mathbf{u}^n=\mathbf{g},$$
where $\mathbf{g}$ is a known vector computable in terms of $\mathbf{u}$ at the previous steps. Let $\boldsymbol{\mathcal{G}}$ be the tensor such that $\mathbf{g}=\text{vec}(\boldsymbol{\mathcal{G}}),$ and let $\boldsymbol{\mathcal{U}}_{(3)},\boldsymbol{\mathcal{G}}_{(3)}\in \mathbb{R}^{m_z\,\times\, m_x m_y}$ be the mode-3 unfolding of tensors $\boldsymbol{\mathcal{U}}$ and $\boldsymbol{\mathcal{G}}$, respectively. We set $X=\boldsymbol{\mathcal{U}}_{(3)}^T$ and $G=\boldsymbol{\mathcal{G}}_{(3)}^T$. Then, \eqref{genODE3d} is equivalent to
\begin{equation}\label{eq:3dstep1}
(aI_{m_x}\otimes I_{m_y}+b I_{m_y}\otimes \mathcal{M}_x+b \mathcal{M}_y \otimes I_{m_x})X+bX\mathcal{M}_z^T=G.
\end{equation}
Consider the spectral factorization $$\mathcal{M}_z=\Gamma_z\Lambda_z\Gamma_z^{-1},$$
and let $Y=X\Gamma_z^{-T}$. Then \eqref{eq:3dstep1} amounts to
$$
(aI_{m_x}\otimes I_{m_y}+b I_{m_y}\otimes \mathcal{M}_x + \mathcal{M}_y\otimes I_{m_x})Y+bY\Lambda_z=G\Gamma_z^{-T}.
$$
Let $\mathbf{z}_j$ and $\mathbf{h}_j$ be the $j-$th column of $Y$ and $G\Gamma_z^{-T},$ respectively. They satisfy the linear systems
\begin{equation}\label{eq:3dstep2}
(aI_{m_x}\otimes I_{m_y}+b I_{m_y}\otimes \mathcal{M}_x + \mathcal{M}_y\otimes I_{m_x})\mathbf{z}_j+b\lambda_j(I_{m_x}\otimes I_{m_y}) \mathbf{z}_j=\mathbf{h}_j,\qquad j=1,\ldots,m_z.
\end{equation}
In general, these systems are still too large, but each of them can be written in Sylvester form. In fact, let be $Z_j$ and $H_j$ the matrices of dimensions $m_x \times m_y$ such that $\mathbf{z}_j=\text{vec}(Z_j)$ and $\mathbf{h}_j=\text{vec}(H_j)$, respectively. Then, the systems \eqref{eq:3dstep2} are equivalent to
\begin{equation*}
\left((1-b\lambda_j)I_{m_x}+b\mathcal{M}_x\right)Z_j+bZ_j\mathcal{M}_y^T=H_j, \qquad j=1,\ldots,m_z,
\end{equation*}
that are in the form \eqref{genSylv}, and can be solved as in Section \ref{sec:MatOriApp}.

\section{IMEX methods on domains with holes}\label{sec:holes}
So far we have considered problem \eqref{AC}--\eqref{hg} for functions $\phi$ and $c$ defined on a simply connected rectangular domain. This assumption is crucial for having a discrete Laplace operator $\mathcal{M}$ with the structure of a Kronecker sum \eqref{kronsum}, which underpins the efficient implementation in Sections \ref{sec:MatOriApp}--\ref{sec:3drec}. However, this assumption is quite restraining for problems arising in real applications. In fact, pitting corrosion is a local phenomenon that causes the formation of holes on the metallic surface where the protective film has broken. Therefore, here we consider a rectangular domain with some internal holes where a corrosive electrolyte reacts with the metal. For simplicity, we limit the discussion to the two-dimensional case. However, the extension of the proposed technique to three-dimensional domains can be achieved straightforwardly by adapting it according to the arguments presented in Section \ref{sec:3drec}. Thus, the domain under consideration is 
$$\widehat \Omega=[0,L_x]\times[0,L_y]\setminus \Theta,$$
where $\Theta$ is the union set of all holes in the rectangle $[0,L_x]\times[0,L_y]$. 
Generic initial conditions
$$\phi(x,y,0)=\phi_0(x,y), \qquad c(x,y,0)=c_0(x,y),$$
are assigned in $\widehat{\Omega}$.
Denoting with $\partial S$ the boundary of a bounded set $S$, the setting above is modeled by equipping problem \eqref{AC}--\eqref{hg} with the following boundary conditions
\begin{align}\label{BCholesDir}
\phi(x,y,t)=&\,0,\qquad c(x,y,t)=0,\qquad \text{if}\,\,(x,y)\in\partial \Theta,
\end{align}
and
\begin{equation}\label{BCholesDir2}
\phi(x,y,t) = 1,\qquad c(x,y,t) = 1, \qquad \text{if}\,\,(x,y)\in\partial \widehat\Omega\setminus \partial \Theta ,\\
\end{equation}
or  
\begin{equation}\label{BCholesNeu}
\frac{\partial}{\partial n} \phi(x,y,t) = 0,\qquad \frac{\partial}{\partial n} c(x,y,t) = 0, \qquad \text{if}\,\,(x,y)\in\partial \widehat\Omega\setminus \partial \Theta ,\\
\end{equation}
where $\partial z/\partial n$ denotes the normal derivative of function $z$. The Dirichlet boundary conditions \eqref{BCholesDir2} can be used when we expect that the corrosion interface will be far from $\partial \widehat\Omega\setminus \partial \Theta$ for all $t\in[0,T].$ Otherwise, the Neumann boundary conditions \eqref{BCholesNeu} are to be preferred.

A standard space discretization of the Laplacian operator on the domain $\widehat{\Omega}$ leads to systems of ODEs in the form \eqref{genODE}, but where matrix $\mathcal{M}$ is not a Kronecker sum. Thus, the efficient implementation in Section \ref{sec:MatOriApp}, based on the equivalent formulation of the problem as a Sylvester system, is not applicable.

For this reason, before discretizing the problem, we extend it to the simply connected rectangular domain $\Omega=[0,L_x]\times [0,L_y]$ as follows:
\begin{align}\label{extAC}
\phi_t=&\,\chi_{\widehat{\Omega}}(x,y)[D_\phi \Delta\phi+F_1(\phi,c)],\\\label{extCH}
c_t=&\, \chi_{\widehat{\Omega}}(x,y)D_c\Delta(c+F_2(\phi)),
\end{align}
where $\chi_{\widehat{\Omega}}$ is the indicator function of set $\widehat{\Omega}$, defined as
\begin{equation}\label{chi}
\chi_{\widehat{\Omega}}(x,y)=\left\{\begin{array}{c} 1,\qquad \text{if}\,\, (x,y)\in S,\\
0,\qquad \text{if}\,\, (x,y)\notin S.
\end{array}\right.
\end{equation}
The extended problem \eqref{extAC}--\eqref{extCH} is equipped with initial conditions
\begin{align}\label{initmet}
\phi(x,y,0)=&\,\phi_0(x,y), \quad c(x,y,0)=c_0(x,y),\qquad (x,y)\in \widehat{\Omega},\\\label{inithol}
\phi(x,y,0)=&\,0, \quad c(x,y,0)=0,\qquad (x,y)\in \Theta,
\end{align}
and boundary conditions
\begin{equation}\label{extbcD}
\phi(x,y,t) = 1,\qquad c(x,y,t) = 1, \qquad \text{if}\,\,(x,y)\in\partial \Omega,\\
\end{equation}
or
\begin{equation}\label{extbcN}
\frac{\partial}{\partial n} \psi(x,y,t) = 0, \qquad \frac{\partial}{\partial n} c(x,y,t) = 0, \qquad (x,y)\in \partial \Omega,\quad t\in[0,T],
\end{equation}
equivalent to \eqref{BCholesDir2} or \eqref{BCholesNeu}, respectively. 

The extension of the original problem to the rectangular domain $\Omega$ as done above, is coherent with the physics of the model under consideration. In fact, if $(x,y)\in \widehat{\Omega}$, problem \eqref{extAC}--\eqref{chi} coincides with the original problem \eqref{AC}--\eqref{CH}. Otherwise, equations \eqref{extAC}--\eqref{extCH}, with \eqref{chi} and the artificial initial conditions \eqref{inithol}, yield
\begin{equation}\label{vanish}
\phi(x,y,t)\equiv 0,\qquad c(x,y,t)\equiv 0, \qquad (x,y)\in\Theta,\quad t\in[0,T].
\end{equation} 
Requiring \eqref{vanish} is totally physical, meaning that the holes have been completely corroded and no material can be generated there at any time.
In particular, the Dirichlet boundary conditions \eqref{BCholesDir} are naturally satisfied by the solutions of \eqref{extAC}--\eqref{inithol}.
\subsection{Space discretization} Throughout this section we use the same notation previously introduced. Discretization in space is still performed by second order centered differences.  We assume that the space grid is fine enough to have at least one node in $\Theta$ or in each of its connected components if disconnected. The space discretization must be such that 
$$\frac{\mathrm{d}}{\mathrm{d} t}\phi(x_i,y_j,t)=0, \qquad \frac{\mathrm{d}}{\mathrm{d} t}c(x_i,y_j,t)=0, \qquad (x_i,y_j)\in\Theta,\quad t\in[0,T],$$
so that \eqref{vanish} is satisfied at the grid points in $\Theta$. We look then at semidiscretizations in the form
\begin{align}\label{L1}
\frac{\mathrm{d}}{\mathrm{d}t}\phi(x_i,y_j,t)=\widehat{\mathcal{L}}_1(x_i,y_j,t),\qquad \frac{\mathrm{d}}{\mathrm{d}t}c(x_i,y_j,t)=\widehat{\mathcal{L}}_2(x_i,y_j,t),
\end{align}
with initial conditions
\begin{align}\label{initsysom}
\phi_{i,j}(0)=&\,\phi_0(x_i,y_j),\quad c_{i,j}(0)=c_0(x_i,y_j),\quad (x_i,y_j)\in\widehat{\Omega},\\\label{initsysthet}
\phi_{i,j}(0)=&\,0,\quad c_{i,j}(0)=0,\quad (x_i,y_j)\in{\Theta}.
\end{align}
The discrete operators $\widehat{\mathcal{L}}_1$ and $\widehat{\mathcal{L}}_2$ must be such that
\begin{align}\label{L11}
\widehat{\mathcal{L}}_1(x_i,y_j,t)\approx &\,D_\phi {\Delta}\phi(x_i,y_j,t)+F_1(\phi(x_i,y_j,t),c(x_i,y_j,t)),\quad \text{if}\,\,(x_i,y_j)\in\widehat{\Omega},\\\label{L21}
\widehat{\mathcal{L}}_2(x_i,y_j,t)\approx&\, D_c \Delta[c(x_i,y_j,t)+F_2(\phi(x_i,y_j,t))],\quad \text{if}\,\,(x_i,y_j)\in\widehat{\Omega},\\\label{L1L2}
\widehat{\mathcal{L}}_1(x_i,y_j,t)=&\,\widehat{\mathcal{L}}_2(x_i,y_j,t)=0,\quad \text{if}\,\,(x_i,y_j)\in{\Theta}.
\end{align}
\begin{remark}\label{remcons}
Note that while conditions \eqref{L11}--\eqref{L21} are necessary for consistency of the operators, conditions \eqref{L1L2} imply that the right hand side of equations \eqref{L1} vanishes on points of $\Theta$. Together with the initial conditions \eqref{initsysthet}, this yields
\begin{equation}\label{phic0}
\phi_{i,j}(t)=0,\qquad c_{i,j}(t)=0,\qquad (x_i,y_j)\in{\Theta},\qquad t\in[0,T].
\end{equation}
Thus, property \eqref{vanish} of the continuous problem \eqref{extAC}--\eqref{inithol} is preserved by a such space discretization.
\end{remark}
Since $F_1(0,0)=F_2(0)=0$, operators in the form
\begin{align*}
\widehat{\mathcal{L}}_1(x_i,y_j,t)=D_\phi \widehat{\Delta}\phi_{i,j}(t)+\chi_{\widehat{\Omega}}(x_i,y_j)F_1(\phi_{i,j}(t),c_{i,j}(t)),\qquad
\widehat{\mathcal{L}}_2(x_i,y_j,t)=D_c \widehat{\Delta}[c_{i,j}(t)+F_2(\phi_{i,j}(t))],
\end{align*}
satisfy conditions \eqref{L11}--\eqref{L1L2}, provided that the operator $\widehat{\Delta}$ is such that
\begin{align}\label{Dhat1}
\widehat{\Delta} u_{i,j}(t)\approx &\, \Delta u(x_i,y_j,t),\quad (x_i,y_j)\in\widehat{\Omega},\quad t\geq 0,\\\label{Dhat2}
 \widehat{\Delta} u_{i,j}(t) =&\, 0,\quad (x_i,y_j)\in\Theta,\quad t\geq 0,
\end{align}
with $ u\in\{\phi,c,F_2(\phi)\}.$ We define such operator as follows,
\begin{equation}\label{Delhat}
\widehat{\Delta}{u}_{i,j}:=\left\{\begin{array}{lr} 
\Delta u_{i,j}-\Delta (\chi_{\Theta}(x_i,y_j)u_{i,j}),& \text{if}\,\,(x_i,y_j)\in\widehat{\Omega},\\
\Delta u_{i,j}-\Delta (\chi_{\widehat{\Omega}}(x_i,y_j)u_{i,j})-\Delta (\chi_{\Theta}(x_i,y_j)u_{i,j}),& \text{if}\,\,(x_i,y_j)\in\Theta,\\
\end{array}\right.
\end{equation}
where $\Delta$ is the finite difference operator in \eqref{centdiff}.
Observe that:
\begin{itemize}
\item Assuming $u_{i,j}=0$ in $\Theta$, the term $\Delta (\chi_{\Theta}(x_i,y_j)u_{i,j})=0$, and could be disregarded.
\item The expression $\Delta u_{i,j}-\Delta (\chi_{\widehat{\Omega}}(x_i,y_j)u_{i,j})-\Delta (\chi_{\Theta}(x_i,y_j)u_{i,j})$ amounts to zero. 
\end{itemize}
\begin{figure}[t!]
\begin{center}
\begin{tikzpicture}[scale=0.8]
\draw [semithick,domain=90:180,fill=gray!50] plot ({3.4*(cos(\x))+5}, {3.4*sin(\x)}) |- (5,0);
\draw [ultra thick,domain=90:180] plot ({3.4*(cos(\x))+5}, {3.4*sin(\x)});
\draw[semithick] (0,0) grid (5,4);
\fill (2,2) circle(.1) node[label=below right:{\!\!\!\!\scriptsize{$(x_i,y_j)$}},draw]{};
\foreach \p in {(2,1),(2,3),(1,2),(3,2)}
\fill \p circle(.1);
\begin{scope}[color=black]
    \node[anchor = center] () at (0.5,3.5){$\widehat{\Omega}$};
    \node[anchor = center] () at (4.5,0.5){$\Theta$};
    \node[anchor = center] () at (2.5,-0.5){(a)};
\end{scope}
\end{tikzpicture}
\qquad
\begin{tikzpicture}[scale=0.8]
\draw [semithick,domain=90:180,fill=gray!50] plot ({4.2*(cos(\x))+5}, {3.7*sin(\x)}) |- (5,0);
\draw [ultra thick,domain=90:180] plot ({4.2*(cos(\x))+5}, {3.7*sin(\x)});
\draw[semithick] (0,0) grid (5,4);
\fill (2,2) circle(.1) node[label=below right:{\!\!\!\!\scriptsize{$(x_i,y_j)$}},draw]{};
\foreach \p in {(2,1),(2,3),(1,2),(3,2)}
\fill \p circle(.1);
\begin{scope}[color=black]
    \node[anchor = center] () at (0.5,3.5){$\widehat{\Omega}$};
    \node[anchor = center] () at (4.5,0.5){$\Theta$};
    \node[anchor = center] () at (2.5,-0.5){(b)};
\end{scope}
\end{tikzpicture}
\caption{Stencil configurations for discrete Laplacian. The shaded region belongs to set $\Theta$, the white region to $\widehat{\Omega}$. (a): Stencil with center in $\widehat{\Omega}$ and points in $\Theta$. (b): Stencil with center  in $\Theta$ and points in $\widehat{\Omega}$.}\label{stenc}
\end{center}
\end{figure}
So, if $u_{i,j}=0$ in $\Theta$, the operator $\widehat{\Delta}$ could be equivalently rewritten as
$$
\widehat{\Delta}{u}_{i,j}=\left\{\begin{array}{lr} 
\Delta u_{i,j}\approx \Delta u(x_i,y_j,t),& \text{if}\,\,(x_i,y_j)\in\widehat{\Omega},\\
0,& \text{if}\,\,(x_i,y_j)\in\Theta,\\
\end{array}\right.
$$
and \eqref{Dhat1} and \eqref{Dhat2} both follow immediately. However, the expressions in \eqref{Delhat} are more convenient since:
\begin{itemize}
\item[i)] In practice, for example due to round-offs (or as the result of an inexact time integration), usually $u_{i,j}$ is small but not exactly zero for $(x_i,y_j)\in\Theta$. Thus, when $(x_i,y_j)\in\widehat{\Omega}$ has neighbors in $\Theta$ (see Figure \ref{stenc}(a)), these spurious quantities contribute to $\Delta u_{i,j}$. Adding the (small but non-zero) compensation term $-\Delta (\chi_{\Theta}(x_i,y_j)u_{i,j})$ allows to annihilate them in $\widehat{\Delta}u_{i,j}$.
\item[ii)] They provide a splitting between terms that in the vector formulation of the method contribute to build a Kronecker sum (those contributing to $\Delta u_{i,j}$) and the others (the compensation terms).
\item[iii)] If $(x_i,y_j)\in\Theta$ has neighbors in $\widehat{\Omega}$ (see Figure \ref{stenc}(b)), definition \eqref{Delhat} identifies two separate terms of compensation having different nature:  $-\Delta (\chi_{\widehat{\Omega}}(x_i,y_j)u_{i,j})$ and  $-\Delta (\chi_\Theta(x_i,y_j)u_{i,j})$. While, as before, the latter compensates small perturbations, the former may be large.
In fact, the values of $u$ at the neighbors in $\widehat{\Omega}$ give non zero and possibly large contributions to $\Delta u_{i,j}$.  If these are not annihilated, condition \eqref{Dhat2} does not hold true (even approximately), resulting in a severe violation of \eqref{phic0}.
\end{itemize}
A remark similar to i) holds also for the term $\chi_{\widehat{\Omega}}(x_i,y_j)F_1(\phi_{i,j}(t),c_{i,j}(t))$ in $\widehat{\mathcal{L}}_1$ where the multiplication by the indicator function is superfluous if $\phi_{i,j}(t)=c_{i,j}(t)=0$ exactly on $\Omega$.

The action of $\widehat{\Delta}$ in \eqref{Delhat} on a vector $\mathbf{u}$ with entries $u_{i,j}$ is given by
\begin{equation}\label{delhat}
\widehat{\Delta}\mathbf{u}=(\mathcal{M}-\mathcal{N}_1-\mathcal{N}_2)\mathbf{u},
\end{equation}
where matrix $\mathcal{M}$ is the Kronecker sum in \eqref{kronsum} forcing the boundary conditions \eqref{extbcD} or \eqref{extbcN}, and matrices $\mathcal{N}_1$ and $\mathcal{N}_2$ are sparse matrices with nonzero entries introducing the correction terms $-\Delta (\chi_{\widehat{\Omega}}(x_i,y_j)u_{i,j})$ and $-\Delta (\chi_{\Theta}(x_i,y_j)u_{i,j})$, respectively, according to \eqref{Delhat}. We stress that while the contributions given by $\mathcal{N}_1\mathbf{u}$ may be large and fundamental to have \eqref{phic0}, the correction $\mathcal{N}_2\mathbf{u}$ is small or even zero if the variables exactly vanish on $\Theta$.

Thus, the system of ODEs \eqref{L1} for vectors $\boldsymbol{\phi}$ and $\mathbf{c}$ amounts to
\begin{align}\label{ACodecirc}
\frac{\mathrm{d}\boldsymbol{\phi}}{\mathrm{d}t}=&\,D_\phi (\mathcal{M}-\mathcal{N}_1-\mathcal{N}_2)\boldsymbol{\phi}+F_1(\mathbf{x},\boldsymbol{\phi},\mathbf{c}),\\\label{CHodecirc}
\frac{\mathrm{d}\mathbf{c}}{\mathrm{d}t}=&\, D_c(\mathcal{M}-\mathcal{N}_1-\mathcal{N}_2)(\mathbf{c}+F_2(\boldsymbol{\phi})),
\end{align}
with initial conditions \eqref{initsysom}--\eqref{initsysthet}. In \eqref{ACodecirc}, $\mathbf{x}$ is a vector with the node coordinates lexicographically ordered. $F_1(\mathbf{x},\boldsymbol{\phi},\mathbf{c})$ and $F_2(\boldsymbol{\phi})$ are vectors with entries $\chi_{\widehat{\Omega}}(x_i,y_j)F_1(\phi_{i,j},c_{i,j})$ and $F_2(\phi_{i,j})$, respectively, sorted analogously.

\begin{remark}
As a consequence of Remark \ref{remcons}, the solutions of  \eqref{ACodecirc}--\eqref{CHodecirc} with \eqref{initsysom}--\eqref{initsysthet} satisfy \eqref{phic0}. In particular, the homogeneous Dirichlet boundary conditions \eqref{BCholesDir} of the original problem are naturally satisfied by the solutions of the initial value problem. 
\end{remark}
\subsection{Time integration}\label{sec:TIholes}
Let $\phi_{i,j}^n$ and $c_{i,j}^n$ be the approximation to $\phi(x_i,y_j,t_n)$ and $c(x_i,y_j,t_n)$, respectively.
In order to preserve property \eqref{vanish} we require that
\begin{equation}\label{tdzero}
\phi_{i,j}^n=c_{i,j}^n=0,\qquad (x_i,y_j)\in\Theta,\quad n\geq 0.
\end{equation}
This is achieved if the time discretization of operators $\widehat{\mathcal{L}}_1$ and $\widehat{\mathcal{L}}_2$ preserve their properties \eqref{L11}--\eqref{L1L2}, i.e., denoting them as $\widetilde{\mathcal{L}}_1$ and $\widetilde{\mathcal{L}}_2$, respectively,
\begin{align*}
\widetilde{\mathcal{L}}_1(x_i,y_j,t_n)\approx &\,D_\phi {\Delta}\phi(x_i,y_j,t_n)+F_1(\phi(x_i,y_j,t_n),c(x_i,y_j,t_n)),\quad \text{if}\,\,(x_i,y_j)\in\widehat{\Omega},\\
\widetilde{\mathcal{L}}_2(x_i,y_j,t_n)\approx&\, D_c \Delta[c(x_i,y_j,t_n)+F_2(\phi(x_i,y_j,t_n))],\quad \text{if}\,\,(x_i,y_j)\in\widehat{\Omega},\\
\widetilde{\mathcal{L}}_1(x_i,y_j,t_n)=&\,\widetilde{\mathcal{L}}_2(x_i,y_j,t_n)=0,\quad \text{if}\,\,(x_i,y_j)\in{\Theta},
\end{align*}
for all $n\geq 0$. Hence, we define
\begin{align*}
\widetilde{\mathcal{L}}_1(x_i,y_j,t_n)=D_\phi \widetilde{\Delta}\phi_{i,j}^n+\chi_{\widehat{\Omega}}(x_i,y_j)F_1(\phi_{i,j}^n,c_{i,j}^n),\qquad \widetilde{\mathcal{L}}_2(x_i,y_j,t_n)=D_c \widetilde{\Delta}[c_{i,j}^n+F_2(\phi_{i,j}^n)],
\end{align*}
where the operator $\widetilde{\Delta}$ applied to a totally discrete variable $u_{i,j}^n$ is such that
$$\widetilde{\Delta}u_{i,j}^n\approx \widehat{\Delta}u_{i,j}(t_n).$$
Depending on the definition of $\widetilde{\Delta}$ we obtain and compare different IMEX methods. 
\subsubsection{First-order accuracy}\label{sec:fo}
As for the case of a rectangular domain, we relax the stiff equation \eqref{ACodecirc} before the time discretization. A straightforward application of the IMEX Euler method in Section \ref{sec:FirstOrder} to problem \eqref{ACodecirc}--\eqref{CHodecirc} leads to the scheme:
\begin{align}\label{EulerHolesAC}
(I_{m_xm_y}-\Delta t(D_\phi\mathcal{M}-wI_{m_xm_y}))\boldsymbol{\phi}^{n+1}&=\boldsymbol{\phi}^n+\Delta t[w\boldsymbol{\phi}^n+ F_1(\mathbf{x},\boldsymbol{\phi}^n,\mathbf{c}^n)- D_\phi(\mathcal{N}_1+\mathcal{N}_2)\boldsymbol{\phi}^n],\\\label{EulerHolesCH}
(I_{m_xm_y}-\Delta tD_c\mathcal{M})\mathbf{c}^{n+1}&=\mathbf{c}^n+\Delta tD_c (\mathcal{M}-\mathcal{N}_1-\mathcal{N}_2)F_2(\boldsymbol{\phi}^{n+1})-\Delta t D_c\,(\mathcal{N}_1+\mathcal{N}_2)\mathbf{c}^n,
\end{align}
and the corresponding operator $\widetilde{\Delta}$ is defined as
\begin{equation}\label{DeltaEuler}
\widetilde{\Delta} \mathbf{u}^n:=\,\mathcal{M}\mathbf{u}^n-(\mathcal{N}_1+\mathcal{N}_2)\mathbf{u}^{n-1}.
\end{equation}
We observe that, since the coefficient matrices in \eqref{EulerHolesAC}--\eqref{EulerHolesCH} are Kronecker sums, this scheme can be reformulated as a system of Sylvester equations that can be efficiently solved by the strategy in Section \ref{sec:MatOriApp}. However, this method is not really effective in practice. In fact, in order to exactly impose the boundary conditions \eqref{BCholesDir} and have \eqref{phic0}, it is crucial that the time discretization preserves \eqref{Dhat2}, i.e.,
\begin{equation}\label{tildel0}
\widetilde{\Delta}u_{i,j}^n=0,\qquad (x_i,y_j)\in\Theta, \quad n\geq 0.
\end{equation}
Unfortunately, operator $\widetilde{\Delta}$ in \eqref{DeltaEuler} does not satisfy this property. We show this fact referring to the point $(x_i,y_j)$ in Figure \ref{stenc}(b), for simplicity. Even assuming that $u_{i,j}^n=0$ for all $(x_i,y_j)\in\Theta$ and all $n$, we have
\begin{align*}
\widetilde{\Delta}u_{i,j}^n=\frac{{u}_{i-1,j}^n-2{u}_{i,j}^n+{u}_{i+1,j}^n}{\Delta x^2}+\frac{{u}_{i,j-1}^n-2{u}_{i,j}^n+{u}_{i,j+1}^n}{\Delta y^2}-\frac{{u}_{i-1,j}^{n-1}}{\Delta x^2}-\frac{{u}_{i,j+1}^{n-1}}{\Delta y^2}=\frac{{u}_{i-1,j}^n-{u}_{i-1,j}^{n-1}}{\Delta x^2}-\frac{{u}_{i,j+1}^n-{u}_{i,j+1}^{n-1}}{\Delta y^2},
\end{align*}
that is nonzero, in general. This quantity is indeed $\mathcal{O}(\Delta t/\Delta x^2)+\mathcal{O}(\Delta t/\Delta y^2)$. A similar result holds for any point of $\Theta$ with neighbors in $\widehat{\Omega}$. We can expect that \eqref{tdzero} is satisfied, at least approximately, when these quantities are negligible. In principle, this happens if $D\Delta t\ll  h^2=\min\{\Delta x^2,\Delta y^2\}$, where $D$ is the diffusion coefficient, that is a too demanding restriction to satisfy.

Thus, we propose here two different versions of the IMEX Euler method by providing two different definitions of the operator $\widetilde{\Delta}$. These are defined by their action on vectors as
\begin{align*}
\widetilde{\Delta}^{I} \mathbf{u}^n:=(\mathcal{M}-\mathcal{N}_1-\mathcal{N}_2)\mathbf{u}^n,\qquad \widetilde{\Delta}^{E} \mathbf{u}^n:=(\mathcal{M}-\mathcal{N}_1)\mathbf{u}^n-\mathcal{N}_2\mathbf{u}^{n-1}.
\end{align*}
Superscripts $I$ and $E$ denote that the terms multiplying matrix $\mathcal{N}_2$ are treated implicitly or explicitly, respectively.

In vector form, the schemes obtained by these two operators are, respectively:\\\\
\textbf{IMEX-I Euler}
\begin{align*}
[I_{m_xm_y}-\Delta t(D_\phi(\mathcal{M}-\mathcal{N}_1-\mathcal{N}_2)-wI_{m_xm_y})]\boldsymbol{\phi}^{n+1}=&\, \boldsymbol{\phi}^n+\Delta t [w\boldsymbol{\phi}^n+F_1(\mathbf{x},\boldsymbol{\phi}^n,\mathbf{c}^n)],\\
[I_{m_xm_y}-\Delta tD_c(\mathcal{M}-\mathcal{N}_1-\mathcal{N}_2)]\mathbf{c}^{n+1}=&\,\mathbf{c}^n+\Delta tD_c (\mathcal{M}-\mathcal{N}_1-\mathcal{N}_2)F_2(\boldsymbol{\phi}^{n+1}),
\end{align*}
\textbf{IMEX-E Euler}
\begin{align*}
[I_{m_xm_y}-\Delta t(D_\phi(\mathcal{M}-\mathcal{N}_1)-wI_{m_xm_y})]\boldsymbol{\phi}^{n+1}=&\,\boldsymbol{\phi}^n+\Delta t [w\boldsymbol{\phi}^n+F_1(\mathbf{x},\boldsymbol{\phi}^n,\mathbf{c}^n)- D_\phi\mathcal{N}_2\boldsymbol{\phi}^n],\\
[I_{m_xm_y}-\Delta tD_c(\mathcal{M}-\mathcal{N}_1)]\mathbf{c}^{n+1}=&\,\mathbf{c}^n+\Delta tD_c (\mathcal{M}-\mathcal{N}_1-\mathcal{N}_2)F_2(\boldsymbol{\phi}^{n+1})-\Delta t D_c\mathcal{N}_2\mathbf{c}^n,
\end{align*}
with $w$ such that \eqref{fov} holds true at each time step with $\mathcal{D}$ as in \eqref{Dord1}, and
\begin{equation}\label{ABII}
\mathcal{A}_w=D_\phi(\mathcal{M}-\mathcal{N}_1-\mathcal{N}_2)-wI_{m_xm_y},\quad \mathcal{B}_w=\mathcal{J}_{F_1}+wI_{m_xm_y},
\end{equation}
for IMEX-I Euler,
\begin{equation}\label{ABIE}
\mathcal{A}_w=D_\phi(\mathcal{M}-\mathcal{N}_1)-wI_{m_xm_y},\quad\mathcal{B}_w=\mathcal{J}_{F_1}+wI_{m_xm_y}-D_\phi\mathcal{N}_2,
\end{equation}
for IMEX-E Euler.

We observe that since both matrices $\mathcal{N}_1+\mathcal{N}_2$ and $\mathcal{N}_1$ do not have the structure in \eqref{kronsum}, the coefficient matrices in IMEX-I and IMEX-E are not Kronecker sums. Hence, the efficient implementation in Section \ref{sec:MatOriApp} can not be applied to these methods.

Solving their vector form requires the solution of two linear systems of dimension $m_xm_y$ at each time step. This can be computationally burdensome even if the space grid is not too fine and the domain not too large.

We then look for a technique to speed up the solution of methods in this form. Before presenting such technique, we extend the arguments to second-order accurate schemes.
\subsubsection{Second-order accuracy}\label{sec:so}
Following similar arguments as in the previous section, we  define the operators
\begin{align*}
\widetilde{\Delta}^{I} \mathbf{u}^n:=(\mathcal{M}-\mathcal{N}_1-\mathcal{N}_2)\mathbf{u}^n,\qquad \widetilde{\Delta}^{E} \mathbf{u}^n:=(\mathcal{M}-\mathcal{N}_1)\mathbf{u}^n-2\mathcal{N}_2\mathbf{u}^{n-1}+\mathcal{N}_2\mathbf{u}^{n-2},
\end{align*}
where the same notation as above is used.

By relaxing the stiff equation \eqref{ACodecirc} as before, the schemes for system \eqref{ACodecirc}--\eqref{CHodecirc} obtained in this way are, respectively:\\\\
\textbf{IMEX-I 2SBDF}
\begin{align*}
[3I_{m_xm_y}\!-2\Delta t[D_\phi(\mathcal{M}-\mathcal{N}_1-\mathcal{N}_2)-wI_{m_xm_y}]]\boldsymbol{\phi}^{n+2}=&\,4 \boldsymbol{\phi}^{n+1}\!-\boldsymbol{\phi}^{n}+2\Delta t [2F_1(\mathbf{x},\boldsymbol{\phi}^{n+1},\mathbf{c}^n)+2w\boldsymbol{\phi}^{n+1}\!-\!F_1(\mathbf{x},\boldsymbol{\phi}^{n},\mathbf{c}^n)-w\boldsymbol{\phi}^{n}],\\
[3I_{m_xm_y}-2\Delta tD_c(\mathcal{M}-\mathcal{N}_1-\mathcal{N}_2)]\mathbf{c}^{n+2}=&\,4\mathbf{c}^{n+1}-\mathbf{c}^n+2\Delta tD_c (\mathcal{M}-\mathcal{N}_1-\mathcal{N}_2)F_2(\boldsymbol{\phi}^{n+2}),
\end{align*}
\textbf{IMEX-E 2SBDF}
\begin{align*}
[3I_{m_xm_y}-2\Delta t[D_\phi(\mathcal{M}-\mathcal{N}_1)-wI_{m_xm_y}]]\boldsymbol{\phi}^{n+2}=&\,4\boldsymbol{\phi}^{n+1}-\boldsymbol{\phi}^n+2\Delta t [2F_1(\mathbf{x},\boldsymbol{\phi}^{n+1},\mathbf{c}^n)\\
&+2w\boldsymbol{\phi}^{n+1}- 2D_\phi\mathcal{N}_2\boldsymbol{\phi}^{n+1}-F_1(\mathbf{x},\boldsymbol{\phi}^{n},\mathbf{c}^n)-w\boldsymbol{\phi}^{n}+ D_\phi\mathcal{N}_2\boldsymbol{\phi}^n],\\
(3I_{m_xm_y}-2\Delta tD_c(\mathcal{M}-\mathcal{N}_1))\mathbf{c}^{n+2}=&\,4\mathbf{c}^{n+1}-\mathbf{c}^n+2\Delta tD_c (\mathcal{M}-\mathcal{N}_1-\mathcal{N}_2)F_2(\boldsymbol{\phi}^{n+2})-2\Delta tD_c\,\mathcal{N}_2 (2\mathbf{c}^{n+1}-\mathbf{c}^{n}).
\end{align*}
In the two methods above, parameter $w$ satisfies at each time step condition \eqref{fov} with $\mathcal{D}$ in \eqref{Dord2}, and matrices $\mathcal{A}_w$ and $\mathcal{B}_w$ as in \eqref{ABII}--\eqref{ABIE}, respectively.

As before, the efficient implementation in Section \ref{sec:MatOriApp} can not be applied to these methods, as they can not be reformulated in Sylvester form.

\begin{remark}
We do not consider the alternative operators
$$\widetilde{\Delta} \mathbf{u}^n:=(\mathcal{M}-\mathcal{N}_2)\mathbf{u}^n-\mathcal{N}_1\mathbf{u}^{n-1}\qquad \text{and} \qquad \widetilde{\Delta} \mathbf{u}^n:=(\mathcal{M}-\mathcal{N}_2)\mathbf{u}^n-2\mathcal{N}_1\mathbf{u}^{n-1}+\mathcal{N}_1\mathbf{u}^{n-2},$$
to define IMEX methods of order one and two, respectively. In fact, not only the resulting schemes can not be implemented efficiently by the strategy in Section \ref{sec:MatOriApp}, but also these operators do not satisfy \eqref{tildel0} and so they suffer from severe restrictions on the time step to fulfill (approximately) \eqref{tdzero} and the boundary conditions on $\partial\Theta$.

It is easy to show that instead all operators $\widetilde{\Delta}^{I}$ and $\widetilde{\Delta}^{E}$ defined above satisfy property \eqref{tildel0}.
\end{remark}

\section{Iterative matrix-oriented approach on domains with holes}\label{sec:iterative}
For the efficient time integration of system \eqref{ACodecirc}--\eqref{CHodecirc}, we embed an iterative procedure in methods IMEX-I and IMEX-E defined in the previous sections. For the first-order Euler methods we define it as follows:\\\\
\textbf{Iter. IMEX-I Euler and Iter. IMEX-E Euler}
\begin{align}\label{ACiter}
((1+w\Delta t)I_{m_xm_y}\!-\!\Delta tD_\phi\mathcal{M})\boldsymbol{\phi}^{n+1,k+1}&=\boldsymbol{\phi}^n\!+\!\Delta t[w\boldsymbol{\phi}^n+ F_1(\mathbf{x},\boldsymbol{\phi}^n,\mathbf{c}^n)-D_\phi\,(\mathcal{N}\boldsymbol{\phi}^{n+1,k}+\mathcal{G}\boldsymbol{\phi}^{n})],\\\label{CHiter}
(I_{m_xm_y}\!-\!\Delta tD_c\mathcal{M})\mathbf{c}^{n+1,k+1}&=\mathbf{c}^n\!+\!\Delta tD_c [(\mathcal{M}-\mathcal{N}_1-\mathcal{N}_2)F_2(\boldsymbol{\phi}^{n+1})\!-(\mathcal{N}\mathbf{c}^{n+1,k}+\mathcal{G}\mathbf{c}^{n})],
\end{align}
where $\mathcal{N}=\mathcal{N}_1+\mathcal{N}_2$ or $\mathcal{N}=\mathcal{N}_1$, respectively, $\mathcal{G}=\mathcal{N}_1+\mathcal{N}_2-\mathcal{N}$ and $k$ is the iteration number.

The iterative versions of the second-order accurate methods are defined as:\\\\
\textbf{Iter. IMEX-I 2SBDF and Iter. IMEX-E 2SBDF}
\begin{align}\nonumber
[(3+2w\Delta t) I_{m_xm_y}-2\Delta tD_\phi\mathcal{M}]\boldsymbol{\phi}^{n+2,k+1}=&\,-2\Delta tD_\phi(\mathcal{N}\boldsymbol{\phi}^{n+2,k}+\mathcal{G}(2\boldsymbol{\phi}^{n+1}-\boldsymbol{\phi}^{n}))\\\label{ACiter2}
&\,+4 \boldsymbol{\phi}^{n+1}\!-\boldsymbol{\phi}^{n}\!+2\Delta t [2F_1(\mathbf{x},\boldsymbol{\phi}^{n+1},\mathbf{c}^n)\!+\!2w\boldsymbol{\phi}^{n+1}\!-\!F_1(\mathbf{x},\boldsymbol{\phi}^{n},\mathbf{c}^n)-w\boldsymbol{\phi}^{n}],\\\label{CHiter2}
\!\![3I_{m_xm_y}\!-\!2\Delta tD_c\mathcal{M}]\mathbf{c}^{n+2,k+1}=4\mathbf{c}^{n+1}\!-&\,\mathbf{c}^n\!-\!2\Delta tD_c(\mathcal{N}\mathbf{c}^{n+2,k}+\mathcal{G}(2\mathbf{c}^{n+1}\!-\!\mathbf{c}^{n}))\!+\!2\Delta tD_c (\mathcal{M}\!-\!\mathcal{N}_1\!-\!\mathcal{N}_2)F_2(\boldsymbol{\phi}^{n+1}),
\end{align}
with, again, $\mathcal{N}=\mathcal{N}_1+\mathcal{N}_2$ or $\mathcal{N}=\mathcal{N}_1$, respectively, $\mathcal{G}=\mathcal{N}_1+\mathcal{N}_2-\mathcal{N}$ and $k$ is the iteration number.

The iterative procedures above solve two sequences of linear systems at each time step, in contrast to the IMEX-I and IMEX-E methods in Section \ref{sec:TIholes}, which solve only two linear systems per time step.
However, the coefficient matrices in \eqref{ACiter}--\eqref{CHiter} and in \eqref{ACiter2}--\eqref{CHiter2} are all Kronecker sums. Thus, these systems can all be written in Sylvester form and solved efficiently as in Section~\ref{sec:MatOriApp}.

Next, we will analyze under which circumstances the linear systems in \eqref{ACiter}--\eqref{CHiter} and \eqref{ACiter2}--\eqref{CHiter2} are solvable, and the resulting iterative methods converge and how fast.

The linear systems defining the iterative methods in \eqref{ACiter}--\eqref{CHiter2} have a unique solution if their coefficient matrices are non singular. All these matrices are in the form
\begin{equation}\label{matP}
\mathcal{P}= \beta I-\alpha\mathcal{M},\qquad \alpha, \beta>0,
\end{equation}
with $\mathcal{M}$ in \eqref{kronsum} and \eqref{MD} or \eqref{MN}.
The following theorem establishes that the proposed methods are well-defined for any choice of the grid steps and of the boundary conditions.
\begin{theorem}\label{th:Pinv}
Any matrix $\mathcal{P}$ in the form \eqref{matP} is invertible.
\end{theorem}
\begin{proof}
Matrix $\mathcal{M}$ is diagonalizable with non-positive eigenvalues (regardless of the type of boundary conditions). Thus, there exists a matrix $Q$ such that $\mathcal{M}=QDQ^{-1}$ with $D$ the diagonal matrix of the eigenvalues. Hence,
$$\det(\mathcal{P})=\det(\beta I-\alpha QDQ^{-1})=\det(\beta I-\alpha D),$$
that is nonzero because matrix $\beta I -\alpha D$ is diagonal with positive entries.
\end{proof}
In the following two sections we establish when and how fast the proposed iterations converge, considering separately the case of boundary conditions of Dirichlet and Neumann type.
Henceforth, the spectral radius of a matrix $A$ is denoted by $\rho(A)$.
\subsection{Dirichlet Boundary Conditions}\label{sec:Dir}
In this section, we limit the discussion to the case of Dirichlet boundary conditions \eqref{extbcD}. Let us denote with $\mathcal{P}_D$ the symmetric matrix in \eqref{matP} with $\mathcal{M}$ defined in \eqref{kronsum}--\eqref{MD} . We start by providing upper bounds for the spectral radius of the iteration matrices in the iterative IMEX-I methods.

\begin{theorem}\label{theorhosig}
The spectral radius of matrix 
\begin{equation}\label{Sig}
\Sigma_D=-\alpha\mathcal{P}_D^{-1}\mathcal{N},\qquad \mathcal{N}=\mathcal{N}_1+\mathcal{N}_2,
\end{equation}
is such that 
\begin{equation}\label{th2th1}
\rho(\Sigma_D)  <  \frac{4\alpha}{\beta}\left(\frac{1}{\Delta x^2}+\frac{1}{\Delta y^2}\right).
\end{equation}
\end{theorem}
\begin{proof}
The spectral radius of matrix $\Sigma_D$ can be bounded as follows,
$$\rho(\Sigma_D) = \alpha\rho(\mathcal{P}_D^{-1}\mathcal{N})\leq \alpha\|\mathcal{P}_D^{-1}\mathcal{N}\|_2\leq \alpha\|\mathcal{P}_D^{-1}\|_2\|\mathcal{N}\|_2.$$
Since $\mathcal{P}_D$ is symmetric, and using well-known norm inequalities, then
$$\rho(\Sigma_D)\leq \alpha\|\mathcal{P}_D^{-1}\|_2\|\mathcal{N}\|_2= \alpha \rho(\mathcal{P}_D^{-1}) \|\mathcal{N}\|_2\leq \alpha\rho(\mathcal{P}_D^{-1}) \sqrt{\|\mathcal{N}\|_1\|\mathcal{N}\|_{\infty}}.$$
The eigenvalues of matrix $\mathcal{M}$ are known and negative, thus if $\lambda$ is the largest eigenvalue of $\mathcal{M},$ then the smallest eigenvalue of $\mathcal{P}_D$ is 
$$\beta-\alpha \lambda >0,$$
and the spectral radius of $\mathcal{P}_D^{-1}$ is then
\begin{equation}\label{eq:rhoPinv}
\rho(\mathcal{P}_D^{-1})=\frac{1}{\beta-\alpha\lambda}<\frac{1}{\beta}.
\end{equation}
In order to bound the two norms of $\mathcal{N}$ we observe that the $i$-th row has nonzero entries only if some nodes in the stencil centered on $\mathbf{x}_i$ are in $\Theta$. Thus, the number of nonzero entries in a row and the sum of their moduli is maximized when the whole stencil is in $\Theta$, and thus,
$$\|\mathcal{N}\|_\infty\leq 4\left(\frac{1}{\Delta x^2}+\frac{1}{\Delta y^2}\right).$$
We can reason in a similar way on the columns, and so,
$$\|\mathcal{N}\|_1\leq 4\left(\frac{1}{\Delta x^2}+\frac{1}{\Delta y^2}\right).$$
We conclude then that
$$\rho(\Sigma_D)\leq \alpha\rho(\mathcal{P}_D^{-1}) \sqrt{\|\mathcal{N}\|_1\|\mathcal{N}\|_{\infty}} < \frac{4\alpha}{\beta}\left(\frac{1}{\Delta x^2}+\frac{1}{\Delta y^2}\right).$$
\end{proof}
The following corollary establishes an upper bound for the spectral radii of the iteration matrices in \eqref{ACiter}--\eqref{CHiter2} with $\mathcal{N}=\mathcal{N}_1+\mathcal{N}_2$.
\begin{corollary}\label{cor1}
The spectral radii of the iteration matrices of IMEX-I Euler in \eqref{ACiter}--\eqref{CHiter} and of IMEX-I 2SBDF in \eqref{ACiter2}--\eqref{CHiter2} fulfill the following bounds
$$\rho_\phi  <  \frac{4D_\phi\Delta t}{\gamma+w \Delta t}\left(\frac{1}{\Delta x^2}+\frac{1}{\Delta y^2}\right),\qquad \rho_c  <  \frac{4D_c\Delta t}\gamma \left(\frac{1}{\Delta x^2}+\frac{1}{\Delta y^2}\right).$$
where the parameter $\gamma$ is defined by
\begin{equation}\label{eq:gamma}
\gamma=1,\qquad \gamma=\frac{3}2,
\end{equation}
for the first- and second-order method, respectively.
\end{corollary}
\begin{proof}
The iteration matrices of equations \eqref{ACiter}--\eqref{CHiter2} are in the form \eqref{Sig} where $\mathcal{P}_D$ is in the form \eqref{matP}, with $(\alpha,\beta)=(\Delta t D_\phi, 1+w\Delta t)$, $(\alpha,\beta)=(\Delta t D_c,1)$, $(\alpha,\beta)=(2\Delta t D_\phi, 3+2w\Delta t)$, $(\alpha,\beta)=(2\Delta t D_c,3)$, respectively. 
The proof is completed as a consequence of Theorem \ref{theorhosig} by substituting these values of $\alpha$ and $\beta$ in equation \eqref{th2th1}.
\end{proof}
{\color{black}\begin{theorem}\label{theo:conI}
Let $h=\min\{\Delta x,\Delta y\}$. The iterations \eqref{ACiter}-\eqref{CHiter} defining the iterative IMEX-I Euler method and the iterations \eqref{ACiter2}-\eqref{CHiter2} defining the iterative IMEX-I 2SBDF method, converge if
\begin{equation}\label{eq:convas}
\left(h^2>\frac{8D_\phi}{w}, \quad \text{or}\quad \Delta t<\frac{\gamma h^2}{8D_\phi-wh^2}\right), \quad \text{and} \quad \Delta t<\frac{\gamma h^2}{8D_c},
\end{equation}
where the parameter $\gamma$ is defined in \eqref{eq:gamma}.
\end{theorem}
\begin{proof}
From Corollary \ref{cor1}, sufficient conditions for convergence of the iterations in the IMEX-I schemes are
$$\frac{8D_\phi\Delta t}{h^2(\gamma+w \Delta t)}<1,\qquad \frac{8D_c\Delta t}{\gamma h^2}<1.$$ 
The first condition is satisfied if
$$h^2>\frac{8D_\phi}{w} \quad \text{or}\quad \Delta t<\frac{\gamma h^2}{8D_\phi-wh^2},$$
the second if $$\Delta t<\frac{\gamma h^2}{8D_c}.$$
Thus, the statement holds.
\end{proof}
}
\begin{remark}\label{rem:notsharp}
With similar steps as in the proofs of Theorem \ref{theorhosig} and Corollary \ref{cor1}, one can obtain similar bounds on the spectral radius of the iteration matrices in \eqref{ACiter}--\eqref{CHiter2} of the iterative IMEX-E methods. In fact, by reasoning in a similar way on the norms of matrix $\mathcal{N}=\mathcal{N}_1$, it is easy to prove that
$$\rho_\phi \leq \frac{2D_\phi\Delta t}{\gamma+w \Delta t}\left(\frac{1}{\Delta x^2}+\frac{1}{\Delta y^2}\right),\qquad \rho_c \leq \frac{2D_c\Delta t}{\gamma} \left(\frac{1}{\Delta x^2}+\frac{1}{\Delta y^2}\right),$$
with $\gamma$ as in \eqref{eq:gamma}.
However, bounds sharper than these are provided below.
\end{remark}
We now aim at establishing bounds on the spectral radius of the iteration matrices defined by the iterative {IMEX-E} method in \eqref{ACiter}--\eqref{CHiter}, which are sharper than those in Remark \ref{rem:notsharp}. These bounds are derived from the fact that matrix $\mathcal{N}=\mathcal{N}_1$ is nilpotent of order two (see Lemma \ref{lemnil}).\\

\begin{Lemma}[see \citep{Shemesh}]\label{Shemesh}
If two matrices $A$ and $B$ satisfy
$$A[A,B]=[A,B]B=0,\qquad [A,B]=AB-BA,$$
then they are simultaneously triangularizable, i.e. there exists an invertible matrix $Q$ such that $Q^{-1}AQ$ and $Q^{-1}BQ$ are upper triangular.
\end{Lemma}

\begin{theorem}\label{theosig}
Let $\Sigma_D$ be a matrix as in \eqref{Sig} with $\mathcal{N}$ nilpotent such that $\mathcal{N}^2=0$. Then 
\begin{equation}\label{thth2}
\rho(\Sigma_D)< \frac{\alpha^2\rho(\mathcal{M})\sqrt{\|\mathcal{N}\|_1\|\mathcal{N}\|_\infty}}{\beta^2+\alpha\beta\rho(\mathcal{M})}.
\end{equation}
\end{theorem}
\begin{proof}
Matrix $\Sigma_D$ can be written as 
$$\Sigma_D=-\frac{\alpha}{\beta^2}\left[\beta\mathcal{N}+{\alpha}\mathcal{S}_D\mathcal{N}\right],\qquad \mathcal{S}_D=\frac{\beta}{\alpha}(\beta\mathcal{P}_D^{-1}-I).$$
We show that the sets of the eigenvalues of matrices $\beta\mathcal{N}+{\alpha}\mathcal{S}_D\mathcal{N}$ and ${\alpha}\mathcal{S}_D\mathcal{N}$ coincide. According to Lemma \ref{Shemesh} matrices ${\alpha}\mathcal{S}_D\mathcal{N}$ and $\beta\mathcal{N}$ are simultaneously triangularizable. In fact, taking into account that $\mathcal{N}^2=0$,
\begin{align*}
&{\alpha}\mathcal{S}_D\mathcal{N}[{\alpha}\mathcal{S}_D\mathcal{N},\beta\mathcal{N}]={\alpha}\mathcal{S}_D\mathcal{N}({\alpha}\beta\mathcal{S}_D\mathcal{N}\mathcal{N}-{\alpha}\beta\mathcal{N}\mathcal{S}_D\mathcal{N})=0,\\
&[{\alpha}\mathcal{S}_D\mathcal{N},\beta\mathcal{N}]\beta\mathcal{N}=\beta({\alpha}\beta\mathcal{S}_D\mathcal{N}\mathcal{N}-{\alpha}\beta\mathcal{N}\mathcal{S}_D\mathcal{N})\mathcal{N}=0.
\end{align*}
Therefore, there exists an invertible matrix $Q$ such that $\alpha Q^{-1}\mathcal{S}_D\mathcal{N}Q$ and $\beta Q^{-1}\mathcal{N}Q$ are both triangular with diagonal entries equal to the eigenvalues of $\mathcal{S}_D\mathcal{N}$ and of $\mathcal{N}$, respectively. The triangularization of matrix $\beta\mathcal{N}+{\alpha}\mathcal{S}_D\mathcal{N}$ is then
$$Q^{-1}\beta\mathcal{N}Q+Q^{-1}{\alpha}\mathcal{S}_D\mathcal{N}Q,$$
whose diagonal entries are sums of an eigenvalue of $\alpha\mathcal{S}_D\mathcal{N}$ and an eigenvalue of $\beta\mathcal{N}$. Since $\mathcal{N}$ is nilpotent, its eigenvalues are all zero, and thus the eigenvalues of $\beta\mathcal{N}+{\alpha}\mathcal{S}_D\mathcal{N}$ are those of ${\alpha}\mathcal{S}_D\mathcal{N}$.
We can conclude that 
\begin{equation}\label{eqrhosi}
\rho(\Sigma_D)=\frac{\alpha^2}{\beta^2}\rho(\mathcal{S}_D\mathcal{N})\leq \frac{\alpha^2}{\beta^2}\|\mathcal{S}_D\|_2\|\mathcal{N}\|_2\leq \frac{\alpha^2}{\beta^2}\rho(\mathcal{S}_D)\sqrt{\|\mathcal{N}\|_1\|\mathcal{N}\|_\infty},
\end{equation}
since matrix $\mathcal{S}_D$ is symmetric. Moreover,
$$\rho(\mathcal{S}_D)=\frac{\beta}{\alpha}\rho(\beta\mathcal{P}_D^{-1}-I),$$
and the eigenvalues of $\beta\mathcal{P}_D^{-1}$ are in the form
$$0<\frac{\beta}{\beta - \alpha \lambda}<1,$$
with $\lambda<0$ eigenvalue of $\mathcal{M}$. Then,
\begin{equation}\label{eq:rhoS}\rho(\beta\mathcal{P}_D^{-1}-I)=\max_\lambda \left|\frac{\beta}{\beta-\alpha \lambda}-1\right|=1-\frac{\beta}{\beta+\alpha \rho(\mathcal{M})}=\frac{\alpha\rho(\mathcal{M})}{\beta+\alpha \rho(\mathcal{M})}.
\end{equation}
So, 
$$
\rho(\mathcal{S}_D)=\frac{\beta\rho(\mathcal{M})}{\beta+\alpha \rho(\mathcal{M})},
$$
and the theorem is proved after substitution in \eqref{eqrhosi}.
\end{proof}
\begin{Lemma}\label{lemnil}
Matrix $\mathcal{N}=\mathcal{N}_1$ in \eqref{delhat} is nilpotent and such that $\mathcal{N}^2=0$ 
\end{Lemma}
\begin{proof}
Let us denote by $\mathbf{x}_k$ the $k$-th grid node following the lexicographic order. 
The statement is proved by showing that for any $i,j$ and $k$
$$\mathcal{N}_{i,k}\,\mathcal{N}_{k,j}=0,$$
where $\mathcal{N}_{m,n}$ denotes the entry $(m,n)$ of matrix $\mathcal{N}$. In fact, for sake of contradiction, we assume that there exist $i,j$ and $k$ such that
\begin{equation}\label{nil}
\mathcal{N}_{i,k}\neq 0,\qquad \mathcal{N}_{k,j}\neq 0.
\end{equation}
By construction of matrix $\mathcal{N}$, entry $\mathcal{N}_{m,n}\neq 0$ if node $\mathbf{x}_m\in\Theta$ and node $\mathbf{x}_n\in\widehat{\Omega}$ is a neighbor of $\mathbf{x}_m$. Assuming \eqref{nil} implies that $\mathbf{x}_k\in \widehat{\Omega}\cap\Theta$, which leads to a contradiction as these sets are disjoint.
\end{proof}
The following corollary provides sharper bounds than Remark \ref{rem:notsharp} on the spectral radius of the iteration matrices of the iterative IMEX-E Euler method in \eqref{ACiter}--\eqref{CHiter} and of IMEX-E 2SBDF method in \eqref{ACiter2}--\eqref{CHiter2}.
\begin{corollary}\label{cor4}
Let $\mathcal{N}=\mathcal{N}_1$ as defined in \eqref{delhat}. Then, the spectral radii of the iteration matrices of IMEX-E Euler methods in \eqref{ACiter} and \eqref{CHiter2} satisfy the bounds
\begin{equation}\label{t1}
\rho_\phi \leq \frac{D_\phi^2\Delta t^2\rho(\mathcal{M})\sqrt{\|\mathcal{N}\|_1\|\mathcal{N}\|_\infty}}{(\gamma+w \Delta t)(\gamma+w \Delta t+\Delta t D_\phi\rho(\mathcal{M}))},\qquad \text{and} \qquad \rho_c \leq \frac{D_c^2{\Delta t^2}\rho(\mathcal{M})\sqrt{\|\mathcal{N}\|_1\|\mathcal{N}\|_\infty}}{\gamma(\gamma+\Delta t D_c\rho(\mathcal{M}))},
\end{equation}
respectively, where the parameter $\gamma$ is defined by \eqref{eq:gamma}
\end{corollary}
\begin{proof}
The iteration matrices of equations \eqref{ACiter} and \eqref{CHiter} are in the form \eqref{Sig}. As a consequence of Lemma \ref{lemnil}, the assumptions of Theorem \ref{theosig} are satisfied. 
The bounds in the statement are proved by replacing in \eqref{thth2} the corresponding values of $\alpha$ and $\beta$.
\end{proof}

\begin{remark}\label{rem3}
More explicit bounds on $\rho_\phi$ and $\rho_c$ can be obtained by considering that the eigenvalues of $\mathcal{M}$ are known, and that
\begin{equation}\label{rhoMb}
\rho(\mathcal{M})<\frac{4}{\Delta x^2}+\frac{4}{\Delta y^2}.
\end{equation}
\end{remark}
The next corollary provides explicit bounds in terms of the discretization steps, when $\Theta$ is a circle. This particular geometry is relevant for one of the experiments in Section \ref{sec:tests}. The arguments can be straightforwardly adapted to sets of different shape or union of disjoint sets.
\begin{corollary}\label{cor2}
Let $\Theta$ be a circle with radius $r>h=\min\{\Delta x,\Delta y\}$. Then the spectral radii of the iteration matrices of IMEX-E Euler methods in \eqref{ACiter} and \eqref{CHiter2} with $\mathcal{N}=\mathcal{N}_1$ in \eqref{delhat} are such that
\begin{align*}
\rho_\phi<\,\frac{8\sqrt{6}D_\phi^2\Delta t^2 }{h^4(\gamma+w \Delta t)(\gamma+w \Delta t+8\Delta t D_\phi/h^2)},\quad
\rho_c <\, \frac{8\sqrt{6}D_c^2{\Delta t^2}}{h^4\gamma(\gamma+8\Delta t D_c/h^2)}
\end{align*}
respectively, with $\gamma$ in \eqref{eq:gamma}.
\begin{proof}
The proof follows from Corollary \ref{cor4}, considering that with this geometry of $\Theta$ each node $(x_i,y_j)\in\widehat{\Omega}$ can not have more than two neighbors in $\Theta$ and if two, they are in perpendicular directions from $(x_i,y_j)$. Thus,
$$\|\mathcal{N}\|_1\leq \frac{1}{\Delta x^2}+\frac{1}{\Delta y^2}  \leq  \frac{2}{h^2}.$$
Similarly, each node in $\Theta$ can only have no more than three neighbors in $\widehat{\Omega}$. Hence,
$$\|\mathcal{N}\|_\infty\leq \frac{1}{\Delta x^2}+\frac{1}{\Delta y^2}+\frac{1}{ h^2} \leq \frac{3}{h^2}.$$
The statement then follows from Corollary \ref{cor4}, considering also \eqref{rhoMb}.
\end{proof}
\end{corollary}
\begin{theorem}\label{theo:conE}

Let $h=\min\{\Delta x,\Delta y\}$. The iterations \eqref{ACiter}--\eqref{CHiter} defining the iterative IMEX-E Euler method and the iterations \eqref{ACiter2}--\eqref{CHiter2} defining the iterative IMEX-E 2SBDF method, converge if
$$\left(h^2>\frac{4\sqrt{2} D_\phi}{w}, \quad \text{or}\quad \Delta t <\frac{\gamma h^2}{4\sqrt{2}D_\phi-w h^2}\right), \quad \text{and} \quad \Delta t <\frac{\gamma(1+\sqrt{3})}{8D_c}h^2,$$
with $\gamma$ as in \eqref{eq:gamma}.
\end{theorem}
\begin{proof}
Consider the bound for $\rho_\phi$ in \eqref{t1}. A sufficient condition for convergence of the iteration \eqref{ACiter} in IMEX-E Euler, is
$$\frac{D_\phi^2\Delta t^2\rho(\mathcal{M})\sqrt{\|\mathcal{N}\|_1\|\mathcal{N}\|_\infty}}{(\gamma+w \Delta t)^2}<1.$$ 
Considering the bound (see Remark \ref{rem3} and, e.g., the proof of Corollary \ref{cor2})\footnote{\color{black}Observe that the factor 32 in \eqref{eq:rhonorm} is obtained combining \eqref{rhoMb} with the largest possible value for $\|\mathcal{N}\|_1=\|\mathcal{N}\|_\infty={4}/{h^2}$. However, depending on the particular geometry of $\Theta$, one can obtain a smaller and more precise value of this factor (see, e.g., the case when $\Theta$ is a circle in Corollary \ref{cor2}).}
\begin{equation}\label{eq:rhonorm}
\rho(\mathcal{M})\sqrt{\|\mathcal{N}\|_1\|\mathcal{N}\|_\infty}\leq \frac{32}{h^4},\qquad h=\min\{\Delta x,\Delta y\},
\end{equation}
convergence is ensured if
$$\left(4\sqrt{2}D_\phi-w h^2\right)\Delta t <\gamma h^2.$$
This is true if
$$h^2>\frac{4\sqrt{2} D_\phi}{w}, \quad \text{or}\quad \Delta t <\frac{\gamma h^2}{4\sqrt{2}D_\phi-w h^2}.$$
Similarly, for ensuring $\rho_c<1$, one obtains from \eqref{t1} the sufficient condition
$$\Delta t <\frac{\gamma(1+\sqrt{3})}{8D_c}h^2.$$
The proof of the statement is completed.
\end{proof}
\begin{remark}\label{r:CFL}
According to Theorem \ref{theo:conI} and Theorem \ref{theo:conE} convergence of all the proposed iterative methods is guaranteed under sufficient conditions of the form 
\begin{equation}\label{eq:CFL}
\Delta t < Kh^2.
\end{equation} 
While such restrictions are impractical in a general context, they are feasible in our case since: 
\begin{itemize}
\item $D_\phi = 6.02 \cdot 10^{-6}$ and $w$ is large (see Remark \ref{rem:stab} and Section \ref{sec:tests}, where $w=4.43\cdot 10^8$). This makes the first condition in these theorems feasible while maintaining satisfactory accuracy.
\item $D_c = 8.50 \cdot 10^{-10}$, implies large $K$ values in the second condition.
\end{itemize}
We observe that restrictions of the type $K\Delta t<h^2$ are typical in the time integration of diffusion-reaction equations by explicit methods. However, for those methods, constant $K$ is large because of the large eigenvalues of $\mathcal{J}_{F_1}$.
\end{remark}

{\color{black}
\subsection{Neumann Boundary Conditions}\label{remsym}}
Theorems \ref{theorhosig} and \ref{theosig} and their corollaries are proved assuming that matrix $\mathcal{P}$ in \eqref{matP} is symmetric. This is the case if Dirichlet boundary conditions \eqref{extbcD} are applied on $\partial \Omega$ or if Neumann boundary conditions \eqref{extbcN} are approximated by first-order finite differences. However, when using the second-order discretization \eqref{kronsum} with \eqref{MN}, this assumption is not satisfied. 
In this section we adapt the arguments to cover this scenario.

Let $\overline{\mathcal{M}}$ be the difference between the Laplacian matrix with Dirichlet and Neumann boundary conditions on $\partial \Omega$, i.e. between \eqref{kronsum} with \eqref{MD} and \eqref{kronsum} with \eqref{MN}. That is,
$$\overline{\mathcal{M}}=I_{m_y}\otimes\frac{1}{\Delta x^2} E_x+\frac{1}{\Delta y^2}E_y\otimes I_{m_x}, \qquad E_r=\left[\begin{array}{ccccc}
0 & -1 & & &\\
0 & 0 & 0 & &\\
& \ddots & \ddots & \ddots &\\
& & 0 & 0 & 0\\
& & & -1 & 0
\end{array}\right]\in \mathbb{R}^{m_r\times m_r},\quad r\in\{x,y\},$$ 
Let
\begin{equation}\label{eq:PN}
\mathcal{P}_N=\mathcal{P}_D-\alpha \overline{\mathcal{M}},
\end{equation} 
where $\mathcal{P}_D$ is the symmetric matrix, obtained from \eqref{matP} in the case of Dirichlet boundary conditions. Then, the following result holds true.
\begin{Lemma}\label{theorhosigNeu} Let $\alpha$ and $\beta$ be positive numbers such that
\begin{equation}\label{eq:alpbetNeu}
\frac{\alpha}{\beta}\left(\frac{1}{\Delta x^2}+\frac{1}{\Delta y^2}\right)<1.
\end{equation}
Then, 
\begin{equation}\label{eq:PNinvNorm}
\|\mathcal{P}_N^{-1}\|_2 < \frac{1}{\beta - \alpha\left(\frac{1}{\Delta x^2}+\frac{1}{\Delta y^2}\right)}.
\end{equation}
\end{Lemma}
\begin{proof}
Taking into account the symmetry of $\mathcal{P}_D$ and using \eqref{eq:rhoPinv}, we have
\begin{equation}\label{eq:step1PNinv}
\|\mathcal{P}_N^{-1}\|_2=\left\|\left(I-\alpha \mathcal{P}_D^{-1}\overline{\mathcal{M}}\right)^{-1}\mathcal{P}_D^{-1}\right\|_2\leq \left\|\left(I-\alpha \mathcal{P}_D^{-1}\overline{\mathcal{M}}\right)^{-1}\right\|_2\left\|\mathcal{P}_D^{-1}\right\|_2=\left\|\left(I-\alpha \mathcal{P}_D^{-1}\overline{\mathcal{M}}\right)^{-1}\right\|_2\rho\left(\mathcal{P}_D^{-1}\right)< \frac{\left\|\left(I-\alpha \mathcal{P}_D^{-1}\overline{\mathcal{M}}\right)^{-1}\right\|_2}{\beta},
\end{equation}
Let us show now that
$$\left\|\alpha \mathcal{P}_D^{-1}\overline{\mathcal{M}}\right\|_2<1.$$
In fact,
$$\left\|\alpha \mathcal{P}_D^{-1}\overline{\mathcal{M}}\right\|_2\leq \alpha \left\|\mathcal{P}_D^{-1}\right\|_2
\left\|\overline{\mathcal{M}}\right\|_2 < \frac{\alpha}{\beta}\left\|\overline{\mathcal{M}}\right\|_2.$$
Moreover,
$$\left\|\overline{\mathcal{M}}\right\|_2=\left\|I_{m_y}\otimes\frac{1}{\Delta x^2} E_x+\frac{1}{\Delta y^2}E_y\otimes I_{m_x}\right\|_2\leq\frac{1}{\Delta x^2}\|E_x\|_2+\frac{1}{\Delta y^2}\|E_y\|_2.$$
Matrices $E_r$ with $r\in\{x,y\}$ are low rank matrices of the form
$$E_r=-e_1e_2^T-e_{m_r}e_{m_r-1}^T,$$
where $e_j$ is the $j$-th vector of the canonical basis of $\mathbb{R}^r$.
Therefore,
\begin{equation*}
\|E_r\|_2 = \rho(E_r^TE_r)=1,
\end{equation*}
since $E_r^TE_r$ is diagonal with 0 or 1 entries.

Combining the above results, and considering assumption \eqref{eq:alpbetNeu}, we have proved that 
\begin{equation}\label{eq:normaPM}
\left\|\alpha \mathcal{P}_D^{-1}\overline{\mathcal{M}}\right\|_2< \frac{\alpha}\beta\left(\frac{1}{\Delta x^2}+\frac{1}{\Delta y^2}\right)<1.
\end{equation}
Therefore, the Neumann series
$$\left(I-\alpha \mathcal{P}_D^{-1}\overline{\mathcal{M}}\right)^{-1}=\sum_{k=0}^\infty \left(\alpha \mathcal{P}_D^{-1}\overline{\mathcal{M}}\right)^k,$$
is convergent. Thus,
\begin{equation}\label{eq:norminv}
\left\|\left(I-\alpha \mathcal{P}_D^{-1}\overline{\mathcal{M}}\right)^{-1}\right\|_2\leq \sum_{k=0}^\infty \left\|\alpha\mathcal{P}_D^{-1}\overline{\mathcal{M}}\right\|_2^k<\sum_{k=0}^\infty \left(\frac{\alpha}\beta\left(\frac{1}{\Delta x^2}+\frac{1}{\Delta y^2}\right)\right)^k= \frac{\beta}{\beta-{\alpha}\left(\frac{1}{\Delta x^2}+\frac{1}{\Delta y^2}\right)}
\end{equation}
It follows then from \eqref{eq:step1PNinv} that
$$\|\mathcal{P}^{-1}_N\|_2<\frac{\left\|\left(I-\alpha \mathcal{P}_D^{-1}\overline{\mathcal{M}}\right)^{-1}\right\|_2}{\beta}<\frac{1}{\beta-{\alpha}\left(\frac{1}{\Delta x^2}+\frac{1}{\Delta y^2}\right)},$$
which concludes the proof.
\end{proof}
Based on the results in Lemma \ref{theorhosigNeu}, the following result holds true. We omit the proof as it follows the same steps that prove Theorem \ref{theorhosig}.

\begin{theorem}\label{thrhosigNeu}
Under assumption \eqref{eq:alpbetNeu}, the spectral radius of matrix 
\begin{equation*}
\Sigma_N=-\alpha\mathcal{P}_N^{-1}\mathcal{N},\qquad \mathcal{N}=\mathcal{N}_1+\mathcal{N}_2
\end{equation*}
with $\mathcal{P_N}$ as in \eqref{eq:PN}, is such that 
\begin{equation}\label{eq:rhosigmaN}
\rho(\Sigma_N)\leq \frac{4\alpha\left(\frac{1}{\Delta x^2}+\frac{1}{\Delta y^2}\right)}{\beta-{\alpha}\left(\frac{1}{\Delta x^2}+\frac{1}{\Delta y^2}\right)}.
\end{equation}
\end{theorem}
We observe that assuming \eqref{eq:alpbetNeu}, is equivalent as requiring that the denominator in \eqref{eq:rhosigmaN} is positive.

\begin{remark}\label{rem:IMEIN}
Note that Lemma \ref{theorhosigNeu} and Theorem \ref{thrhosigNeu} rely on assumption \eqref{eq:alpbetNeu}. After replacing the values of $\alpha$ and $\beta$ appearing in the coefficient matrices of \eqref{ACiter}--\eqref{CHiter2}, assumption \eqref{eq:alpbetNeu} amounts to
\begin{equation}\label{eq:cond1N}
\frac{\Delta t D_\phi}{\gamma+w\Delta t}\left(\frac{1}{\Delta x^2}+\frac{1}{\Delta y^2}\right)<1,\qquad \frac{\Delta t D_c}\gamma \left(\frac{1}{\Delta x^2}+\frac{1}{\Delta y^2}\right)<1,
\end{equation}
with $\gamma$ given in \eqref{eq:gamma}. Constraints on the steps can be derived from \eqref{eq:cond1N}, but we omit them here, since they are automatically satisfied when the sufficient convergence conditions in Theorem \ref{theo:convIN} and Theorem \ref{theo:convEN} below are met.
\end{remark}

By substituting the values of $\alpha$ and $\beta$ into \eqref{eq:rhosigmaN}, the following corollary provides an upper bound for the spectral radii of the iteration matrices in \eqref{ACiter}--\eqref{CHiter2} with $\mathcal{N}=\mathcal{N}_1+\mathcal{N}_2$, under Neumann boundary conditions on $\partial \Omega$.
\begin{corollary}\label{corrhoneu}
In the case of Neumann boundary conditions on $\partial \Omega$, the spectral radii of the iteration matrices of IMEX-I methods in \eqref{ACiter}--\eqref{CHiter2} fulfill the following bounds:
\begin{equation}\label{eq:upbN}
\rho_\phi \leq \frac{4D_\phi\Delta t \left(\frac{1}{\Delta x^2}+\frac{1}{\Delta y^2}\right)}{\gamma+w\Delta t-{D_\phi\Delta t }\left(\frac{1}{\Delta x^2}+\frac{1}{\Delta y^2}\right)},\qquad \rho_c \leq \frac{4D_c\Delta t\left(\frac{1}{\Delta x^2}+\frac{1}{\Delta y^2}\right)}{\gamma-{D_c\Delta t}\left(\frac{1}{\Delta x^2}+\frac{1}{\Delta y^2}\right)},
\end{equation}
assuming that the steps satisfy \eqref{eq:cond1N}, and with $\gamma$ given in \eqref{eq:gamma}.
\end{corollary}
{\color{black}The following result gives sufficient conditions on the steps for the convergence of the IMEX-I methods. The proof is analogous to that of Theorem \ref{theo:conI}.
\begin{theorem}\label{theo:convIN}
Let $h=\min\{\Delta x, \Delta y\}$ and $\Delta t$ satisfy \eqref{eq:cond1N}. The iterations defined in equations \eqref{ACiter}–\eqref{CHiter} for the IMEX-I Euler method, and by \eqref{ACiter2}–\eqref{CHiter2} for the IMEX-I 2SBDF method, converge provided that
$$\left(h^2>\frac{10D_\phi}w, \quad \text{or}\quad \Delta t<\frac{\gamma h^2}{10D_\phi-wh^2} \right), \quad \text{and} \quad \Delta t<\frac{\gamma h^2}{10 D_c}.$$
\end{theorem}
}

Observe that the sufficient condition obtained here for convergence are slightly stronger than their analogue for the case of Dirichlet boundary condition in Theorem \ref{theo:conI}. On the other hand, as noted in Remark \ref{r:CFL}, these restrictions do not prevent the use of relatively large time steps in our case.

The following Lemma is key to adapt the result in Theorem \ref{theosig} to the case of boundary conditions of Neumann type on $\partial\Omega$, and to study sufficient conditions for the convergence of IMEX-E schemes.
\begin{Lemma}\label{Lem:SN}
Let
$$\mathcal{S_N}=\frac{\beta}{\alpha}(\beta\mathcal{P_N}^{-1}-I),$$
with $P_N$ in \eqref{eq:PN}. If $\alpha$ and $\beta$ satisfy \eqref{eq:alpbetNeu}, then
$$\|\mathcal{S_N}\|_2\leq \frac{\beta\rho(\mathcal{M})}{\beta+\alpha \rho(\mathcal{M})}+\frac{\beta\left(\frac{1}{\Delta x^2}+\frac{1}{\Delta y^2}\right)}{\beta-\alpha \left(\frac{1}{\Delta x^2}+\frac{1}{\Delta y^2}\right)}.$$
\end{Lemma}
\begin{proof}
The symmetry of $\mathcal{P}_D$ and \eqref{eq:rhoS}, yield
\begin{equation}\label{eq:normSN}
\|\mathcal{S_N}\|_2\leq \frac{\beta^2}{\alpha}\|\mathcal{P}_N^{-1}-\mathcal{P}_D^{-1}\|_2+\frac{\beta}{\alpha}\|\beta\mathcal{P}_D^{-1}-I\|_2\leq 
\frac{\beta^2}{\alpha}\|\mathcal{P}_N^{-1}-\mathcal{P}_D^{-1}\|_2+\frac{\beta\rho(\mathcal{M})}{\beta+\alpha \rho(\mathcal{M})}.
\end{equation}
To bound the first term at the right hand side, we observe that
$$\mathcal{P}_N^{-1}-\mathcal{P}_D^{-1}=\left[\left(I-\alpha\mathcal{P}_D^{-1}\overline{\mathcal{M}}\right)^{-1}-I\right]\mathcal{P}_D^{-1}=\alpha\mathcal{P}_D^{-1}\overline{\mathcal{M}}\left(I-\alpha\mathcal{P}_D^{-1}\overline{\mathcal{M}}\right)^{-1}\mathcal{P}_D^{-1}.$$
Hence, combining \eqref{eq:rhoPinv}, the first bound in \eqref{eq:normaPM}, and \eqref{eq:norminv} yields,
$$\|\mathcal{P}_N^{-1}-\mathcal{P}_D^{-1}\|_2\leq \|\alpha\mathcal{P}_D^{-1}\overline{\mathcal{M}}\|_2\left\|\left(I-\alpha\mathcal{P}_D^{-1}\overline{\mathcal{M}}\right)^{-1}\right\|_2\left\|\mathcal{P}_D^{-1}\right\|_2\leq \frac{\alpha\left(\frac{1}{\Delta x^2}+\frac{1}{\Delta x^2}\right)}{\beta^2-\alpha\beta\left(\frac{1}{\Delta x^2}+\frac{1}{\Delta x^2}\right)}
$$
The proof is completed by substituting the bound obtained above in \eqref{eq:normSN}.
\end{proof}
\begin{theorem}\label{theosigN}
Let $\alpha$ and $\beta$ satisfy \eqref{eq:alpbetNeu} and $$\Sigma_N=-\alpha\mathcal{P}_N^{-1}\mathcal{N}$$ with $\mathcal{N}$ nilpotent such that $\mathcal{N}^2=0$. Then 
\begin{equation}\label{thth2N}
\rho(\Sigma_N)< \alpha^2\sqrt{\|\mathcal{N}\|_1\|\mathcal{N}\|_\infty}\left(\frac{\rho(\mathcal{M})}{\beta^2+\alpha\beta\rho(\mathcal{M})}+\frac{\frac{1}{\Delta x^2}+\frac{1}{\Delta y^2}}{\beta^2-\alpha\beta \left(\frac{1}{\Delta x^2}+\frac{1}{\Delta y^2}\right)}\right).
\end{equation}
\end{theorem}
\begin{proof}
We can rewrite
$$\Sigma_N=-\frac{\alpha}{\beta^2}\left[\beta\mathcal{N}+{\alpha}\mathcal{S}_N\mathcal{N}\right],\qquad \mathcal{S}_N=\frac{\beta}{\alpha}(\beta\mathcal{P}_N^{-1}-I).$$
With similar steps as those leading to equation \eqref{eqrhosi} in the proof of Theorem \ref{theosig}, we obtain
$$\rho(\Sigma_N)\leq \frac{\alpha^2}{\beta^2}\|\mathcal{S}_N\|_2\|\mathcal{N}\|_2\leq \frac{\alpha^2}{\beta^2}\|\mathcal{S}_N\|_2\sqrt{\|\mathcal{N}\|_1\|\mathcal{N}\|_\infty}.
$$
Lemma \ref{Lem:SN} then yields
$$\rho(\Sigma_N)\leq \frac{\alpha^2}{\beta^2}\|\mathcal{S}_N\|_2\sqrt{\|\mathcal{N}\|_1\|\mathcal{N}\|_\infty}\leq \alpha^2\sqrt{\|\mathcal{N}\|_1\|\mathcal{N}\|_\infty}\left(\frac{\rho(\mathcal{M})}{\beta^2+\alpha\beta \rho(\mathcal{M})}+\frac{\frac{1}{\Delta x^2}+\frac{1}{\Delta y^2}}{\beta^2-\alpha\beta \left(\frac{1}{\Delta x^2}+\frac{1}{\Delta y^2}\right)}\right),
$$
which completes the proof.
\end{proof}
The following corollary provides bounds on the spectral radius of the iteration matrices of iterative IMEX-E Euler in \eqref{ACiter}--\eqref{CHiter} and iterative IMEX-E 2SBDF in \eqref{ACiter2}--\eqref{CHiter2}. It is the analogue of Corollary \ref{cor4}  when Neumann boundary conditions are assigned on the boundary of the extended rectangular domain $\Omega$. The proof is readily obtained substituting the values of $\alpha$ and $\beta$ that define the methods.
\begin{corollary}\label{cor4N}
Let $\mathcal{N}=\mathcal{N}_1$ be as defined in \eqref{delhat}. Then, if the steps satisfy \eqref{eq:cond1N}, the spectral radii of the iteration matrices of IMEX-E Euler methods in \eqref{ACiter}--\eqref{CHiter2} are such that
\begin{equation}\label{t1N}
\begin{array}{l}
\displaystyle{\rho_\phi \leq\frac{ D_\phi^2\Delta t^2\sqrt{\|\mathcal{N}\|_1\|\mathcal{N}\|_\infty}}{\gamma+w \Delta t}\left(\frac{\rho(\mathcal{M})}{\gamma+w \Delta t+\Delta t D_\phi\rho(\mathcal{M})}+\frac{\frac{1}{\Delta x^2}+\frac{1}{\Delta y^2}}{\gamma+w \Delta t-\Delta tD_\phi \left(\frac{1}{\Delta x^2}+\frac{1}{\Delta y^2}\right)}\right)},\\[.8cm]
\displaystyle{\rho_c \leq \frac{D_c^2\Delta t^2\sqrt{\|\mathcal{N}\|_1\|\mathcal{N}\|_\infty}}\gamma\left(\frac{\rho(\mathcal{M})}{\gamma+D_c\Delta t \rho(\mathcal{M})}+\frac{\frac{1}{\Delta x^2}+\frac{1}{\Delta y^2}}{\gamma-D_c\Delta t \left(\frac{1}{\Delta x^2}+\frac{1}{\Delta y^2}\right)}\right),}
\end{array}
\end{equation} 
respectively.
\end{corollary}
The following corollary adapts the bounds to the case when $\Theta$ is a circle. The arguments can be generalized straightforwardly for different shapes.
\begin{corollary}\label{corcircN}
Let $\Theta$ be a circle with radius $r>h=\min\{\Delta x,\Delta y\}$. Then, if \eqref{eq:cond1N} holds true, the spectral radii of the iteration matrices of IMEX-E Euler methods in \eqref{ACiter}--\eqref{CHiter2} with $\mathcal{N}=\mathcal{N}_1$ in \eqref{delhat}, satisfy the bounds:
\begin{equation}\label{eq:upbNE}
\begin{array}{l}
\displaystyle{\rho_\phi \leq \frac{2\sqrt{6}D_\phi^2\Delta t^2}{h^4(\gamma+w \Delta t)} \left(\frac{4}{\gamma+w \Delta t+{8}\Delta t D_\phi/h^2}+\frac{1}{\gamma+w \Delta t- {2}\Delta tD_\phi/h^2}\right)},\\[0.6cm]
\displaystyle{\rho_c \leq \frac{2\sqrt{6}D_c^2\Delta t^2}{\gamma h^4}\left(\frac{4}{\gamma+{8}D_c\Delta t/h^2}+\frac{1}{\gamma-{2}D_c\Delta t/h^2}\right),}
\end{array}
\end{equation} 
respectively.
\end{corollary}
\begin{proof}
The proof follows straightforwardly from the bounds in Corollary \ref{cor4N}, \eqref{rhoMb} and the bounds on the norms of $\mathcal{N}$ in the proof of Corollary \ref{cor2}.
\end{proof}
{\color{black}As usual, we obtain sufficient conditions for convergence by requiring that the bounds on the spectral radii in \eqref{t1N} are smaller than 1. With similar steps as in the proof of Theorem \ref{theo:conE} (but with slightly more involved calculations) the following theorem can be proved.
\begin{theorem}\label{theo:convEN}
Let $h=\min\{\Delta x, \Delta y\}$ and $\Delta t$ satisfy \eqref{eq:cond1N}. The iterations defined in equations \eqref{ACiter}–\eqref{CHiter} for the IMEX-E Euler method, and by \eqref{ACiter2}–\eqref{CHiter2} for the IMEX-E 2SBDF method, converge provided that
$$\left(h^2>\frac{(1+\sqrt{41})D_\phi}{w}, \quad \text{or}\quad \Delta t<\frac{\gamma(\sqrt{41}-1)h^2}{40D_\phi+2w h^2} \right), \quad \text{and} \quad \Delta t<\frac{\gamma(\sqrt{41}-1)h^2}{40D_c},$$
where the parameter $\gamma$ is defined as in \eqref{eq:gamma}.
\end{theorem}
}
Similarly as observed after Theorem \ref{theo:convIN}, the sufficient requirements for convergence are slightly stricter than their analogues  in Theorem \ref{theo:conE} for the case of Dirichlet boundary condition on the boundary of $\Omega$. The constant factors in these bounds can be reduced and refined by exploiting the knowledge of the specific shape of $\Theta$ to have smaller and more precise bounds on the norms of $\mathcal{N}$. The observations in Remark \ref{r:CFL} still hold true.
\begin{remark}
A different approach could be that of keeping the symmetric matrix $\mathcal{P}_D$ in the iterative matrix, and incorporate (if possible) the non-symmetric part arising from $\overline{\mathcal{M}}$ in \eqref{eq:PN} into $\mathcal{N}$. In this case, one may be able to use the slightly more convenient convergence results in Section \ref{sec:Dir}.

In this spirit, the iterative IMEX-I methods can be adjusted simply by redefining $\mathcal{N}=\mathcal{N}_1+\mathcal{N}_2+\overline{\mathcal{M}}$ in \eqref{ACiter}--\eqref{CHiter2}. Note that although matrix $\mathcal{N}$ is changed, the result in Corollary~\ref{cor1} still holds true as the bounds on the norms of $\mathcal{N}$ do not change.

The rationale in IMEX-E methods is that small correction terms can be treated explicitly. We distinguish here two circumstances.

\begin{itemize}
\item[a)] The corrosion interface does not reach $\partial \Omega$ for all $t\in[0,T]$. In this case we do not expect a great difference in the values of the dependent variables close to the boundary when stepping from time $t_n$ to $t_{n+1}$. Thus, we can modify the iterative IMEX-E methods by setting $\mathcal{G}=\mathcal{N}_2+\overline{\mathcal{M}}$.
\item[b)] 
The corrosion interface will reach $\partial \Omega$ at some time $t_c\in[0,T]$. After $t_c$ the value of the dependent variables close to the boundary may change sensibly from one time step to the next one. Thus, it would be recommendable to redefine $\mathcal{N}=\mathcal{N}_1+\overline{\mathcal{M}}$ and $\mathcal{G}=\mathcal{N}_2$ for accuracy. However, in this case the new matrix $\mathcal{N}$ is not nilpotent of order two and thus Theorem \ref{theosig} can not be applied. Nevertheless, the bounds in Remark \ref{rem:notsharp} are valid. 
\end{itemize}
In case a) matrix $\mathcal{N}$ is unchanged and all the convergence results in Section \ref{sec:Dir} still apply. In case b), although one could introduce control techniques to detect the critical time $t_c$ and redefine the matrices only for $t>t_c$, the bounds in Remark \ref{rem:notsharp} are larger than the ones in Corollary \ref{cor4}. So the speed of convergence may be drastically reduced.
\end{remark}
{\color{black}
\begin{remark}
The case of mixed boundary conditions, (Dirichlet type on some edges and Neumann type on the other edges) can be treated in a similar way as done in this section. In this case the bound \eqref{eq:PNinvNorm} may be lower. In fact, if Dirichlet boundary conditions are assigned on parallel edges, for example on the horizontal edges, then
$$\overline{\mathcal{M}}=\frac{1}{\Delta y^2}E_y\otimes I_{m_x}.$$
In this case, assumptions \eqref{eq:alpbetNeu} and \eqref{eq:PNinvNorm} reduce to
$$\frac{\alpha}{\beta\Delta y^2}<1,\qquad \|\mathcal{P}_N^{-1}\|_2<\frac{1}{\beta -\frac{\alpha}{\Delta y^2}},$$
respectively. The consequent results can then be adapted in a straightforwardly.
\end{remark}}
\subsection{Error Analysis}
We observe that even assuming convergence, we can not expect that the solutions of the iterative versions of the IMEX methods satisfy exactly
$$ {\phi}_{i,j}^n=0,\qquad c_{i,j}^n=0,\qquad (x_i,y_j)\in\Theta, $$
for each time step $n=1,\ldots N$, unless infinitely many iterations are performed. So, in an actual algorithm, a stopping criterion is necessary and these conditions hold within a certain tolerance.

In this section, we analyze the propagation of the error in the proposed methods, taking into account both the local truncation error of the scheme and the effect of the chosen tolerance. For brevity we limit the discussion only to the first-order IMEX-I methods with Dirichlet boundary conditions on $\partial \Omega$.

We shall use the Lipschitz continuity of functions $F_1$ and $F_2$,
\begin{equation}\label{eq:Lips}
 \|F_1(\mathbf{x},\boldsymbol{\phi}_1,\mathbf{c}_1)-F_1(\mathbf{x},\boldsymbol{\phi}_2,\mathbf{c}_2)\|_2 \leq L_1 (\|\boldsymbol{\phi}_1-\boldsymbol{\phi}_2\|_2 + \|\mathbf{c}_1-\mathbf{c}_2\|_2), \quad \|F_2(\boldsymbol{\phi}_1)-F_2(\boldsymbol{\phi}_2)\|_2\leq L_2 \|\boldsymbol{\phi}_1 - \boldsymbol{\phi}_2\|_2.
\end{equation}
Let $\boldsymbol{\phi}^{N}=\boldsymbol{\phi}^{N,k_N}$ and $\mathbf{c}^N=\mathbf{c}^{N,k_N}$ denote the solution vectors of the iterative IMEX-I Euler method \eqref{ACiter}--\eqref{CHiter} at time $T$ obtained when, after $k_N$ iterations, the set tolerance is satisfied. Let $\phi^{N}=\phi(t_N)$ and $c^{N}=c(t_N)$ denote the vector with entries given by the exact solution evaluated at the spatial grid points. For any vector function $\mathbf{v}$ defined on the spatial mesh let $\mathbf{v}_{\widehat{\Omega}}$ be the smaller vector extracted from $\mathbf{v}$ with entries given by the values of $\mathbf{v}$ on points in $\widehat{\Omega}.$ Vector $\mathbf{v}_{\Theta}$ is defined analogously.

Let us define the error vectors at points in $\widehat{\Omega}$ and in $\Theta$ as
$$e_\phi^N = [\boldsymbol{\phi}^{N} - \phi^{N}]_{\widehat{\Omega}}, \qquad e_c^N =[ \mathbf{c}^N - c^{N}]_{\widehat{\Omega}},\qquad \epsilon_\phi^N = [\boldsymbol{\phi}^{N} - \phi^{N}]_{\Theta}, \qquad \epsilon_c^N = [\mathbf{c}^N - c^{N}]_{\Theta},$$
and
$$\epsilon_\phi^{N,k} = [\boldsymbol{\phi}^{N,k} - \boldsymbol{\phi}^{N,k-1}]_{\Theta}, \qquad \epsilon_c^{N,k} = [\mathbf{c}^{N,k} - \mathbf{c}^{N,k-1}]_{\Theta}.$$
\begin{theorem}\label{th:error}
Let us define the errors 
$$E^N_{\widehat{\Omega}} = \max\left\{\|e_\phi^N\|_2, \|e_c^N\|_2\right\},\qquad E^N_{\Theta}=\max\left\{ \|\epsilon_\phi^N\|_2, \,\|\epsilon^N_c\|_2\right\},$$
and fix two positive values $\eta$ and $\delta$ independent of the steps. Assume that the steps are such that 
\begin{equation}\label{hstab}
\min\{\Delta x^2, \Delta y^2\} = h^2 > \max \left\{ \frac{8D_\phi\Delta t}{1+w\Delta t}, \frac{L_2D_c}\eta \right\} \qquad \Delta t<\frac{h^2}{8D_c}
\end{equation} and that the iterations stop when 
\begin{equation}\label{stopth}
 \max\left\{\| \epsilon_\phi^{N,k} \|_2, \| \epsilon_{c}^{N,k} \|_2\right\}<\varepsilon<\delta \Delta t.
\end{equation}
Then, the following bounds on the errors hold:
\begin{equation}\label{errth}
E^N_{\widehat{\Omega}}  \leq  C \mathrm{e}^{\xi T}  \left[\Delta t + \Delta x^2 + \Delta y^2 +  \delta \right],\qquad E^N_{\Theta} < \delta T,
\end{equation}
where $C$ and $\xi$ are independent of the steps, and $\xi$ is proportional to $\eta$.
\end{theorem}

\begin{proof}
For simplicity, we assume $\Delta x = \Delta y=h$. Otherwise, the proof can be adapted straightforwardly. Along this proof $C$ denotes generic constants independent of $\Delta t$, $h$ and $T$.

The iterative IMEX-I Euler method can be equivalently written as
\begin{align}\label{stabAC}
[(1+w\Delta t)I_{m_xm_y}\!-\!\Delta tD_\phi(\mathcal{M}-\mathcal{N})]\boldsymbol{\phi}^{N,k}&=\boldsymbol{\phi}^{N-1}\!+\!\Delta t[w\boldsymbol{\phi}^{N-1}+ F_1(\mathbf{x},\boldsymbol{\phi}^{N-1},\mathbf{c}^{N-1})+D_\phi\,\mathcal{N}(\boldsymbol{\phi}^{N,k}-\boldsymbol{\phi}^{N,k-1})],\\\label{stabCH}
[I_{m_xm_y}\!-\!\Delta tD_c(\mathcal{M}-\mathcal{N})]\mathbf{c}^{N,k}&=\mathbf{c}^{N-1}\!+\!\Delta tD_c [(\mathcal{M}-\mathcal{N})F_2(\boldsymbol{\phi}^{N})+\mathcal{N}(\mathbf{c}^{N,k}-\mathbf{c}^{N,k-1})],
\end{align}
with $\mathcal{N}=\mathcal{N}_1+\mathcal{N}_2$. Conditions \eqref{hstab} imply \eqref{eq:convas}, so the iterative procedure converges.
Now we observe that:
\begin{itemize}
\item[•] If $\mathbf{x}_i \in \Theta$, then the $i$-th row and column of $(\mathcal{M}-\mathcal{N})$ are zero. Therefore, by a simultaneous permutation of rows and columns, we can rewrite the matrix in block-diagonal form with a zero block (of size equal to the number of nodes in $\Theta$) and $\mathcal{M}_1$, a principal submatrix of $\mathcal{M}$. Thus each of equations \eqref{stabAC}--\eqref{stabCH} splits into two parts: one involving only the nodes in $\Theta$, the other describing the evolution in $\widehat{\Omega}$. Hence, we can write
$$\phi^N = [\phi_\Theta^N, \phi^N_{\widehat{\Omega}}]^T, \qquad c^N = [c_\Theta^N, c_{\widehat{\Omega}}^N]^T,$$
and we use the same notation for the corresponding blocks of the numerical solutions.  
\item[•] Let $\zeta$ be an eigenvalue of $\beta I - \alpha\mathcal{M}_1$. Since $\mathcal{M}$ is symmetric, by Cauchy interlacing theorem, one has that
$$\beta - \alpha \bar{\lambda} \leq \zeta \leq \beta - \alpha \lambda^*,$$
where $\lambda^*<\bar{\lambda}<0$ are the smallest and largest eigenvalues of $\mathcal{M}$, respectively.
Therefore, for positive $\alpha$ and $\beta$,
\begin{equation}\label{specM1}
  \left\|\left(\beta I - \alpha \mathcal{M}_1\right)^{-1}\right\|_2 \leq \frac{1}{\beta-\alpha \bar{\lambda}}\leq \frac{1}\beta.
\end{equation}
\item[•] Let $\rho_\phi$ and $\rho_c$ be the local truncation errors obtained by inserting the exact solution into \eqref{stabAC} and \eqref{stabCH}, respectively. There exists a constant $C$ such that
\begin{equation}\label{rhobou}
\rho_\phi \leq C(\Delta t + h^2),\qquad \rho_c \leq C(\Delta t + h^2).
\end{equation}
\end{itemize}
Let us consider first the block of equations involving only the values of the solution in $\Theta$. This is
\begin{equation*}
(1+w\Delta t)\boldsymbol{\phi}^{N}_\Theta=(1+w\Delta t)\boldsymbol{\phi}^{N-1}_\Theta+ \Delta t D_\phi\,\mathcal{N}\epsilon_\phi^{N,k},\qquad
\mathbf{c}^{N}_\Theta=\mathbf{c}^{N-1}_\Theta\!+\!\Delta tD_c \mathcal{N}\epsilon^{N,k}_c.
\end{equation*}
Since the exact solution is identically zero in $\Theta$, the above give the equations of the error. Then, using \eqref{hstab} and \eqref{stopth},
\begin{align*}
\|\epsilon_{\phi}^{N}\|_2&\leq \|\epsilon_{\phi}^{N-1}\|_2 + \frac{\Delta t D_\phi}{1+w\Delta t}\| \mathcal{N}\epsilon_\phi^{N,k}\|_2\leq \|\epsilon_{\phi}^{N-1}\|_2 + \frac{\Delta t D_\phi}{1+w\Delta t}\sqrt{\| \mathcal{N}\|_1\| \mathcal{N}\|_\infty}\|\epsilon_\phi^{N,k}\|_2 \leq \frac{8T \varepsilon D_\phi}{(1+w\Delta t) h^2}<\delta T,\\
\|\epsilon_{c}^{N}\|_2&\leq \|\epsilon_{c}^{N-1}\|_2+\Delta tD_c \|\mathcal{N}\epsilon^{N,k}_c\|_2\leq \|\epsilon_{c}^{N-1}\|_2+\Delta tD_c \sqrt{\| \mathcal{N}\|_1\| \mathcal{N}\|_\infty} \|\epsilon_{c}^{N-1}\|_2 \leq \frac{8T\varepsilon D_c}{h^2}<\delta T.
\end{align*}
Thus, we can conclude that $E^N_{\Theta} < \delta T$ and that the error in $\Theta$ increases at most linearly as time advances.

Let us consider now the second block of equations. By inserting the exact solution into these equations and subtracting the resulting relations from them, we obtain the following error equations:
\begin{align*}
[(1+w\Delta t)I-\Delta tD_\phi\mathcal{M}_1]e_\phi^{N}&\,= (1+w\Delta t ) e_\phi^{N-1}+\Delta t\left[F_1(\mathbf{x},\boldsymbol{\phi}^{N-1},\mathbf{c}^{N-1})-F_1(\mathbf{x},\phi^{N-1},c^{N-1})\right]+\Delta t\left[\rho_\phi +\frac{D_\phi}{h^2}\mathbf{k}\circ \epsilon_\phi^{N,k}\right],\\
[I-\Delta tD_c\mathcal{M}_1]e_c^{N}&\,=e_{c}^{N-1}+\Delta tD_c \mathcal{M}_1\left[F_2(\boldsymbol{\phi}^{N})-F_2({\phi}^{N})\right]+\Delta t \left[\rho_c+\frac{D_c}{h^2}\mathbf{k}\circ \epsilon_c^{N,k}\right],
\end{align*}
where $\mathbf{k}$ is a vector with entry $k_{i}$ equal to the number of neighbors of $\mathbf{x}_i$ in $\Theta$, and $\circ$ is the Hadamard product.

Then, considering \eqref{eq:Lips} and \eqref{specM1}, 
\begin{align*}
\|e_\phi^N\|_2 & \leq \left(1+\frac{L_1\Delta t}{1+w\Delta t}\right)\|e_\phi^{N-1}\|_2+\frac{L_1\Delta t}{1+w\Delta t} \|e_c^{N-1}\|_2+\frac{\Delta t}{1+w\Delta t}\left[\rho_\phi +\frac{ \varepsilon D_\phi}{h^2}\|\mathbf{k}\|_2\right],\\
\|e_c^N\|_2 &\leq \|e_c^{N-1}\|_2 + \frac{8 L_2 D_c \Delta t}{h^2} \|e_\phi^N\|_2 + \Delta t\left[\rho_c+\frac{\varepsilon D_c}{h^2}\|\mathbf{k}\|_2\right].
\end{align*}
Substituting the former bound into the latter, one has
$$E^N\leq M E^{N-1}+  \mathbf{v}\Delta t, \qquad \text{where} \qquad M = \left(\begin{array}{cc} 1 + a & a\\
(1+a)b & 1 + ab
\end{array}\right),\qquad \mathbf{v}= \left(\begin{array}{c}
v_1\\ v_2
\end{array}\right),$$ with $$a= \frac{L_1\Delta t}{1+w\Delta t},\quad b= \frac{8 L_2 D_c \Delta t}{h^2},\quad v_1 = \frac{1}{1+w\Delta t}\left[\rho_\phi +\frac{ \varepsilon D_\phi}{h^2}\|\mathbf{k}\|_2\right], \quad v_2 = bv_1 + \rho_c+\frac{\varepsilon D_c}{h^2}\|\mathbf{k}\|_2.$$
Thus,
$$E^N \leq \Delta t\sum_{k=0}^{N-1} M^k \mathbf{v},$$
with
$$\|M\|_2\leq \sqrt{\|M\|_1 \|M\|_\infty}\leq 1+b+C\Delta t \leq 1+\xi\Delta t,$$
where, considering \eqref{hstab}, $\xi$ does not depend on the steps.
Then,
$$\|E^N\|_2 \leq  C \Delta t \sum_{k=0}^{N-1} (1+\xi\Delta t)^k \|\mathbf{v}\|_2\leq C \frac{\mathrm{e}^{ \xi T }}{\xi }\|\mathbf{v}\|_2.$$
Considering \eqref{rhobou}, we have that
$$ v_1 \leq \frac{C}{1+w\Delta t}\left[\Delta t + h^2 +\frac{ \varepsilon D_\phi}{h^2}\right],\quad v_2 \leq C\left[\Delta t + h^2 +\frac{ \varepsilon}{h^2}\left(\frac{D_\phi}{1+w\Delta t}+D_c\right)\right],$$
therefore, there exists a constant $C$ such that
$$\|E^N\|_2 \leq C \mathrm{e}^{ \xi T }\left[\Delta t + h^2 +\frac{ \varepsilon}{h^2}\left(\frac{D_\phi}{1+w\Delta t}+D_c\right)\right]\leq C \mathrm{e}^{ \xi T }\left[\Delta t + h^2 + \delta \right],$$
proving the statement.\\
\end{proof}
\begin{remark}\label{rem:delta}
We observe that Theorem \ref{th:error} guarantees stability of the method for sufficiently small $\Delta t$ and that the impact of the tolerance in \eqref{stopth} is limited if $\delta < \max\{\Delta t,h^2\}$. However, condition \eqref{hstab} provides an absolute lower bound on $h$, which is required to control $\xi$. We emphasize that this condition is only sufficient for controlling numerical errors, and that in practice, we have not observed instabilities for small values of $h$, provided that $\Delta t$ satisfies \eqref{hstab}. 
\end{remark}

\subsection{Stopping Criterion}\label{sec:stop}
While Theorem \ref{th:error} provides one single stopping criterion \eqref{stopth}, in practical implementations, we fix suitably small positive tolerances $\varepsilon_1$, $\varepsilon_2$ and $\varepsilon_3$, and terminate the iterations once the following conditions are satisfied:
\begin{align}\label{tolphiomega}
&\max_{i,j}|{\phi}_{i,j}^{n,k+1}-{\phi}_{i,j}^{n,k}|<\varepsilon_1,\quad (x_i,y_j)\in \widehat{\Omega},\\\label{tolphitheta}
& \max_{i,j}|{\phi}_{i,j}^{n,k+1}|<n\varepsilon_2{\color{black}/N}, \quad\text{or} \quad \max_{i,j}|{\phi}_{i,j}^{n,k+1}-{\phi}_{i,j}^{n,k}|<\varepsilon_3, \quad (x_i,y_j)\in {\Theta}\\\label{tolcomega}
&\max_{i,j}|{c}_{i,j}^{n,k+1}-{c}_{i,j}^{n,k}|<\varepsilon_1,\quad (x_i,y_j)\in \widehat{\Omega}, \\\label{tolctheta}
& \max_{i,j}|c_{i,j}^{n,k+1}|<n\varepsilon_2{\color{black}/N},\quad \text{or} \quad \max_{i,j}|{c}_{i,j}^{n,k+1}-{c}_{i,j}^{n,k}|<\varepsilon_3\quad (x_i,y_j)\in {\Theta}.
\end{align}
In \eqref{tolphitheta} and \eqref{tolctheta} we introduce alternative stopping conditions that rely on the approximations within $\Theta$. In this region, the exact solution is known a priori and is identically zero. Hence, instead of estimating the error indirectly, we can compute it exactly and apply a direct control. 

The first conditions in \eqref{tolphitheta} and \eqref{tolctheta} ensure that, at the final time step ($n=N$), the absolute actual error in $\Theta$ remains below the prescribed tolerance $\varepsilon_2$, accounting for the fact that the accumulated error in $\Theta$ may grow linearly in time, as established in Theorem \ref{th:error} (see \eqref{errth}). 

This provides a practical relaxation of the stopping criterion \eqref{stopth} for large $n$, when the corrosive front is far from $\Theta$, and the influence of the errors inside $\Theta$ on the most critical part of the simulation becomes marginal.

Conversely, in the early stages (small $n$), when the interface is close to $\Theta$, the stopping condition from Theorem~\ref{th:error} remains the dominant one. 

On the other hand, conditions \eqref{tolphiomega} and \eqref{tolcomega} are introduced to extend the error control to $\widehat{\Omega}$, ensuring that the solution in $\widehat{\Omega}$ is only marginally influenced by the residual in $\Theta$, especially when \eqref{stopth} is relaxed by \eqref{tolphitheta} and \eqref{tolctheta}.
\subsection{Implementation Sketch}
This section provides a schematic recap of the workflow for the proposed technique, specifically when problem \eqref{AC}--\eqref{CH} is defined on non-rectangular domains. Algorithm \ref{alg:nonrec} outlines the implementation of the iterative IMEX Euler methods. As noted in Section \ref{sec:alg1}, the bulk of the computation is given by the matrix multiplications. In Algorithm \ref{alg:nonrec}, these need to be performed multiple times per time step. Assuming that a total of $K$ iteration is needed to solve \eqref{ACiter} and \eqref{CHiter} within the desired tolerance, the overall computational cost of the algorithm can be estimated as $\sim 4KNm_xm_y(m_x+m_y)$ floating-point operations.
The procedure for the iterative IMEX 2SBDF methods is analogous, but requires an initialization step due to its two-step nature.

The memory requirements are similar as those of Algorithm \ref{alg:rec}. However, in this case, matrices $\mathcal{N}$ and $\mathcal{G}$ must also be stored. Since these are sparse, with the number of non-zero entries equal to the number of grid points in $\Theta$ at the initial time, the additional memory overheads are minimal. Hence, the algorithm remains memory efficient overall.

\begin{algorithm}[h!]
\caption{
Iterative IMEX Euler methods (non-rectangular domains)}\label{alg:nonrec}
\begin{algorithmic}[1]
\State Initialize the system parameters and functions and the spatial grid on an extended rectangular domain $\Omega$.
\State Define the initial conditions $\boldsymbol{\phi}^0$ and $\mathbf{c}^0$ in $\Omega$ (assign zero values in $\Theta$).
\State Build the matrices $\mathcal{M}_x$ and $\mathcal{M}_y$ in \eqref{genSylv} considering the assigned boundary conditions on $\partial\Omega$.

\State Compute their spectral factorization.
\State Choose a value for $w$ considering Remark\ref{rem:stab}.

\State Define the sparse matrices $\mathcal{N}$ and $\mathcal{G}$. \Comment{Adapt based on IMEX-I or IMEX-E formulas}

\For{$n=1,\ldots N$}\Comment{Advance in time}

\State Set $k=0$, $\boldsymbol{\phi}^{n+1,k}=\boldsymbol{\phi}^{n}$, $\mathbf{c}^{n+1,k}=\mathbf{c}^{n}$, and the tolerances $\varepsilon_1, \varepsilon_2, \varepsilon_3.$

\While{\eqref{tolphiomega}--\eqref{tolphitheta} are not satisfied both} \Comment{Iterations to solve \eqref{ACiter}}

\State Recast the right hand side of \eqref{ACiter} into a matrix of dimension $(m_x,m_y)$.
\State Compute $\boldsymbol{\phi}^{n+1,k+1}$ using formula \eqref{solSyl}.
\EndWhile
\While{\eqref{tolcomega}--\eqref{tolctheta} are not satisfied both}\Comment{Iterations to solve \eqref{CHiter}}

\State Recast the right hand side of \eqref{CHiter} into a matrix of dimension $(m_x,m_y)$.
\State Compute $\mathbf{c}^{n+1,k+1}$ using formula \eqref{solSyl}.
\EndWhile

\EndFor
\end{algorithmic}
\end{algorithm}

\section{Numerical experiments}\label{sec:tests}
In this section we propose some numerical tests to show the performances of the proposed methods in terms of accuracy and efficiency. 
In all tests, the two-step second-order IMEX 2SBDF with step $\Delta t$ is initialized by performing $4/\Delta t$ steps of the first-order IMEX Euler with time step $(\Delta t/2)^2$ to have an accurate approximate solution at time $t=\Delta t$. We compare the proposed IMEX schemes with the numerical method introduced in \citep{GaoFD}, that combines a standard second-order finite difference approximation in space with the second-order accurate exponential Rosenbrock-Euler method in time for the stiff equation \eqref{AC} and with the implicit midpoint rule for the time integration of \eqref{CH}. The two equations are suitably decoupled, by using an explicit second-order approximation for $c$ in \eqref{AC}. This scheme is implemented in \textsc{Matlab} as in \citep{Conte} and the function $\varphi_1$ in the exponential Rosenbrock-Euler method is calculated on matrix arguments by the algorithm \texttt{expmvp.m} in \citep{Niesen}.

We also compare our schemes with a method that couples a second-order finite difference spatial discretization with the renowned two-stage second-order Runge Kutta IMEX scheme proposed in \citep{Ascher}. This method, commonly referred as ARS(2,2,2), is widely used for the solution of stiff equations including phase-field models (see the recent references \citep{Zhang,Fu}) and it has several advantageous properties: it is $L$-stable, stiffly accurate in both its implicit and explicit components, and its implicit part is only diagonally implicit with identical diagonal elements. Consequently, its computational cost is limited, since the numerical solution coincides with the second stage, and both stages require solving linear systems with the same sparse coefficient matrix. Therefore, an LU factorization of this matrix can be computed once at the beginning, after which only the solution of four narrow-banded triangular systems (two per stage) is needed for each PDE at every time step.

It is crucial to note that due to the stiffness of the reaction term $F_1$, also this method is unstable for equation \eqref{AC}, unless tiny time steps are used. Thus, as done for the IMEX schemes proposed in this paper, in order to use larger steps, we apply this method to system \eqref{AC}-\eqref{CH} after regularizing the first equation as in \eqref{relgen}.

All tests have been performed in \textsc{Matlab} R2024b on a 64-bit Windows desktop computer with a 2.10 GHz 12th Gen. Intel(R) Core(TM) i7-12700 Processor and 16.0 GB Memory. 
\subsection{Pencil electrode test}\label{sec:pencil2D}
As a first benchmark problem we consider a variant of the pencil electrode test in \citep{Ernst,GaoFD,GaoFE,mai}, already proposed in \citep{Conte}. The specimen is a 300 $\mu$m long and 25 $\mu$m wide metal wire mounted on an epoxy coating which leaves only the two ends of the wire exposed to an electrolyte solution.

This setting is modeled by system \eqref{AC}--\eqref{hg} on a rectangular domain $\Omega=[0,25]\times[0,300]$ with homogeneous Dirichlet boundary conditions at the wire ends,
$$\phi(x,0,t)=c(x,0,t)=\phi(x,300,t)=c(x,300,t)=0,$$
homogeneous Neumann boundary conditions on the rest of the boundary,
$$\frac{\partial}{\partial n}\phi(0,y,t)=\frac{\partial}{\partial n}c(0,y,t)=\frac{\partial}{\partial n}\phi(25,y,t)=\frac{\partial}{\partial n}c(25,y,t)=0,$$
and the constant initial condition
$$\phi(x,y,0)=c(x,y,0)=1$$
in the interior of $\Omega$. As in \citep{Conte} we integrate till the final time $T=225$s. 

To measure the accuracy of the different methods we compute a reference solution $(\phi_{\text{ref}},c_{\text{ref}})$, by using the method in \citep{Conte,GaoFD} on a fine grid with steps
$$\Delta x=\Delta y=0.5\mu\text{m},\qquad \Delta t=10^{-4}\text{s}$$
and tolerance $10^{-8}$ for the computation of the matrix function $\varphi_1$.  We observe that halving the time step does not produce any effect on the values of the errors that we report in this section, so we can consider the solution obtained in this way a good reference to compare the accuracy of the different methods. The errors in the solution $(\phi,c)$ are computed at the final time as 
$$\text{Err}_\phi=\frac{\|\phi-\phi_{\text{ref}}\|}{\|\phi_{\text{ref}}\|},\qquad \text{Err}_c=\frac{\|c-c_{\text{ref}}\|}{\|c_{\text{ref}}\|},$$
where $\|\cdot\|$ is the Euclidean vector norm.

We first solve the problem by IMEX Euler method \eqref{finAC}--\eqref{finCH} with
$$
\Delta x=\Delta y=1\mu\text{m},\qquad \Delta t=10^{-3}\text{s},\qquad w=4.43\cdot 10^8.
$$
The value of $w$ has been obtained by suitably estimating (a priori) the module of the largest eigenvalue of the Jacobian matrix $\partial_{\boldsymbol{\phi}}F_1(\boldsymbol{\phi},\mathbf{c})$ on the solution along the full integration time and space domain. With this value of $w$ condition \eqref{fov} with \eqref{AwBw2} is satisfied at each time step. 
\begin{figure}[t]
\centerline{
\includegraphics[scale=0.45]{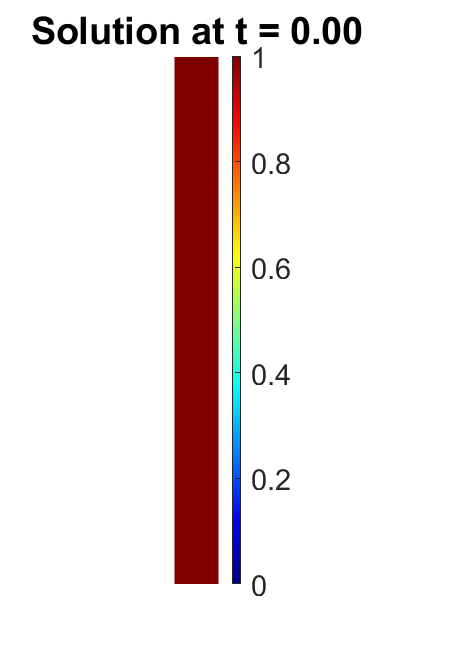}
\includegraphics[scale=0.45]{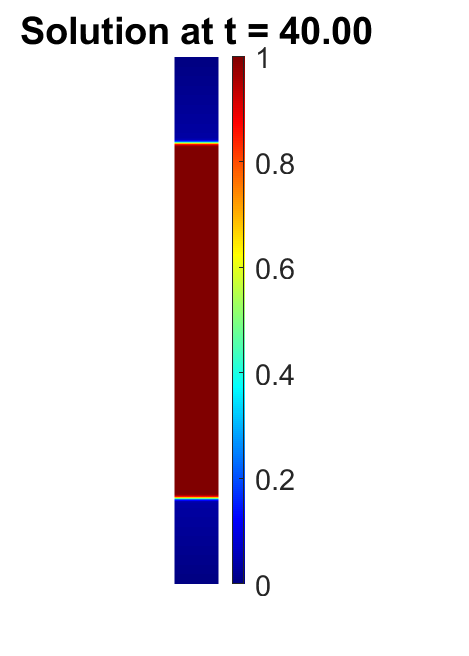}
\includegraphics[scale=0.45]{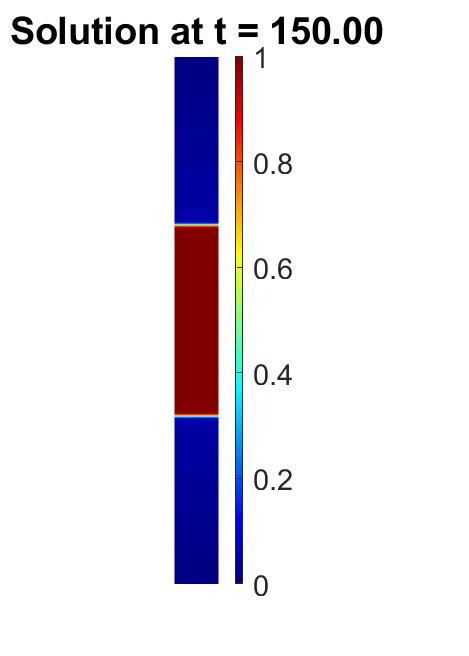}
\includegraphics[scale=0.45]{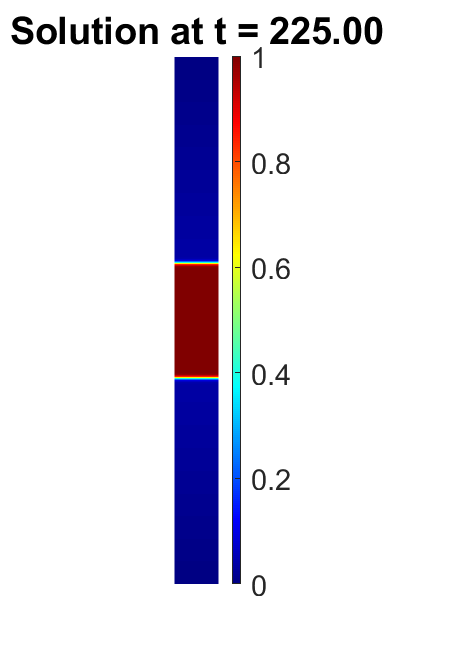}
}
\caption{Pencil electrode test. Configuration of $c$ given by IMEX Euler with $\Delta x=\Delta y=1\mu\text{m},$ and $\Delta t=10^{-3}\text{s}$ at different times.}\label{fig:pencil}
\end{figure}
\begin{figure}[t]
\centerline{
\includegraphics[scale=0.45]{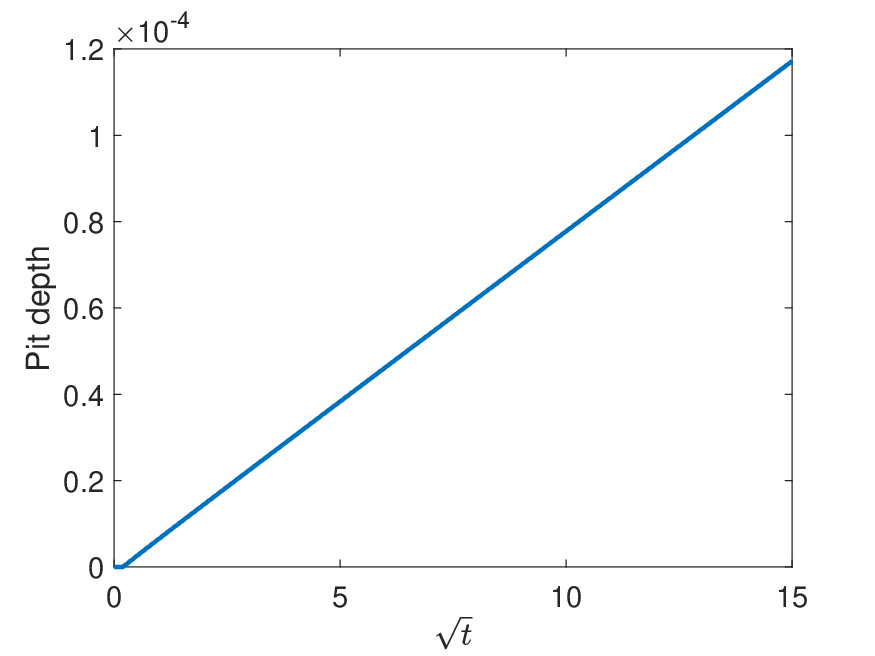}
}
\caption{Pencil electrode test. Position of the lower corrosion front in the solution computed by IMEX Euler with $\Delta x=\Delta y=1\mu\text{m},$ and $\Delta t=10^{-3}\text{s}$ at different times.}\label{fig:pitpencil}
\end{figure}
In Figure \ref{fig:pencil} we show the solution at times $t=0, 40, 150$ and $225$. Figure \ref{fig:pitpencil} shows that the corrosive front advances with a velocity that is proportional to the square root of time, in agreement with the chemical theory of corrosion and the laboratory and computational results in \citep{Ernst,GaoFD,GaoFE,mai,Conte}.

The analogue figures relative to the solutions of IMEX 2SBDF, of ARS(2,2,2), and of the method in \citep{GaoFD} on the same grid do not show any visible difference, and so we omit them.

In order to make some quantitative comparisons, we first calculate the error in the solution of IMEX Euler method at different times. The results are reported on the left of Table \ref{tab:pencil}.

We now solve the system on the same spatial grid, first with IMEX 2SBDF and ARS(2,2,2) using the same value of $w$ as above, and finally with the method in \citep{GaoFD} with tolerance of $10^{-4}$ for the computation of function $\varphi_1$ in the exponential Rosenbrock-Euler method. The time steps are chosen in order to obtain solution errors comparable to those given by IMEX Euler. They are $\Delta t=20\cdot 10^{-3}$s for IMEX 2SBDF, $\Delta t = 3\cdot{10}^{-3}$s for ARS(2,2,2), and $\Delta t=4\cdot{10}^{-3}$s for the method in \citep{GaoFD}.
\begin{table}
\begin{center}
\begin{tabular}{||c||c|c|c||c|c|c||c|c|c||c|c|c||}
\hline
\hline
&\multicolumn{3}{|c||}{IMEX Euler \tabstretch} & \multicolumn{3}{c||}{IMEX 2SBDF} & \multicolumn{3}{c||}{Method in \citep{GaoFD}} & \multicolumn{3}{c||}{ARS(2,2,2)}\\
&\multicolumn{3}{|c||}{$\Delta t=1\cdot 10^{-3}$s \tabstretch} & \multicolumn{3}{c||}{$\Delta t=20\cdot 10^{-3}$s} & \multicolumn{3}{c||}{$\Delta t=4\cdot 10^{-3}$s} & \multicolumn{3}{c||}{$\Delta t=3\cdot 10^{-3}$s}\\
\hline
$t$ (s)& $\text{Err}_\phi$\tabstretch & $\text{Err}_c$ & El. time & $\text{Err}_\phi$\tabstretch & $\text{Err}_c$ & El. time & $\text{Err}_\phi$\tabstretch & $\text{Err}_c$ & El. time & $\text{Err}_\phi$\tabstretch & $\text{Err}_c$ & El. time \\
\hline
1&0.044&0.053&0.20&0.047&0.057&0.07&0.044&0.052&9.27 & 0.043 & 0.052 & 0.66\\
20&0.041&0.048&3.75&0.027&0.032&0.30&0.046&0.054&142.21 & 0.042 & 0.050 & 12.13\\
40&0.042&0.050&7.68&0.027&0.032&0.52&0.048&0.057&256.92 & 0.044 & 0.052 & 24.26\\
70&0.045&0.053&14.04&0.029&0.034&0.88&0.052&0.061&427.42 & 0.047 & 0.056 &42.04\\
100&0.049&0.058&20.13&0.031&0.036&1.11&0.056&0.066&597.91 & 0.051 & 0.061 & 60.03\\150&0.056&0.067&30.27&0.035&0.042&1.78&0.064&0.076&859.14 & 0.059 & 0.070 & 89.69\\
225&0.073&0.086&46.83&0.045&0.053&2.47&0.084&0.099&1217.50 & 0.077& 0.091&139.12\\
\hline
\hline
\end{tabular}
\end{center}
\caption{Pencil electrode test. 2-norm relative error in $\phi$ and $c$ and elapsed time in seconds for different methods at different times.}\label{tab:pencil}
\end{table}

A comparison of the results is given in Table \ref{tab:pencil}. We observe that IMEX 2SBDF is by far the most efficient of the four methods, yielding the most accurate solutions in a significantly shorter computational time. However, both the proposed IMEX methods are more efficient then the methods from the literature. Note that, as reported in \citep{GaoFD,Conte} for similar experiments, the exponential method advances much slower than the real phenomenon. Thus, accurate predictions can be obtained on highly powerful computers only, in contrast with the methods proposed in this paper (see the times reported in Table \ref{tab:pencil}).

We complete this section with one last experiment to test how the size of the steps affects the computational times of the proposed IMEX methods. We first repeat the experiment, with final time $T=225$s, choosing $\Delta t = 10^{-3}$s and $\Delta t = 10^{-2}$s for IMEX Euler and IMEX 2SBDF, respectively, and compare the computational times as $h=\Delta x=\Delta y$ is progressively reduced starting from $h=1\mu$m. Next, we reset $\Delta x=\Delta y=1\mu$m and compare the computational time for different values of the time steps.
 
The results, shown in Figure \ref{fig:dtvsCPU}, indicate that the computational time scales linearly with the time step and roughly cubically with the grid resolution $h$. This is in agreement with the observations made in Section \ref{sec:alg1}.

\begin{figure}[t]
\begin{tikzpicture}
\begin{axis}[
    xmode=log, 
    ymode=log, 
    xlabel={Space step}, 
    ylabel={CPU time}, 
    grid=major, 
    width=6.5cm, 
    height=5cm, 
    xtick={0.00000025,0.0000005, 0.00000075,0.000001}, 
    xticklabels={$0.25h$,$0.5h$, $0.75h$,$1h$}, 
    ticklabel style={/pgf/number format/fixed}, 
    ymax=10000,
]
\addplot coordinates {
	(0.00000025, 2014.63)    
    (0.0000005, 277.55)
    (0.00000075, 99.25)
    (0.000001, 46.83)
};
\addplot coordinates {
	(0.00000025, 191.64)    
    (0.0000005, 26.37)
    (0.00000075, 9.85)
    (0.000001, 4.78)
};
\addplot[
        black,
        dashed,
        thick,
        domain=0.35e-6:0.7e-6,
    ]
    {1.23125e-17 * x^(-3)};
\node at (axis cs:0.00000085, 50) [anchor=north, font=\small, black] {slope 3};
\end{axis}
\end{tikzpicture}
\begin{tikzpicture}
\begin{axis}[
   xmode=log, 
    ymode=log, 
    xlabel={Time step},
    ylabel={CPU time}, 
	legend style={
                at={(1.05,1)}, 
                anchor=north west,
                cells={align=left},
            },
            clip=false, 
    grid=major, 
    width=6.5cm, 
    height=5cm, 
    xtick={0.5,0.75,1, 2}, 
    xticklabels={$0.5\Delta t$, $0.75\Delta t$, $1 \Delta t$,$2 \Delta t$}, 
    ticklabel style={/pgf/number format/fixed}, 
    ymax=150,
]
\addplot coordinates {    
    (0.5, 91.03)
    (0.75, 61.68)
    (1, 46.83)
    (2,21.89)
};
\addlegendentry{IMEX Euler ($\Delta t = 10^{-3}$s)}
\addplot coordinates {
	(0.5,9.91)	
	(0.75, 6.81)
    (1, 5.09)
    (2, 2.47)
};
\addlegendentry{IMEX 2SBDF ($\Delta t = 10^{-2}$s)}
\addplot[
        black,
        dashed,
        thick,
        domain=0.65:1.25,
    ]
    {15 * x^(-1)};    
\node at (axis cs:1.5, 15) [anchor=north, font=\small, black] {slope 1};
\end{axis}
\end{tikzpicture}
\caption{{\color{black}Pencil electrode test. Effect of the spatial grid (left) and of the time step (right) on the computational times of matrix-oriented IMEX methods. We have set $\Delta t = 10^{-3}$s for IMEX Euler and $\Delta t = 10^{-2}$s for IMEX 2SBDF. In all cases, $h=10^{-6}$m and the final time is $T=225$s.}}
    \label{fig:dtvsCPU}
\end{figure}
\subsection{Circular pit growth}\label{sec:circ}
Next experiment is a variant of the semi-circular pit growth test in \citep{Ernst,GaoFD,GaoFE,mai}. To mimic a common real pitting phenomenon we consider a rectangular specimen with a circular pit in the interior.
We assume that the specimen is 200$\mu$m wide and 100$\mu$m high, that the pit has diameter 4$\mu$m and that it is centered at the middle of the specimen. 
In order to describe this situation, we assign to system \eqref{AC}--\eqref{hg} homogeneous Neumann boundary conditions on the exterior rectangular boundary and homogeneous Dirichlet boundary conditions on the pit surface. As in \citep{Conte} we integrate till the final time $T=100$s. The Laplace operator is discretized by matrix $\mathcal{M}$ in \eqref{kronsum} with \eqref{MN}.

Also in this case a formula for the exact solution is not available, and thus we calculate a reference solution exactly as before, to estimate the accuracy of the methods to compare. Again, we have verified that halving the time step used for the calculation of the reference solution does not alter the values of the errors reported in this section.
\begin{figure}[t]
\centerline{
\includegraphics[scale=0.6]{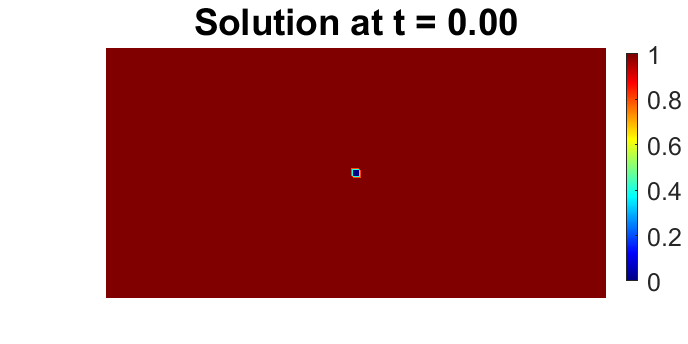}
\includegraphics[scale=0.6]{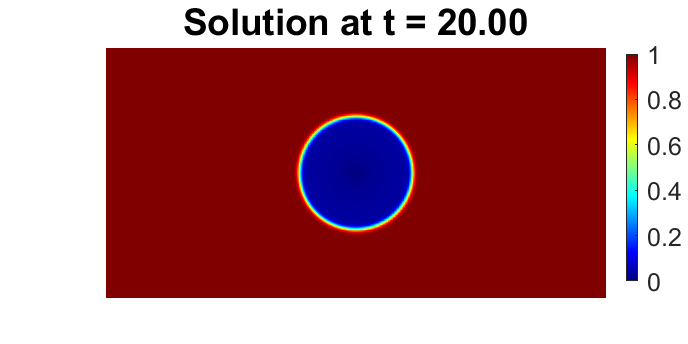}
}
\centerline{
\includegraphics[scale=0.6]{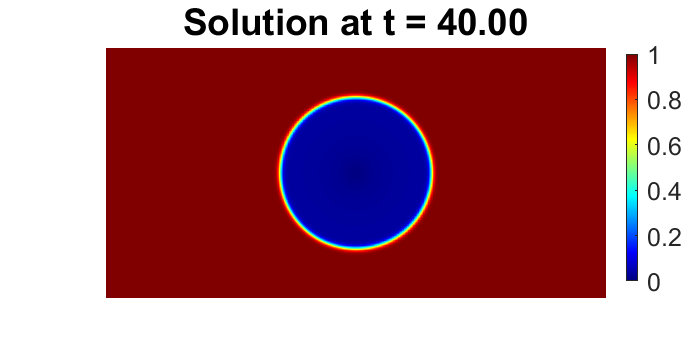}
\includegraphics[scale=0.6]{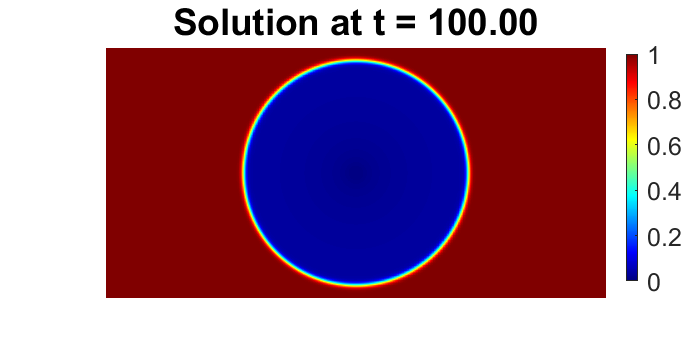}
}
\caption{Circular pit growth. Configuration of $c$ given by iterative IMEX-E Euler with $\Delta x=\Delta y=1\mu\text{m},$ and $\Delta t=2\cdot 10^{-3}\text{s}$ at different times.}\label{fig:circ}
\end{figure}

We first solve this problem by the first-order iterative IMEX-E Euler method \eqref{ACiter}--\eqref{CHiter} with 
$$
\Delta x=\Delta y=1\mu\text{m},\qquad \Delta t=2\cdot 10^{-3}\text{s},\qquad w=4.43\cdot 10^8.
$$
The iterations are performed until satisfying the stopping criteria in  Section \ref{sec:stop} with $\varepsilon_1=10^{-4}$, $\varepsilon_2=10^{-3}$ and $\varepsilon_3=10^{-8}$. Here and henceforth, the tolerances $\varepsilon_2$ and $\varepsilon_3$ are tuned in order to have errors no larger than $10^{-3}$ in $\Theta$. Note that here and in all circumstances we choose $\varepsilon_3<\delta \Delta t$ with $\delta < \Delta t$ according to Remark \ref{rem:delta}.

For this problem bounds on the spectral radius of the iteration matrices of the iterative IMEX-E methods are explicitly given in Corollary~\ref{cor2} and are discussed in Section \ref{sec:valspecrad}.

The graphs of the solution at different times are given in Figure \ref{fig:circ}. As in the previous test, the pit depth growth is directly proportional to the square root of time, in agreement with the results in \citep{GaoFD,GaoFE, Ernst, mai, Conte}. For sake of brevity, we omit the corresponding graph which is analogue to the one reported in Figure \ref{fig:pitpencil} for the pencil electrode test.
\begin{table}
\begin{center}
\begin{tabular}{||c||c|c|c||c|c|c||c|c|c||c|c|c||}
\hline
\hline
&\multicolumn{3}{|c||}{Iter. IMEX-E Euler \tabstretch} & \multicolumn{3}{c||}{Iter. IMEX-E 2SBDF} & \multicolumn{3}{c||}{Method in \citep{GaoFD}} & \multicolumn{3}{c||}{ARS(2,2,2)} \\
&\multicolumn{3}{|c||}{$\Delta t=2\cdot 10^{-3}$s \tabstretch} & \multicolumn{3}{c||}{$\Delta t=6\cdot 10^{-3}$s \tabstretch} & \multicolumn{3}{c||}{$\Delta t=5\cdot 10^{-3}$s} & \multicolumn{3}{c||}{$\Delta t=3\cdot 10^{-3}$s}\\
\hline
$t$ (s)& $\text{Err}_\phi$\tabstretch & $\text{Err}_c$ & El. time & $\text{Err}_\phi$\tabstretch & $\text{Err}_c$ & El. time & $\text{Err}_\phi$\tabstretch & $\text{Err}_c$ & El. time & $\text{Err}_\phi$\tabstretch & $\text{Err}_c$ & El. time \\
\hline
1&0.031&0.036&0.76&0.017&0.020&1.62&0.025&0.030&22.26 & 0.023  & 0.027 & 15.98\\
20&0.059&0.068&11.84&0.048&0.056&10.02&0.056&0.065&624.66 & 0.054 & 0.063 & 284.05\\
40&0.079&0.091&24.41&0.068&0.079&17.27&0.075&0.087&1219.02 & 0.074 & 0.086  & 562.56 \\
70&0.106&0.121&42.44&0.095&0.109&29.08&0.113&0.098&2055.81 & 0.101  & 0.116  & 987.93\\
100&0.132&0.150&59.19&0.120&0.137&39.00&0.121&0.138&2809.03 & 0.127 & 0.144 & 1380.19\\
\hline
\hline
\end{tabular}
\end{center}
\caption{Circular pit growth. 2-norm relative error in $\phi$ and $c$ and elapsed time in seconds for different methods at different times.}\label{tab:circ}
\end{table}

In Table \ref{tab:circ} we compare the efficiency of iterative IMEX-E Euler, iterative IMEX-E 2SBDF, the method in \citep{GaoFD}, and ARS(2,2,2).  The choice of the different time steps is made in order to obtain similar solution errors. As in the previous experiment, we set tolerance $10^{-4}$ when evaluating function $\varphi_1$ in the implementation of the exponential method, whereas we choose $\varepsilon_3=3\cdot 10^{-8}$ in the stopping criteria \eqref{tolphiomega}--\eqref{tolctheta} for the iterative IMEX-E 2SBDF. The same value of $w$ is used for all IMEX schemes.

We observe that the methods in the literature are the least efficient. 
The two iterative IMEX schemes perform similarly for short time simulations. However, the second-order accurate method is the most efficient over medium/long time spans. The reason for this improved performance is a drastic reduction in the number of iterations to fulfill the stopping criteria \eqref{tolphiomega}--\eqref{tolctheta} as time increases (see Figure \ref{fig:iter}). We observe that while both methods fulfill \eqref{tolphiomega}--\eqref{tolphitheta} after no more than four iterations, satisfying \eqref{tolcomega}--\eqref{tolctheta} is more demanding. However, the number of iterations required by Iter. IMEX-E Euler is relatively small, whereas more than 40 iterations are performed at the first steps of Iter. IMEX-E 2SBDF. Our guess is that this is related to the lack of regularity of the initial condition, that has jumps. As time advances and the solution is smoothed, the number of iterations decreases, and it stays below 15 after time $t=10$s. We observe that both iterative IMEX-E methods make actual predictions on our machine, since their computational times are lower than the effective time of the experiment.
\begin{figure}[t]
\centerline{
\includegraphics[scale=0.45]{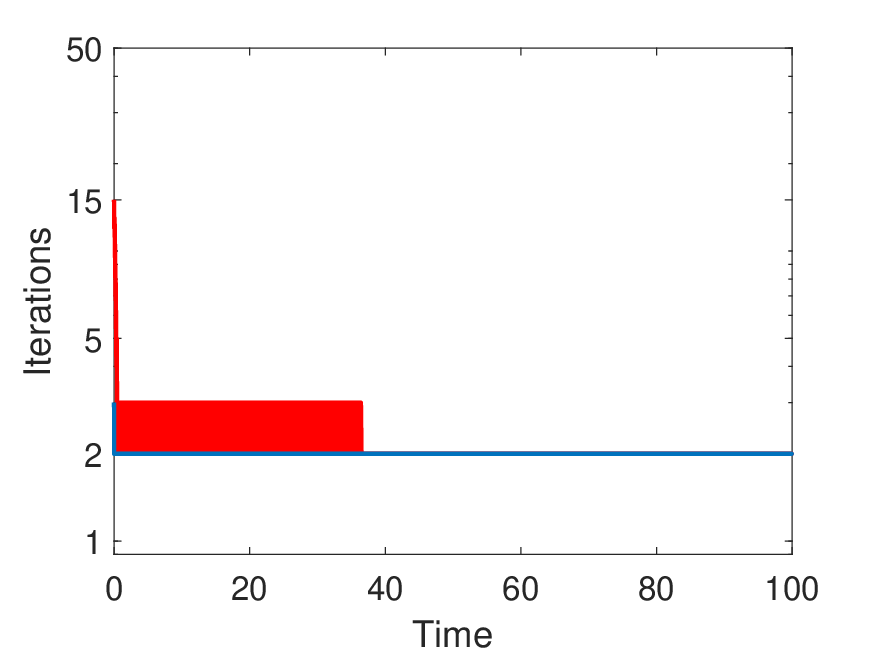}
\includegraphics[scale=0.45]{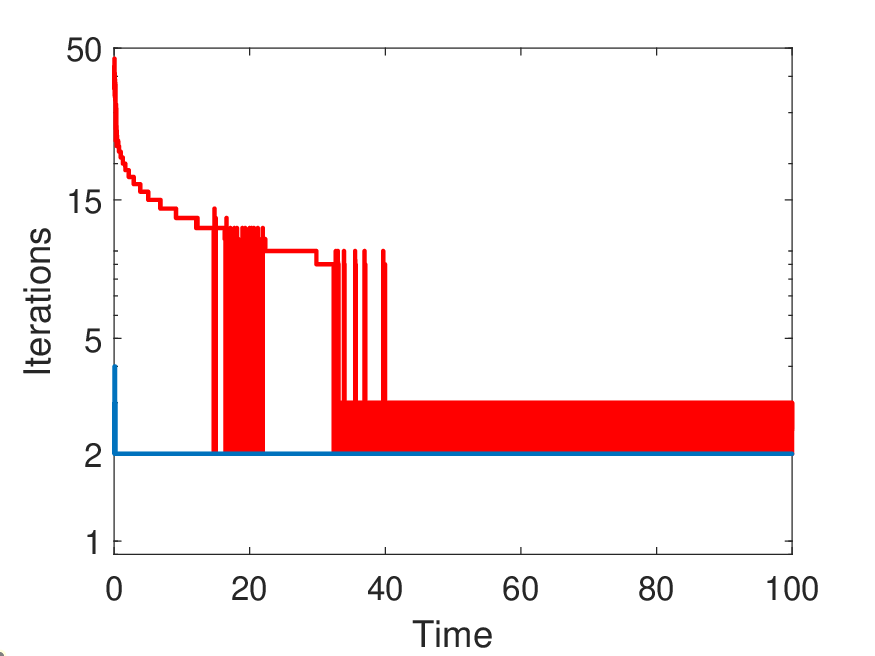}
}
\caption{Number of iterations per time step in Iter. IMEX-E Euler (left) and Iter. IMEX-E 2SBDF (right) to fulfill \eqref{tolphiomega}--\eqref{tolphitheta} (blue line) and \eqref{tolcomega}--\eqref{tolctheta} (red line).}\label{fig:iter}
\end{figure}
\begin{table}[t]
\begin{center}
\begin{tabular}{||c||c|c|c||c|c|c||}
\hline
\hline
&\multicolumn{3}{|c||}{Iter. IMEX-I Euler \tabstretch} & \multicolumn{3}{c||}{Iter. IMEX-I 2SBDF} \\
&\multicolumn{3}{|c||}{$\Delta t=2\cdot 10^{-3}$s \tabstretch} & \multicolumn{3}{c||}{$\Delta t=6\cdot 10^{-3}$s} \\
\hline
$t$ (s)& $\text{Err}_\phi$\tabstretch & $\text{Err}_c$ & El. time & $\text{Err}_\phi$\tabstretch & $\text{Err}_c$ & El. time\\
\hline
1&0.031&0.036&2.81&0.017&0.020&2.18\\
20&0.059&0.068&20.63&0.048&0.056&11.46\\
40&0.079&0.091&39.05&0.068&0.079&23.38\\
70&0.106&0.121&67.94&0.095&0.109&43.53\\
100&0.132&0.150&94.15&0.120&0.137&62.42\\
\hline
\hline
\end{tabular}
\end{center}
\caption{Circular pit growth. 2-norm relative error in $\phi$ and $c$ and elapsed time in seconds for iterative IMEX-I methods at different times.}\label{tab:circII}
\end{table}

We now solve the same problem by iterative IMEX-I Euler and iterative IMEX-I 2SBDF using the same time steps and stopping criteria of iterative IMEX-E Euler and iterative IMEX-E 2SBDF, respectively. The number of iterations to fulfill \eqref{tolphiomega}--\eqref{tolctheta} are shown in Figure~\ref{fig:iter_I}. In agreement with the convergence analysis in Section \ref{sec:iterative} (see also Section \ref{sec:valspecrad}), the convergence of these iterative procedures is slower.  For completeness, we note that the average iteration time required to satisfy \eqref{tolphiomega}--\eqref{tolphitheta} (blue lines in Figure \ref{fig:iter} and \ref{fig:iter_I}) and \eqref{tolcomega}--\eqref{tolctheta} (red lines in Figure \ref{fig:iter} and \ref{fig:iter_I}) is about $2 \cdot 10^{-4}$s for all methods.

From the figures in Table \ref{tab:circII} we observe that the accuracy provided by these methods is the same of iterative IMEX-E Euler and iterative IMEX-E 2SBDF, respectively (compare with Table \ref{tab:circ}), but the overheads due to the extra iterations increase meaningfully the computational time. However, the computational time of both iterative IMEX-I methods are inferior to the physical time $T=100$s.
\begin{figure}[t]
\centerline{
\includegraphics[scale=0.45]{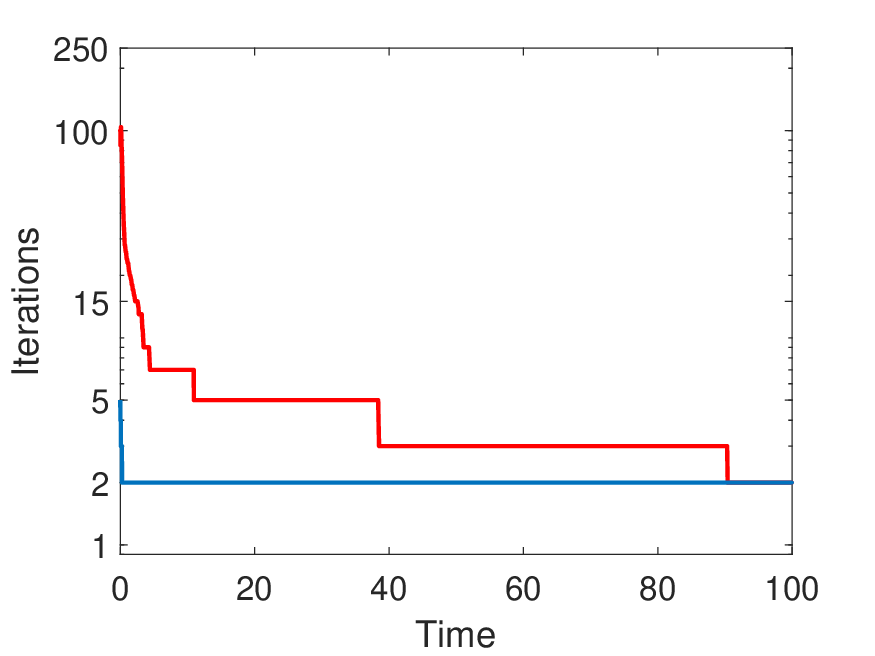}
\includegraphics[scale=0.45]{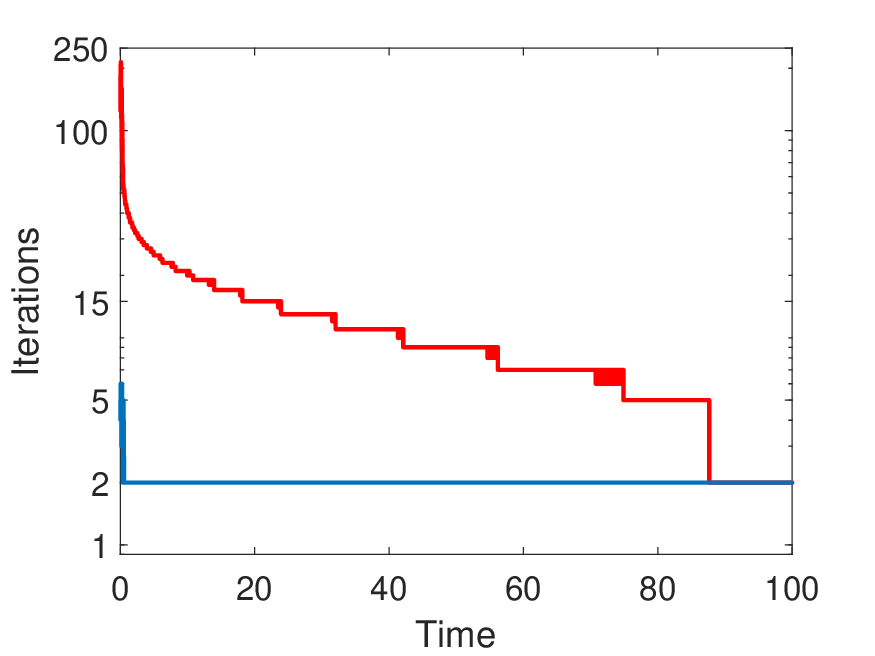}
}
\caption{Number of iterations per time step in Iter. IMEX-I Euler (left) and Iter. IMEX-I 2SBDF (right) to fulfill \eqref{tolphiomega}--\eqref{tolphitheta} (blue line) and \eqref{tolcomega}--\eqref{tolctheta} (red line).}\label{fig:iter_I}
\end{figure}
\begin{table}[t]
\begin{center}
\begin{tabular}{||c|c||c|c||c|c||c|c||}
\hline
\hline
\multicolumn{2}{||c||}{Iter. IMEX-E Euler \tabstretch} & \multicolumn{2}{c||}{Iter. IMEX-E 2SBDF} & \multicolumn{2}{c||}{Iter. IMEX-I Euler \tabstretch}& \multicolumn{2}{c||}{Iter. IMEX-I 2SBDF}\\
\multicolumn{2}{||c||}{$\Delta t=2\cdot 10^{-3}$s \tabstretch} & \multicolumn{2}{c||}{$\Delta t=6\cdot 10^{-3}$s \tabstretch} &\multicolumn{2}{|c||}{$\Delta t=2\cdot 10^{-3}$s \tabstretch} & \multicolumn{2}{c||}{$\Delta t=6\cdot 10^{-3}$s \tabstretch}\\
\hline
$\max_\Theta |\phi|$\tabstretch & $\max_\Theta |c|$ & $\max_\Theta |\phi|$\tabstretch & $\max_\Theta |c|$ & $\max_\Theta |\phi|$\tabstretch & $\max_\Theta |c|$ & $\max_\Theta |\phi|$\tabstretch & $\max_\Theta |c|$  \\
\hline
6.89e-14&7.00e-04&6.75e-14&1.00e-03&6.81e-14&9.25e-04&7.13e-14&8.89e-04\\
\hline
\hline
\end{tabular}
\end{center}
\caption{Circular pit growth. Maximum value of $|\phi|$ and $|c|$ in the interior of the initial pit at time $T=100$s.}\label{tab:pit}
\end{table}

As the main difference between Iter. IMEX-E and Iter. IMEX-I relies in the treatment of the internal part of the initial circular pit, where the exact solution is identically zero, in Table \ref{tab:pit} we show the maximum value attained by the module of the dependent variables in this region at the final time $T=100$s. We observe that the errors are no larger than $10^{-3}$ as required by our choice of $\varepsilon_2$ and $\varepsilon_3$. However, while the values of $\phi$ are of the order of the machine precision, numerical errors are evident in $c$.

On the basis of these results and in order to improve the efficiency of the methods, we observe that:
\begin{enumerate}
\item Choosing small values of $\varepsilon_2$ and $\varepsilon_3$ allows for small values of $\phi$ and $c$ in $\Theta$. This is a desirable property from a mathematical and physical point of view. However, $\Theta$ is not part of the original domain and, from a practical perspective, the results in this region may be disregarded at the end of the computations. Thus, if relative large errors in $\Theta$ do not significantly effect the solution in $\widehat\Omega$, one may require only low accuracy in this region. 
\item Only few iterations are needed to satisfy the stopping criteria \eqref{tolphiomega}--\eqref{tolphitheta} (see Figures \ref{fig:iter}-\ref{fig:iter_I}). 
\end{enumerate}
Based on these remarks, we now repeat the experiments above by performing only a single iteration in equations \eqref{ACiter} and \eqref{ACiter2}, and using the following unique criterion
\begin{equation}\label{eq:stop2}
 \max_{i,j}|c_{i,j}^{n,k}-c_{i,j}^{n,k-1}|<\varepsilon_1,\quad (x_i,y_j)\in {\Omega}
\end{equation}
to stop the iterations in equations \eqref{CHiter} and \eqref{CHiter2}, with $\varepsilon_1=10^{-4}$ as before.
\begin{table}[t]
\begin{center}
\begin{tabular}{||c|c||c||c||c||}
\hline
\hline
Method & $\Delta t$ & $\max_\Theta |\phi| $ & $\max_\Theta |c|$ & Elaps. time\\
\hline
\hline
{Iter. IMEX-E Euler\tabstretch} & $2\cdot 10^{-3}$s  & {7.81e-14} & {0.0054} & 45.71\\
\hline
{Iter. IMEX-E 2SBDF\tabstretch} & $6\cdot 10^{-3}$s  & {0.2138} & {0.1529} & 20.45\\
\hline
{Iter. IMEX-I Euler\tabstretch} & $2\cdot 10^{-3}$s & {0.0646} & {0.0672} & 43.73\\
\hline
{Iter. IMEX-I 2SBDF\tabstretch} & $6\cdot 10^{-3}$s  & {1.42e-13} & {0.0226} & 19.79\\
\hline
\hline
\end{tabular}
\end{center}
\caption{Circular pit growth. Maximum value of $|\phi|$ and $|c|$ in the interior of the initial pit at time $T=100$s and computational time of the iterative methods with stopping criterion \eqref{eq:stop2}.}\label{tab:fin}
\end{table}

In Table \ref{tab:fin} we report the computational time in seconds and the errors in $\Theta$ at the final time $T=100$s.
Despite the reported errors are larger than before, these do not significantly affect the overall accuracy in the solution: the relative errors on $\Omega$ are the same as reported in Table \ref{tab:circ} and in Table \ref{tab:circII}. We observe that methods of the same order have similar computational time, since the stopping criterion \eqref{eq:stop2} is satisfied after very few iterations at each time by all methods. Moreover, all methods have computational times substantially lower than 100 seconds, and thus all methods are able to predict in advance the corrosion of the considered system. However, the second-order accurate methods are the most efficient.

\begin{figure}[t]
\centering
\begin{tabular}{cc}
\begin{tikzpicture}[scale=0.9]
\begin{axis}[
    xmode=log, 
    ymode=log, 
    log basis x=10,
    log basis y=10,
    xlabel={$\Delta t$}, 
    ylabel={$\rho_\phi$},
    title={IMEX-I Euler}, 
	grid=major, 
    ymin=1e-5,
    ymax=0.3,
    xmin=1e-13,
    xmax=1,
    minor grid style={gray!30, dotted}, 
    major grid style={gray!60}, 
    width=6cm, 
    height=5.3cm, 
]
\addplot coordinates {
  	(1.000000000000000e-01, 8.637471004607730e-02)
    (1.000000000000000e-02, 8.637469403530627e-02)
    (1.000000000000000e-03, 8.637453392792160e-02)
    (1.000000000000000e-04, 8.637293288672544e-02)
    (1.000000000000000e-05, 8.635692573896318e-02)
    (1.000000000000000e-06, 8.619718001089438e-02)
    (1.000000000000000e-07, 8.463164083467740e-02)
    (1.000000000000000e-08, 7.162340523962354e-02)
    (1.000000000000000e-09, 2.823202411123258e-02)
    (1.000000000000000e-10, 3.999940450875542e-03)
    (1.000000000000000e-11, 4.173918775305153e-04)
    (1.000000000000000e-12, 4.192152722415401e-05)
};
\addplot coordinates {
  	(1.0e-1, 0.111613366953889)
    (1.0e-2, 0.111613343645773)
    (1.0e-3, 0.111613110565149)
    (1.0e-4, 0.111610779812448)
    (1.0e-5, 0.111587477638210)
    (1.0e-6, 0.111354989935675)
    (1.0e-7, 0.109082307060108)
    (1.0e-8, 0.090592888158836)
    (1.0e-9, 0.033615208615712)
    (1.0e-10, 0.004611506115792)
    (1.0e-11, 0.000478960822805)
    (1.0e-12, 0.000048081780106)
};
\draw[black!50,line width=0.3pt]
    foreach \x in {1e-11,1e-7,1e-3} {
        (axis cs:\x,1e-5) -- ++(0,2.5pt)
        (axis cs:\x,0.3) -- ++(0,-2.5pt)
    };
\end{axis}
\end{tikzpicture}
&
\begin{tikzpicture}[scale=0.9]
\begin{axis}[
    xmode=log, 
    ymode=log, 
    log basis x=10, 
    log basis y=10,
    xlabel={$\Delta t$}, 
    ylabel={$\rho_c$}, 
    title={IMEX-I Euler}, 
	legend style={
                at={(1.05,1)}, 
                anchor=north west,
            },
            clip=false, 
	grid=both, 
    ymin=1e-8,
    ymax=40,
    xmax=1,
    minor grid style={gray!30, dotted}, 
    major grid style={gray!60}, 
    width=6cm, 
    height=5.3cm, 
    ytick={1e-8,1e-6,1e-4,1e-2,1e0},
]
\addplot coordinates {
  (1e-1, 3.091910110791370e+00)
  (1e-2, 1.323492574989924e+00)
  (1e-3, 8.535416915860942e-01)
  (1e-4, 3.706844524002007e-01)
  (1e-5, 5.592342166978443e-02)
  (1e-6, 5.893591694358645e-03)
  (1e-7, 5.925534844986371e-04)
  (1e-8, 5.928748461707247e-05)
  (1e-9, 5.929070017585222e-06)
  (1e-10, 5.929102175116261e-07)
};
\addlegendentry{ Actual spectral radius }
\addplot coordinates {
  (1e-4, 8.192771084337350e-01)
  (1e-5, 6.917599186164801e-02)
  (1e-6, 6.811579685465291e-03)
  (1e-7, 6.801156196553414e-04)
  (1e-8, 6.800115601965233e-05)
  (1e-9, 6.800011560019652e-06)
  (1e-10, 6.800001156000196e-07)
};
\addlegendentry{ Upper bound in \eqref{eq:upbN} }
\draw[black!50,line width=0.3pt]
    foreach \x in {1e-10,1e-9,1e-7,1e-6,1e-4,1e-3,1e-1} {
        (axis cs:\x,1e-8) -- ++(0,2.5pt)
        (axis cs:\x,40) -- ++(0,-2.5pt)
    };

\draw[black!50,line width=0.3pt]
    foreach \y in {1e-7,1e-6,1e-4,1e-3,1e-1,1} {
        (axis cs:1e-11,\y) -- ++(2.5pt,0)
        (axis cs:1,\y) -- ++(-2.5pt,0)
    };
\end{axis}
\end{tikzpicture}\\
\begin{tikzpicture}[scale=0.85]
\begin{axis}[
    xmode=log, 
    ymode=log, 
    log basis x=10, 
    log basis y=10,
    xlabel={$\Delta t$}, 
    ylabel={$\rho_\phi$},
    title={IMEX-E Euler},
	grid=major, 
    ymin=1e-12,
    ymax=0.3,
    xmin=1e-13,
    xmax=1,
    minor grid style={gray!30, dotted}, 
    major grid style={gray!60}, 
    width=6cm, 
    height=5.3cm, 
]
\addplot coordinates {
  	(1.0e-1, 3.328539998948901e-04)
    (1.0e-2, 3.328538710837327e-04)
    (1.0e-3, 3.328525829762772e-04)
    (1.0e-4, 3.328397023135201e-04)
    (1.0e-5, 3.327109368536959e-04)
    (1.0e-6, 3.314273872405369e-04)
    (1.0e-7, 3.189909101822185e-04)
    (1.0e-8, 2.255072933040868e-04)
    (1.0e-9, 3.361066837631665e-05)
    (1.0e-10, 6.599943809386105e-07)
    (1.0e-11, 7.163711442551826e-09)
    (1.0e-12, 7.224033912371286e-11)
};
\addplot coordinates {
  	(1.0e-1, 4.184151222630537e-03)
    (1.0e-2, 4.184149582066146e-03)
    (1.0e-3, 4.184133176475385e-03)
    (1.0e-4, 4.183969125881926e-03)
    (1.0e-5, 4.182329151208647e-03)
    (1.0e-6, 4.165982376444751e-03)
    (1.0e-7, 4.007661771504543e-03)
    (1.0e-8, 2.821594670673185e-03)
    (1.0e-9, 4.150416017749488e-04)
    (1.0e-10, 8.092542283922112e-06)
    (1.0e-11, 8.774769133424373e-08)
    (1.0e-12, 8.847704834164351e-10)
};
\draw[black!50,line width=0.3pt]
    foreach \x in {1e-11,1e-7,1e-3} {
        (axis cs:\x,1e-12) -- ++(0,2.5pt)
        (axis cs:\x,0.3) -- ++(0,-2.5pt)
    };

\draw[black!50,line width=0.3pt]
    foreach \y in {1e-11,1e-10,1e-8,1e-7,1e-5,1e-4,1e-2,1e-1} {
        (axis cs:1e-13,\y) -- ++(2.5pt,0)
        (axis cs:1,\y) -- ++(-2.5pt,0)
    };
\end{axis}
\end{tikzpicture}
&
\begin{tikzpicture}[scale=0.9]
\begin{axis}[
    xmode=log,
    ymode=log,
    log basis x=10,
    log basis y=10,
    xlabel={$\Delta t$}, 
    ylabel={$\rho_c$}, 
    title={IMEX-E Euler}, 
	legend style={
                at={(1.05,1)}, 
                anchor=north west,
            },
            clip=false, 
	grid=both, 
    ymin=1e-15,
    ymax=50,
    xmax=1,
    minor grid style={gray!30, dotted}, 
    major grid style={gray!60}, 
    width=6cm, 
    height=5.3cm, 
]
\addplot coordinates {
  (1e-1, 3.064979496668973e+00)
  (1e-2, 1.221432377218188e+00)
  (1e-3, 2.041390772076423e-01)
  (1e-4, 8.881254647181853e-03)
  (1e-5, 1.353817561683699e-04)
  (1e-6, 1.435249634836683e-06)
  (1e-7, 1.444018161547574e-08)
  (1e-8, 1.444901747617482e-10)
  (1e-9, 1.444990174072320e-12)
  (1e-10, 1.444999017407002e-14)
};
\addlegendentry{ Actual spectral radius }
\addplot coordinates {
  (1e-4, 1.269188424527285e-01)
  (1e-5, 1.685732679484615e-03)
  (1e-6, 1.760796631151032e-05)
  (1e-7, 1.768854427882198e-07)
  (1e-8, 1.769666088236080e-09)
  (1e-9, 1.769747313470006e-11)
  (1e-10, 1.769755436585778e-13)
};
\addlegendentry{ Upper bound in \eqref{eq:upbNE} }
\draw[black!50,line width=0.3pt]
    foreach \x in {1e-10,1e-9,1e-7,1e-6,1e-4,1e-3,1e-1} {
        (axis cs:\x,1e-15) -- ++(0,2.5pt)
        (axis cs:\x,5e1) -- ++(0,-2.5pt)
    };

\draw[black!50,line width=0.3pt]
    foreach \y in {1e-14,1e-13,1e-12,1e-11,1e-9,1e-8,1e-7,1e-6,1e-2,1e-4,1e-3,1e-1} {
        (axis cs:1e-11,\y) -- ++(2.5pt,0)
        (axis cs:1,\y) -- ++(-2.5pt,0)
    };
\end{axis}
\end{tikzpicture}
\end{tabular}
\caption{Actual spectral radii $\rho_\phi$ (left column) and $\rho_c$ (right column) and bounds given in \eqref{eq:upbN} for IMEX-I (top row) and in \eqref{eq:upbNE} for IMEX-E (bottom row). The comparison is performed for $\Delta x = \Delta y = h = 1\mu$m and different values of $\Delta t$.}\label{fig:specrad}
\end{figure}

\subsubsection{Validating the theoretical bounds}\label{sec:valspecrad}
In this section, we compute the actual spectral radii of the iteration matrices for the circular pit growth experiment, using different step sizes to validate the bounds on $\rho_\phi$ and $\rho_c$ in Corollaries \ref{corrhoneu} and \ref{corcircN} for the iterative IMEX-I and IMEX-E methods, respectively, under boundary conditions on $\partial \Omega$ of different type. First, we fix the value $\Delta x=\Delta y=h=1\mu$m and let $\Delta t$ vary. Next, we fix $\Delta t=10^{-5}$s and let $h$ vary. 

For first-order iterative IMEX Euler methods applied to the Neumann problem, the actual values of $\rho_\phi$ and $\rho_c$, together with their estimated upper bounds, are shown in Figures \ref{fig:specrad} and \ref{fig:specradh}. Similar trends are observed for the IMEX 2SBDF methods (up to a scaling factor) and for the Dirichlet problem, where the actual spectral radii lie slightly, but not significantly, close to the theoretical bounds.

In a general perspective, we observe that 
\begin{itemize}
\item The theoretical upper bounds are very sharp for the iterative IMEX-I methods. 
\item For IMEX-E methods, they reproduce the correct trends, and exceed the actual values of the spectral radii of about one order of magnitude. 
\item The IMEX-E methods have smaller spectral radii than their IMEX-I counterparts, justifying their faster convergence as observed by comparing iteration counts in Figures \ref{fig:iter} and \ref{fig:iter_I}.
\item Figure \ref{fig:specrad} shows that for $\Delta t\in (10^{-3},10^{-2})$, $\rho_\phi$ is small, while $\rho_c$ is large and approaches 1 when $\Delta t \sim 10^{-2}$. This aligns with the iteration counts in Figures \ref{fig:iter} and \ref{fig:iter_I}, where the fewer iterations required to meet the stopping criteria \eqref{tolphiomega}--\eqref{tolphitheta} (blue lines) correspond to small $\rho_\phi$. In contrast, the larger number of iterations needed to satisfy the stopping criteria \eqref{tolcomega}--\eqref{tolctheta} (red lines) correspond to large $\rho_c$.
\end{itemize}
Finally, we note that in the case of Neumann boundary conditions, the theoretical upper bound on $\rho_c$ is not shown in Figure \ref{fig:specrad} for the largest values of $\Delta t$, as the corresponding stepsizes violate the second condition in \eqref{eq:cond1N}, and the bound cannot be computed. Consistent with Remark \ref{rem:IMEIN}, the sufficient conditions for convergence in Theorem \ref{theo:convIN} and Theorem \ref{theo:convEN} are not satisfied in such cases. This aligns with the trend observed in the figures, where the actual spectral radii approach values exceeding 1. For Dirichlet boundary conditions, the bound can be computed for all values of $\Delta t$, but it exceeds 1 for the largest ones, so convergence is not guaranteed by the theory in those cases either.

\begin{figure}
\centering
\begin{tabular}{cc}
\begin{tikzpicture}[scale=0.85]
\begin{axis}[
    xmode=log, 
    ymode=log, 
    log basis x=10, 
    log basis y=10,
    xlabel={$h$}, 
    ylabel={$\rho_\phi$},
    title={IMEX-I Euler},
	grid=major, 
    ymin=1e-5,
    ymax=5,
    xmin=2e-7,
    xmax=1e-4,
    minor grid style={gray!30, dotted}, 
    major grid style={gray!60}, 
    width=6cm, 
    height=5.3cm, 
]
\addplot coordinates {
  	(5.000000000000000e-05, 2.801969347207059e-05)
    (2.500000000000000e-05, 1.048215422504233e-04)
    (1.250000000000000e-05, 4.191383814877552e-04)
    (6.250000000000000e-06, 1.674193432479055e-03)
    (3.125000000000000e-06, 6.659298220763632e-03)
    (1.562500000000000e-06, 3.443953501140443e-02)
    (7.812500000000000e-07, 1.420013370501510e-01)
    (3.906250000000000e-07, 4.104938257701839e-01)
};
\addplot coordinates {
  	(5.000000000000000e-05, 4.342407973211980e-05)
    (2.500000000000000e-05, 1.737019760648229e-04)
    (1.250000000000000e-05, 6.948984331825478e-04)
    (6.250000000000000e-06, 2.781043139620443e-03)
    (3.125000000000000e-06, 1.114742365804826e-02)
    (1.562500000000000e-06, 4.496563285178135e-02)
    (7.812500000000000e-07, 1.861399571358377e-01)
    (3.906250000000000e-07, 8.653697438075987e-01)
};
\draw[black!50,line width=0.3pt]
    foreach \x in {1e-11,1e-7,1e-3} {
        (axis cs:\x,1e-5) -- ++(0,2.5pt)
        (axis cs:\x,0.3) -- ++(0,-2.5pt)
    };
\end{axis}
\end{tikzpicture}
&
\begin{tikzpicture}[scale=0.9]
\begin{axis}[
    xmode=log, 
    ymode=log, 
    log basis x=10, 
    log basis y=10,
    xlabel={$h$}, 
    ylabel={$\rho_c$}, 
    title={IMEX-I Euler}, 
	legend style={
                at={(1.05,1)}, 
                anchor=north west,
            },
            clip=false, 
	grid=major, 
    ymin=1e-5,
    ymax=1e0,
    xmax=1e-4,
    minor grid style={gray!30, dotted}, 
    major grid style={gray!60}, 
    width=6cm, 
    height=5.3cm, 
]
\addplot coordinates {
  	(5.000000000000000e-05, 0.000017551392121)
    (2.500000000000000e-05, 0.000065661774034)
    (1.250000000000000e-05, 0.000262589096270)
    (6.250000000000000e-06, 0.001049429496038)
    (3.125000000000000e-06, 0.004182958342929)
    (1.562500000000000e-06, 0.021865121909030)
    (7.812500000000000e-07, 0.093959649189172)
    (3.906250000000000e-07, 0.303736950461763)
};
\addlegendentry{ Actual spectral radius }
\addplot coordinates {
 	(5.000000000000000e-05, 2.720018496125774e-05)
    (2.500000000000000e-05, 1.088029594404968e-04)
    (1.250000000000000e-05, 4.352473549122145e-04)
    (6.250000000000000e-06, 1.741557926009399e-03)
    (3.125000000000000e-06, 6.975342676531306e-03)
    (1.562500000000000e-06, 2.804810456168392e-02)
    (7.812500000000000e-07, 1.146032205822328e-01)
    (3.906250000000000e-07, 5.015197130551274e-01)
};
\addlegendentry{ Upper bound in \eqref{eq:upbN} }
\end{axis}
\end{tikzpicture}\\
\begin{tikzpicture}[scale=0.9]
\begin{axis}[
    xmode=log, 
    ymode=log, 
    log basis x=10, 
    log basis y=10,
    xlabel={$h$}, 
    ylabel={$\rho_\phi$}, 
    title={IMEX-E Euler}, 
	grid=major, 
    ymin=1e-11,
    ymax=1,
    xmin=1e-7,
    xmax=1e-4,
    minor grid style={gray!30, dotted}, 
    major grid style={gray!60}, 
    width=6cm, 
    height=5.3cm, 
]
\addplot coordinates {
  	(5.000000000000000e-05, 1.767682390573589e-10)
    (2.500000000000000e-05, 1.885282290869258e-09)
    (1.250000000000000e-05, 3.014880784665053e-08)
    (6.250000000000000e-06, 4.813778893403701e-07)
    (3.125000000000000e-06, 7.638442024657471e-06)
    (1.562500000000000e-06, 1.478789789424484e-04)
    (7.812500000000000e-07, 1.258095514023115e-03)
    (3.906250000000000e-07, 1.388983236940416e-02)
};
\addplot coordinates {
  	(5.000000000000000e-05, 7.216611484025484e-10)
    (2.500000000000000e-05, 1.154545046629422e-08)
    (1.250000000000000e-05, 1.846550722427323e-07)
    (6.250000000000000e-06, 2.949877498338948e-06)
    (3.125000000000000e-06, 4.690669766076992e-05)
    (1.562500000000000e-06, 7.326693317071058e-04)
    (7.812500000000000e-07, 1.075787022868093e-02)
    (3.906250000000000e-07, 1.376823853399944e-01)
};

\draw[black!50,line width=0.3pt]
    foreach \y in {1e-10,1e-9,1e-7,1e-6,1e-4,1e-3,1e-1} {
        (axis cs:1e-7,\y) -- ++(2.5pt,0)
        (axis cs:1e-4,\y) -- ++(-2.5pt,0)
    };
\end{axis}
\end{tikzpicture}
&
\begin{tikzpicture}[scale=0.9]
\begin{axis}[
    xmode=log, 
    ymode=log, 
    log basis x=10, 
    log basis y=10,
    xlabel={$h$}, 
    ylabel={$\rho_c$}, 
    title={IMEX-E Euler}, 
	legend style={
                at={(1.05,1)}, 
                anchor=north west,
            },
            clip=false, 
	grid=major, 
    ymin=1e-11,
    ymax=1,
    xmin=1e-7,
    xmax=1e-4,
    minor grid style={gray!30, dotted}, 
    major grid style={gray!60}, 
    width=6cm, 
    height=5.3cm, 
]
\addplot coordinates {
  	(5.000000000000000e-05, 6.935811345610684e-11)
    (2.500000000000000e-05, 7.397595131331728e-10)
    (1.250000000000000e-05, 1.183229034213584e-08)
    (6.250000000000000e-06, 1.890698445477370e-07)
    (3.125000000000000e-06, 3.009413532283642e-06)
    (1.562500000000000e-06, 5.896311776737326e-05)
    (7.812500000000000e-07, 5.233566725608669e-04)
    (3.906250000000000e-07, 6.450993686696793e-03)
};
\addlegendentry{ Actual spectral radius }
\addplot coordinates {
  	(5.000000000000000e-05, 2.831552379512535e-10)
    (2.500000000000000e-05, 4.530206576801610e-09)
    (1.250000000000000e-05, 7.246557032110279e-08)
    (6.250000000000000e-06, 1.158316094475330e-06)
    (3.125000000000000e-06, 1.846105300407659e-05)
    (1.562500000000000e-06, 2.908955692412377e-04)
    (7.812500000000000e-07, 4.396899677484391e-03)
    (3.906250000000000e-07, 5.917127188134166e-02)
};
\addlegendentry{ Upper bound in \eqref{eq:upbNE} }

\draw[black!50,line width=0.3pt]
    foreach \y in {1e-11,1e-10,1e-9,1e-8,1e-7,1e-6,1e-2,1e-4,1e-3,1e-1} {
        (axis cs:1e-7,\y) -- ++(2.5pt,0)
        (axis cs:1e-4,\y) -- ++(-2.5pt,0)
    };
\end{axis}
\end{tikzpicture}
\end{tabular}
\caption{Actual spectral radii $\rho_\phi$ (left column) and $\rho_c$ (right column) and bounds given in \eqref{eq:upbN} for IMEX-I (top row) and in \eqref{eq:upbNE} for IMEX-E (bottom row). The comparison is performed for $\Delta t = 10^{-5}$s and different values of $\Delta x = \Delta y = h$.}
    \label{fig:specradh}
\end{figure}

\subsection{Electropolishing of a metallic surface}
Electropolishing is an electrochemical process used to smooth, clean, and polish metal surfaces by removing a thin layer of material. The metal part is  immersed in an acid-based solution and made the anode in an electrolytic cell. When direct current is applied, surface irregularities, such as microscopic peaks, dissolve faster, resulting in a smoother and more uniform finish. This process improves corrosion resistance and cleanliness, removing surface roughness without mechanical stress. Inspired by analogue experiments in \citep{GaoFD,GaoFE,mai}, we implement the phase-field model to simulate the electropolishing process.

We consider a domain with an irregular rough edge (see Figure \ref{fig:ep}). Zero Dirichlet boundary conditions are assigned on the irregular edge, while homogeneous Neumann boundary conditions are assigned elsewhere. The extended rectangular domain is $\Omega=[200\mu$m$,100\mu$m$]$.

As in the previous examples, we compare on a spatial grid with resolution $\Delta x=\Delta y=1\mu\text{m}$, whereas we select a different time step for each method in order to have similarly accurate solutions, and they are $\Delta t=6\cdot 10^{-4}$s for iterative IMEX-E Euler, $\Delta t=5\cdot 10^{-3}$s for iterative IMEX-E 2SBDF, $\Delta t= 10^{-3}$s for the exponential Rosenbrock-Euler method, and $\Delta t= 15\cdot 10^{-4}$s for ARS(2,2,2). For computing $\varphi_1$ in the exponential method we set tolerance $10^{-4}$, whereas for the IMEX methods we choose the regularization parameter $w$ as before. For the iterative schemes we use the stopping criteria in Section \ref{sec:stop} with $\varepsilon_1=10^{-4}$, $\varepsilon_2=10^{-3}$ and $\varepsilon_3=10^{-7}$. Again, $\varepsilon_2$ and $\varepsilon_3$ are tuned to have errors no larger than $10^{-3}$ in $\Theta$.

We compute a reference solution for this problem by using the method in \citep{Conte,GaoFD} on a fine grid with steps
$\Delta x=\Delta y=0.5\mu\text{m}$ and $\Delta t=2\cdot 10^{-4}\text{s}$ setting tolerance $10^{-6}$ for the computation of the matrix function $\varphi_1$. Efficiency comparisons between the schemes are reported in Table \ref{tab:EP}. The results show once again that the proposed approach outperforms the ones in the literature.

The solution $c$ of IMEX-E 2SBDF scheme is portrayed in Figure~\ref{fig:ep}. We show the initial configuration, snapshots at intermediate times and at the final time $T=20$s. As expected, protruding regions corrode more rapidly due to shorter diffusion distances, leading to higher limiting current density \citep{mai}.

\begin{figure}[t]
\centerline{
\includegraphics[scale=0.6]{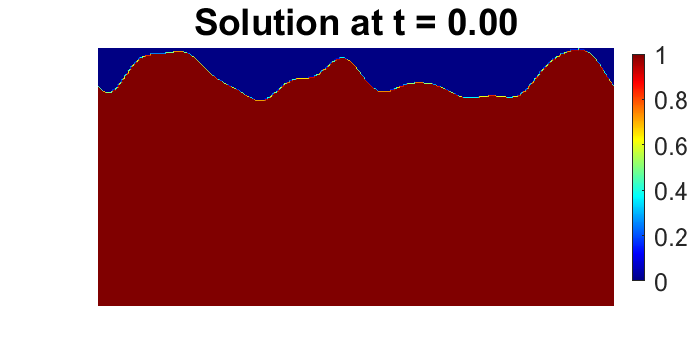}
\includegraphics[scale=0.6]{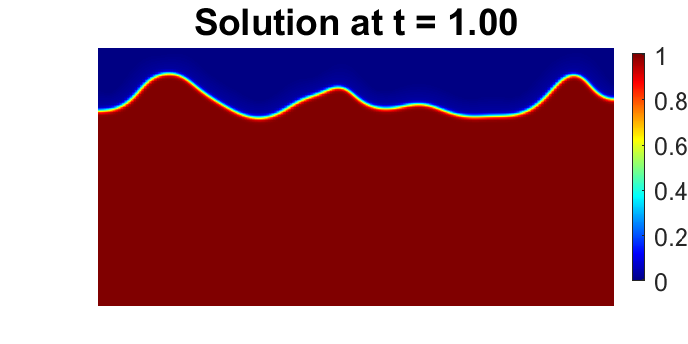}
}
\centerline{
\includegraphics[scale=0.6]{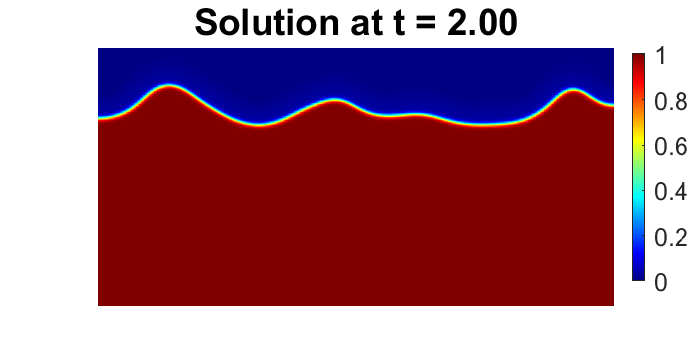}
\includegraphics[scale=0.6]{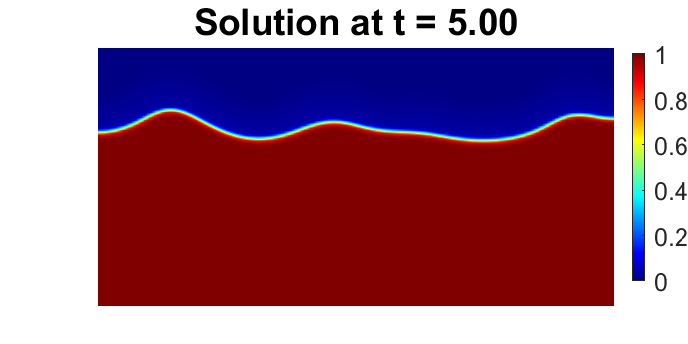}
}
\centerline{
\includegraphics[scale=0.6]{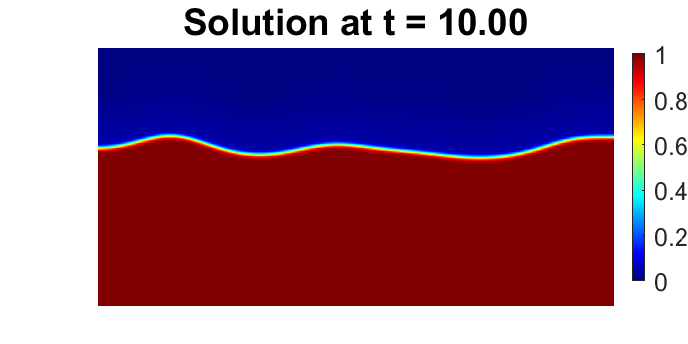}
\includegraphics[scale=0.6]{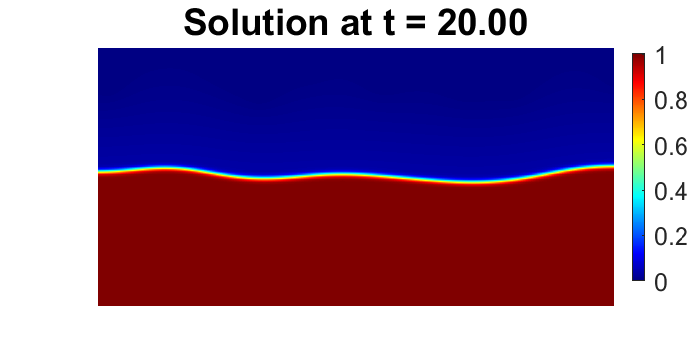}
}
\caption{\color{black}Electropolishing. Configuration of $c$ given by iterative IMEX-E 2SBDF with $\Delta x=\Delta y=1\mu\text{m},$ and $\Delta t=5\cdot 10^{-3}\text{s}$ at different times.}\label{fig:ep}
\end{figure}

\begin{table}[t]
\begin{center}
\begin{tabular}{||c||c|c|c||c|c|c||c|c|c||c|c|c||}
\hline
\hline
&\multicolumn{3}{|c||}{Iter. IMEX-E Euler \tabstretch} & \multicolumn{3}{c||}{Iter. IMEX-E 2SBDF} & \multicolumn{3}{c||}{Method in \citep{GaoFD}} & \multicolumn{3}{c||}{ARS(2,2,2)}\\
&\multicolumn{3}{|c||}{$\Delta t=6\cdot 10^{-4}$s \tabstretch} & \multicolumn{3}{c||}{$\Delta t=50\cdot 10^{-4}$s \tabstretch} & \multicolumn{3}{c||}{$\Delta t=10\cdot 10^{-4}$s} & \multicolumn{3}{c||}{$\Delta t=15\cdot 10^{-4}$s}\\
\hline
$t$ (s)& $\text{Err}_\phi$\tabstretch & $\text{Err}_c$ & El. time & $\text{Err}_\phi$\tabstretch & $\text{Err}_c$ & El. time & $\text{Err}_\phi$\tabstretch & $\text{Err}_c$ & El. time & $\text{Err}_\phi$\tabstretch & $\text{Err}_c$ & El. time \\
\hline
1&0.031&0.038&2.45&0.019&0.022&2.61&0.022&0.026& 95.70 & 0.028 & 0.035 & 26.69\\
2&0.031&0.036&4.53&0.017&0.020&3.93&0.022&0.026& 189.79 & 0.028 & 0.034 & 50.87\\
5&0.030&0.035&10.76&0.015&0.019&7.05&0.022&0.026&475.66 & 0.028 & 0.033 & 123.70\\
10&0.029&0.035&21.16&0.015&0.018&11.26&0.022&0.027&915.63 & 0.028 & 0.033 & 245.27\\
20&0.030&0.036&42.79&0.016&0.019&15.91&0.023&0.028&1889.73& 0.029 & 0.035 & 498.29\\
\hline
\hline
\end{tabular}
\end{center}
\caption{Electropolishing. 2-norm relative error in $\phi$ and $c$ and elapsed time in seconds for different methods at different times.}\label{tab:EP}
\end{table}

\subsection{Three dimensional experiments}
In this section we consider 3D domains, with the only goal to test the performance of the proposed numerical methods in a three-dimensional setting, rather than to validate the model against experimental data. Although the phase field model considered in this paper has been validated in \citep{mai} for 2D domains only, it naturally extends to three dimensions. Therefore, we reuse the model parameters from Table \ref{tab:parameters} without recalibration. In the same spirit, analogues experiments on 3D domains have been proposed in \citep{GaoFE}.

\subsubsection{Three dimensional pencil electrode test}
A three-dimensional version of the pencil electrode test is examined in \citep{GaoFE}, where time integration is performed using the exponential Rosenbrock-Euler method, coupled with an adaptive finite element spatial discretization. 

We consider the same setting consisting of a metal wire of dimensions $25\mu$m$ \times 25\mu$m$ \times 150\mu$m, coated with a corrosion-resistant epoxy on its lateral surface, leaving only the top end exposed to the electrolyte solution. Thus, we define $\Omega=[0,25]\times [0,25] \times [0,150]$, assigning homogeneous Dirichlet boundary conditions at the top,
$$\phi(x,y,150,t)=c(x,y,150,t)=0,$$
and homogeneous Neumann conditions on the remaining boundary. At the initial time the entire domain is in the pure metal state, so
$$\phi(x,y,z,0)=c(x,y,z,0)=1,$$
in the interior of $\Omega$. 
We solve this problem using the IMEX 2SBDF method \eqref{finAC}-\eqref{finCH} with
$$\Delta x = \Delta y = \Delta z = 1\mu\text{m},\qquad \Delta t=0.02\text{s},\qquad w=4.43\cdot 10^8,$$
and final time $T=225$s. These discretization parameters are chosen based on those adopted for the analogous 2D case presented in Section \ref{sec:pencil2D}. 

\begin{figure}[t]
\centerline{
\includegraphics[scale=0.45]{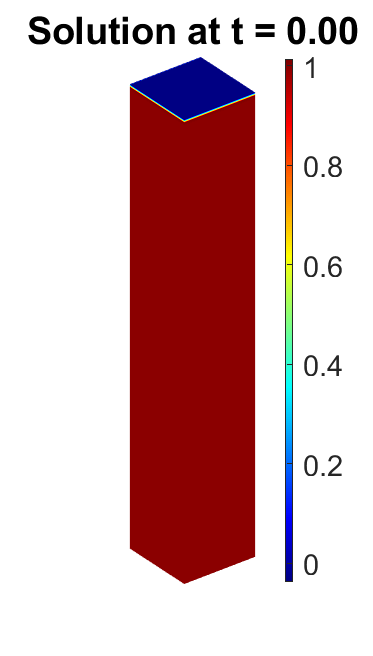}
\includegraphics[scale=0.45]{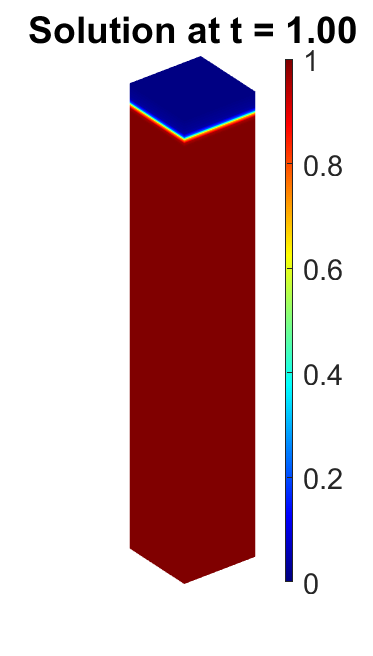}
\includegraphics[scale=0.45]{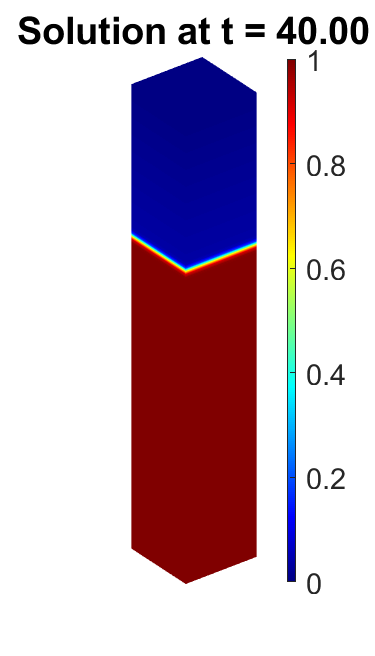}\!\!
\includegraphics[scale=0.45]{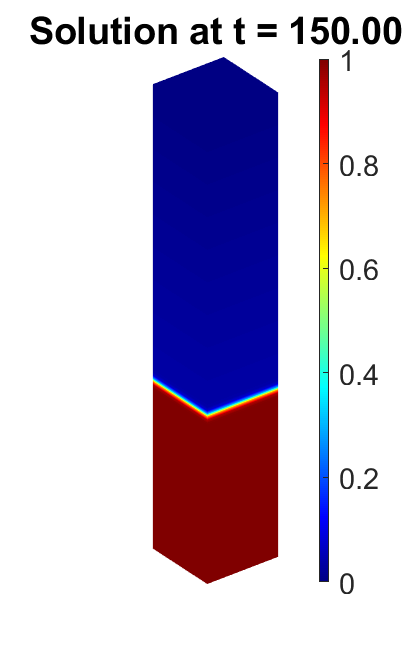}\!\!\!
\includegraphics[scale=0.45]{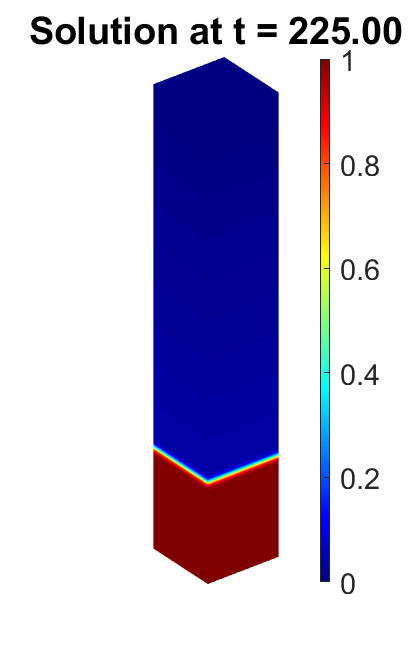}

}
\caption{\color{black}3D pencil electrode test. Configuration of $c$ given by IMEX 2SBDF with $\Delta x=\Delta y=\Delta z=1\mu\text{m},$ and $\Delta t=0.02\text{s}$ at different times.}\label{fig:3Dpencil}
\end{figure}

Figure \ref{fig:3Dpencil} shows the approximations of the state variable $c$ at different times. The correct evolution of the corrosion interface is confirmed by a graph identical to that in Figure \ref{fig:pitpencil}, which shows that its advancement is proportional to the square root of time, in agreement with theoretical expectations. Moreover, these results are visually indistinguishable from the corresponding figures in \citep{GaoFE}. 

We solve now the problem by the IMEX Euler method \eqref{finACord2}--\eqref{finCHord2} on the same spatial grid and with time step $\Delta t = 0.001$s, as done in Section \ref{sec:pencil2D}.

No visible difference can be observed in the evolution of the solution. However, the computational times required to complete the full simulation differ significantly: our computations required 871.38 seconds using IMEX Euler and 155.67 seconds using IMEX 2SBDF. In contrast, the simulation reported in \citep{GaoFE} took 8 hours and 20 minutes, despite the use of an adaptive spatial grid that significantly reduces the number of degrees of freedom from 102,076 in our uniform grid to just 6,000.

\subsubsection{Three dimensional semi-cylinder pit growth}
Here we present a final experiment simulating the evolution of a semi-cylindrical pit in three dimensions. The sample size is $200 \mu\text{m}\times 25 \mu\text{m}\times 100 \mu\text{m}$, featuring a semi-cylindrical cavity at the center line of the top surface, with a $4\mu$m diameter and a length $25\mu$m. On this cavity, homogeneous Dirichlet boundary conditions are imposed for both $\phi$ and $c$, while the rest of the outer boundary is subject to homogeneous Neumann boundary conditions. This configuration is motived by a related study in \citep{GaoFE}, where the problem was addressed using a finite element method with adaptive meshing. 

The problem is solved by the iterative IMEX-E 2SBDF with steps 
$$\Delta x=\Delta y=\Delta z=1\mu{\text{m}}, \qquad \Delta t = 6\cdot 10^{-3}\text{s},$$
and the usual value of $w$, till the final time $T=225$s. The initial configuration and snapshots of the solution $c$ at intermediate and final times are shown in Figure \ref{fig:3Dcirc}. The stopping criteria are those in Section \ref{sec:stop} with $\varepsilon_1= 10^{-4}, \varepsilon_2 = 10^{-3}, \varepsilon_3 = 10^{-8}$. With this choice of the tolerances, the errors in $\Theta$ are
$$\max\nolimits_{\Theta}|\phi| = 6.47\cdot 10^{-14},\qquad \max\nolimits_{\Theta} |c| = 10^{-3}.$$

\begin{figure}[t]
\centerline{
\includegraphics[scale=0.75]{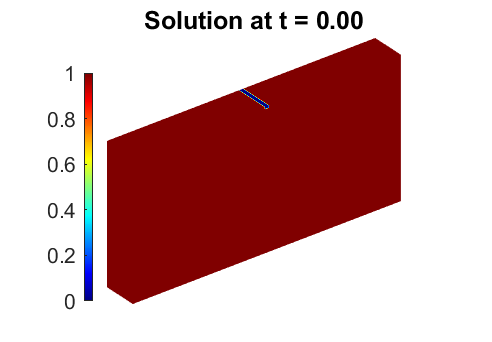}\!\!\!\!\!\!\!\!\!\!\!\!\!\!\!
\includegraphics[scale=0.75]{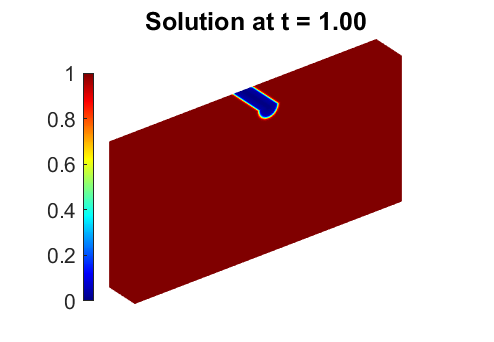}\!\!\!\!\!\!\!\!\!\!\!\!\!\!\!
\includegraphics[scale=0.75]{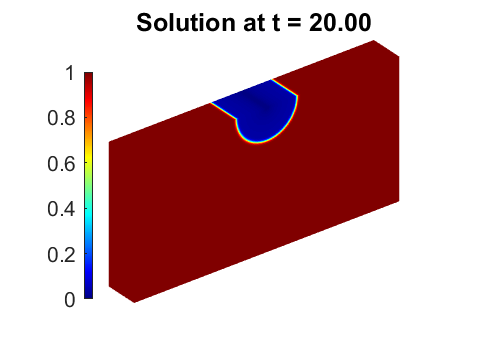}}
\centerline{
\includegraphics[scale=0.75]{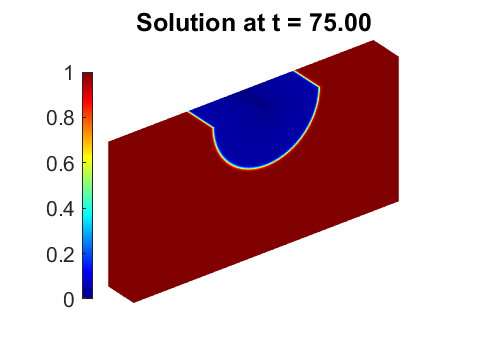}\!\!\!\!\!\!\!\!\!\!\!\!\!\!
\includegraphics[scale=0.75]{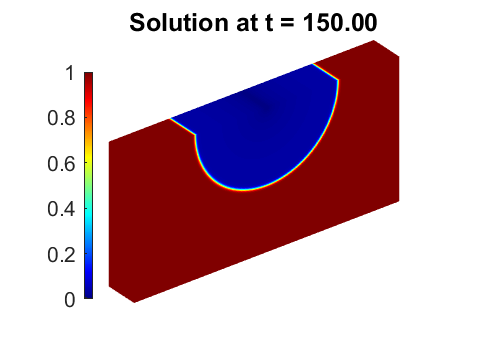}\!\!\!\!\!\!\!\!\!\!\!\!
\includegraphics[scale=0.75]{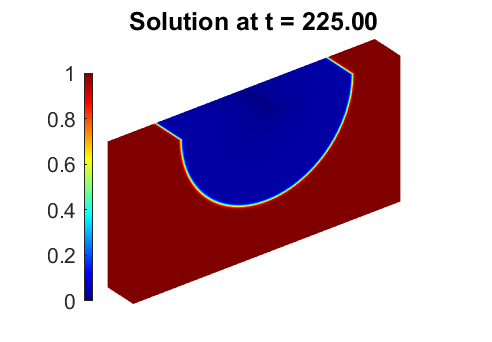}

}
\caption{\color{black}3D semi-cylinder pit growth. Configuration of $c$ given by IMEX 2SBDF with $\Delta x=\Delta y=\Delta y=1\mu\text{m},$ and $\Delta t=6\cdot 10^{-3}\text{s}$ at different times.}\label{fig:3Dcirc}
\end{figure}

The computational time to complete the simulation is 7083.23 seconds (just under two hours). This result seems consistent with the previous observations, and further confirms the computational advantages of the proposed methods. For comparison, the computation times reported in \citep{GaoFE} for three-dimensional domains with complicated geometries are of the order of the days, despite the use of adaptive meshes that significantly reduces the number of degrees of freedom. For this specific experiment, the reported computational time is nearly 86 hours, despite the number of degrees of freedom in \citep{GaoFE} (68,400) is over 7 times smaller than the number of nodes in our spatial grid (500,000).

\section{Conclusions}\label{sec:concl}
Due to the Kronecker structure of the diffusion matrix, classical finite difference space discretization of a phase-field model for metal corrosion on rectangular domains can be reformulated as a system of matrix ODEs. Time-stepping IMEX methods, when applied in matrix form, can be efficiently solved, as they require the solution of small matrix problems at each step.

When the domain is not rectangular, as is frequent in real corrosion simulations, the diffusion matrix does not have a Kronecker structure. In this case the efficient matrix-oriented approach can not be directly employed. However, this structure can be recovered by consistently reformulating the problem on an extended rectangular domain and defining appropriate iterative IMEX methods. 

While in the direct implementation of the matrix-oriented technique on rectangular domains, only one matrix problem needs to be solved per time step, the proposed iterative methods require the solution of a sequence of problems to converge. Thus, the convergence of these methods and the associated error propagation has been analyzed. 

Numerical experiments have shown that the proposed approach achieves the same level of accuracy as existing efficient methods in the literature, but with significantly lower computational time, reducing it from the order of hours to just seconds or minutes on a standard workstation. It should be noted, however, that the ARS(2,2,2) scheme is implemented here by solving sparse narrow-banded triangular systems, to exploit the structural properties of the method for a generic problem. Nevertheless, for the problem here considered, a matrix-oriented implementation of this scheme can be devised by following the approach in this paper. Although such an extension may be conceptually straightforward, the convergence analysis of the resulting iterative procedures for applications to non rectangular domains requires careful investigations. These aspects, along with extensions to other IMEX Runge-Kutta schemes and theoretical as well as computational comparisons with the multistep schemes presented in this paper, will be the subject of future research.

The results in this paper provide a substantial contribution to the efficient computational solution of corrosion models, addressing a major challenge often noted in the literature. The computational savings are particularly significant in the case of three-dimensional domain, where qualitatively accurate results can be obtained within just a few minutes or hours, whereas existing methods in the literature, even those employing adaptive meshes, require hours or days, respectively.
Future research should focus on the application of this procedure to other highly stiff phase-field models. For instance, the phase-field framework underlying the model in this paper was also used in the deformation–diffusion–material dissolution theory described in \citep{Ma}. That work proposes a novel phase-field formulation to simulate pitting corrosion and stress corrosion cracking, capturing the processes of film rupture, dissolution, and repassivation. To couple the electrochemical behavior with mechanical effects, the resulting model consists of an Allen–Cahn equation for the phase-field variable $\phi$, a Cahn–Hilliard equation for the state variable $c$, and an additional PDE for $\phi$ to enforce mechanical force balance. Extending the methods presented in this paper to efficiently solve this complex electro-chemo-mechanical problem could be an interesting topic of future research.

\subsection*{Acknowledgements}
The authors are members of the research group Gruppo Nazionale per il Calcolo
Scientifico, Istituto Nazionale di Alta Matematica (GNCS-INdAM). This work has
been supported by GNCS-INdAM projects and by the Italian Ministry of University
and Research (MUR), through the PRIN PNRR 2022 project P20228C2PP (CUP:
F53D23010020001) BAT-MEN (BATtery Modeling, Experiments \& Numerics).

The authors thank the anonymous reviewers for their careful reading and constructive comments, which helped improve the quality and completeness of this manuscript.

\end{document}